\documentclass[a4paper,10pt,oneside,reqno]{amsart}

\usepackage[utf8]{inputenc}
\usepackage[english]{babel}
\usepackage{
   indentfirst
  ,amsfonts %more symbols
  ,amssymb
  ,amsthm
  ,mathrsfs
  ,leftidx  %left sub- and superscript in math mode
  ,cite     %improved handling of numeric citations
  ,enumitem
  ,stmaryrd %more symbols
  ,todonotes
  ,xspace
  ,bbm
  ,epigraph
  ,mathtools
  ,blindtext
  ,wasysym  %more symbols
  ,graphicx
  ,xcolor
  ,relsize
  ,harpoon
}
\usepackage{tikz-cd}
\usepackage[all,cmtip]{xy}
\usepackage[cal=pxtx, scr=boondoxo]{mathalfa}
\usepackage{hyperref}
\usepackage[top=3cm, bottom=3cm, left=4cm, right=4cm,
  marginparwidth=80pt]{geometry}
\usepackage[bbgreekl]{mathbbol}
\DeclareSymbolFontAlphabet{\mathbbm}{bbold}
\DeclareSymbolFontAlphabet{\mathbb}{AMSb}

\usepackage{xparse}
\usepackage{ifthen}

\ExplSyntaxOn
\NewDocumentCommand{\makeabbrev}{mmm}
 {
  \yoruk_makeabbrev:nnn { #1 } { #2 } { #3 }
 }

\cs_new_protected:Npn \yoruk_makeabbrev:nnn #1 #2 #3
 {
  \clist_map_inline:nn { #3 }
   {
    \cs_new_protected:cpn { #2 } { #1 { ##1 } }
   }
 }
\ExplSyntaxOff

\makeabbrev{\textsf}{sf#1}{a,b,c,d,e,f,g,h,i,j,k,l,m,n,o,p,q,r,s,t,u,v,w,x,y,z,A,B,C,D,E,F,G,H,I,J,K,L,M,
N,O,P,Q,R,S,T,U,V,W,X,Y,Z}
\makeabbrev{\mathfrak}{fk#1}{a,b,c,d,e,f,g,h,i,j,k,l,m,n,o,p,q,r,s,t,u,v,w,x,y,
z,A,B,C,D,E,F,G,H,I,J,K,L,M,N,O,P,Q,R,S,T,U,V,W,X,Y,Z}
\makeabbrev{\mathcal}{cal#1}{A,B,C,D,E,F,G,H,I,J,K,L,M,N,O,P,Q,R,S,T,U,V,W,X,Y,
Z}
\makeabbrev{\mathbb}{bb#1}{A,B,C,D,E,F,G,H,I,J,K,L,M,N,O,P,Q,R,S,T,U,V,W,X,Y,Z}
\makeabbrev{\mathbf}{bf#1}{a,b,c,d,e,f,g,h,i,j,k,l,m,n,o,p,q,r,s,t,u,v,w,x,y,
z,A,B,C,D,E,F,G,H,I,J,K,L,M,N,O,P,Q,R,S,T,U,V,W,X,Y,Z}
\makeabbrev{\mathrm}{rr#1}{a,b,c,d,e,f,g,h,i,j,k,l,m,n,o,p,q,r,s,t,u,v,w,x,y,
z,A,B,C,D,E,F,G,H,I,J,K,L,M,N,O,P,Q,R,S,T,U,V,W,X,Y,Z}
\makeabbrev{\mathscr}{scr#1}{A,B,C,D,E,F,G,H,I,J,K,L,M,N,O,P,Q,R,S,T,U,V,W,X,Y,
Z}

\newcommand{\rr}[1]{\mathrm{#1}}

\usepackage{scalerel,stackengine}
\stackMath
\newcommand\rlywdhat[1]{
  \savestack{\tmpbox}{\stretchto{
    \scaleto{\scalerel*[\widthof{\ensuremath{#1}} * \real{0.8}]
      {\kern.1pt\mathchar"0362\kern.1pt}{\rule{0ex}{\textheight}}}
    {\textheight}
    }{1.8ex}
  }
\stackon[-7pt]{#1}{\tmpbox}
}

\newcommand{\xto}[1]{\xrightarrow{#1}}

\renewcommand{\phi}{\varphi}
\renewcommand{\geq}{\geqslant}
\renewcommand{\leq}{\leqslant}
\newcommand{\eps}{\varepsilon}
\newcommand{\istar}{\scalebox{.6}{$\bigstar$}}

\newcommand{\Z}{\mathbb{Z}}
\newcommand{\bE}{\mathbb{E}}
\newcommand{\cD}{\mathrm{D}}
\newcommand{\cC}{\mathcal{C}}
\DeclareMathOperator{\Spec}{Spec}
\DeclareMathOperator{\Fun}{Fun}
\DeclareMathOperator{\Ext}{Ext}
\DeclareMathOperator{\Tor}{Tor}
\DeclareMathOperator{\Aut}{Aut}

\DeclareMathOperator{\map}{map}
\DeclareMathOperator{\innmap}{\underline{Map}}
\DeclareMathOperator{\Map}{Map}

\DeclareMathOperator{\Hom}{Hom}

\DeclareMathOperator{\Gr}{Gr}

\DeclareMathOperator{\coker}{coker}
\DeclareMathOperator{\im}{im}
\DeclareMathOperator{\fib}{fib}
\DeclareMathOperator{\cofib}{cof}
\DeclareMathOperator{\totcof}{tot-cof}
\DeclareMathOperator{\totfib}{tot-fib}
\DeclareMathOperator{\nerve}{N}

\DeclareMathOperator{\gr}{gr}
\DeclareMathOperator{\Un}{Un}
\DeclareMathOperator{\St}{St}
\DeclareMathOperator{\Lan}{Lan}
\DeclareMathOperator{\Ran}{Ran}

\DeclareMathOperator{\dec}{D\acute{e}c}

\DeclareMathOperator{\Free}{Free}
\DeclareMathOperator{\CoAlg}{CoAlg}
\DeclareMathOperator{\BBar}{Bar} %\Bar already defined
\DeclareMathOperator{\Tot}{Tot}
\DeclareMathOperator{\Cobar}{Cobar}
\DeclareMathOperator{\Cotor}{Cotor}

\def\tensor{\qopname\relax m{\otimes}}
\def\tplus{\qopname\relax m{\oplus}}
\def\cpr{\qopname\relax m{\amalg}}
\def\colim{\qopname\relax m{\operatorname{colim}}}

\newcommand{\simplex}{\reflectbox{
  \scalebox{1.5}[1.1]{\text{$\mathbbm{\Delta}$}}}}

\newcommand{\Yo}{\scalebox{1.25}[1.25]{\text{\usefont{U}{min}{m}{n}\symbol{'207}}}}
\DeclareFontFamily{U}{min}{}
\DeclareFontShape{U}{min}{m}{n}{<-> udmj30}{}

\newcommand{\stYo}{\Yo^{\raisebox{-.4em}{\scalebox{.7}{\kern-.3em st}}}}
\newcommand{\ptYo}{\Yo^{\raisebox{-.4em}{\scalebox{.7}{\kern-.3em pt}}}}

\newcommand{\Pre}[1]{\mathcal{P}(#1)}

\newcommand{\Prept}[1]{\mathcal{P}_{*}(#1)}
\newcommand{\Prefinpt}[1]{\mathcal{P}^{\mathrm{fin}}_{*}(#1)}
\newcommand{\Pst}[1]{\mathcal{P}_{\mathrm{st}}(#1)}

\newcommand{\wtrunc}[1]{\tau_{\scalebox{.7}{$\geq\kern-.3em #1$}}}
\newcommand{\ptrunc}[1]{\tau_{\scalebox{.7}{$\leq\kern-.3em #1$}}}

\newcommand{\Alg}{\mathrm{Alg}}
\newcommand{\op}{^{\mathrm{op}}}

\newcommand{\Sp}{\mathrm{Sp}}
\newcommand{\Ab}{\mathrm{Ab}}
\newcommand{\Mod}{\mathrm{Mod}}

\newcommand{\id}{\mathrm{id}}

\newcommand{\ominfty}{\Omega ^{\infty}}
\newcommand{\susinftyp}{\Sigma ^{\infty}_+}

\newcommand{\pt}{\mathrm{pt}}

\newcommand{\cat}{$\infty$-cat\-e\-go\-ry\xspace}
\newcommand{\cats}{$\infty$-cat\-e\-gories\xspace}
\newcommand{\wt}[1]{\widetilde{#1}}
\newcommand{\spaces}{\scrS}

\newcommand{\catCh}{\mathbf{Ch}}
\newcommand{\Ch}{\operatorname{Ch}}
\newcommand{\Cha}{\Ch_{*}}
\newcommand{\coCh}{\Ch^{*}}
\newcommand{\zero}{\mathbf{0}}
\newcommand{\de}{\partial}
\newcommand{\Fild}{\mathrm{Fil}^\text{\scalebox{0.85}{$\downarrow$}}}
\newcommand{\Fili}
  {\mathrm{Fil}_\text{\raisebox{1.3\depth}{\scalebox{0.85}{$\uparrow$}}}}
\newcommand{\cFild}{\rlywdhat{\Fild}}

\newcommand{\spch}[2]{\bbS^{#2}_{[#1]}}
\newcommand{\aush}{\widehat{\calA}}
\newcommand{\preaush}{\widetilde{\calA}}
\newcommand{\caush}{\calA}
\newcommand{\imp}{\calI}
\newcommand{\tothom}{\calH}
\newcommand{\kinv}{\calK}
\newcommand{\realiz}[1]{\lvert - \rvert_{#1}}

\newcommand{\pjct}{\pi}
\newcommand{\tw}{\scalebox{.9}{\textsc{tw}}}
\newcommand{\ev}{\mathrm{ev}}
\newcommand{\const}{\textsc{cns}}

\newcommand{\ptize}{\textsc{pt}}
\newcommand{\spineZ}{\mathrm{I}_{[-\infty,+\infty]}}
\newcommand{\DelPar}[1]{\Delta^{\{#1\}}}
\newcommand{\cprPar}[1]{\cpr_{\DelPar{#1}}}
\newcommand{\evalfun}{y}
\newcommand{\undch}{u}

\newcommand{\ifil}{\iota_{\mathrm{Fil}}}
\newcommand{\ich}{\iota_{\mathrm{Ch}}}
\newcommand{\oC}{C_{\scalebox{.6}{$\heartsuit$}}}
\newcommand{\ptMap}{\Map{}_*}
\newcommand{\cTP}{\bbT\bbP}
\newcommand{\cHP}{\bbH\bbP}

\newcommand{\cHCm}{\bbH\bbC^{-}}

\newtheoremstyle{reference}%
   {}
   {}
   {}                      % Font del testo
   {}                      % Rientro margini
   {\bfseries}             % Font del titolo dell'ambiente
   {}                     % Punteggiatura dopo "Teorema"\"Definizione"
   { }                  % Spazio tra titolo e testo.
   {\thmname{#1}           % #1 : Definizione\Teorema\ecc
    \thmnumber{#2}         % #2 : Contatore
    \thmnote{{\bf #3}}}  % #3 : Testo tra "[" e "]"

\theoremstyle{reference}
  \newtheorem{thm}{Theorem}[section]
  \newtheorem{prop}[thm]{Proposition}
  \newtheorem{lemma}[thm]{Lemma}
  \newtheorem{cor}[thm]{Corollary}
  \newtheorem{defi}[thm]{Definition}
  \newtheorem{rem}[thm]{Remark}
  \newtheorem{ex}[thm]{Example}
  
  \newtheorem{notat}[thm]{Notation}
  
  \newtheorem{constr}[thm]{Construction}
  
  \newtheorem{variant}[thm]{Variant}
  \newtheorem*{thm*}{Theorem}
  \newtheorem*{cor*}{Corollary}
  \newtheorem*{counterex*}{Counterexample}
  \newtheorem*{defi*}{Definition}
  \newtheorem*{ex*}{Example}
  \newtheorem*{exercise*}{Exercise}
  \newtheorem*{lemma*}{Lemma}
  \newtheorem*{notat*}{Notation}
  \newtheorem*{prop*}{Proposition}
  \newtheorem*{question*}{Question}
  \newtheorem*{rem*}{Remark}
  \newtheorem*{them*}{Theorem}
  \newtheorem*{constr*}{Construction}
  \newtheorem*{guess*}{Guess}

\hypersetup{
    colorlinks=true,         % false: boxed links; true: colored links
    linkcolor=teal,         % color of internal links (change box color with linkbordercolor)
    citecolor=teal,         % color of links to bibliography
    urlcolor=teal           % color of external links
}

\allowdisplaybreaks

\author{Stefano Ariotta}
\title[Coherent cochain complexes and Beilinson \MakeLowercase{t}-structures]
  {Coherent cochain complexes and Beilinson \MakeLowercase{t}-structures,
  with an appendix by Achim Krause}
\date{}

\begin{document}
\begin{abstract}
  We define and study coherent cochain complexes in arbitrary stable \cats,
  following Joyal.
  Our main result is that the \cat of coherent cochain complexes in
  a stable \cat $\scrC$ is equivalent to the \cat of complete filtered objects
  in $\scrC$.

  We then show how the Beilinson t-structure can be interpreted in light
  of such equivalence, and analyze its behavior in the presence of symmetric
  monoidal structures.
  We also examine the relationship between the notion of (higher) Toda brackets
  and coherent cochain complexes.
  Finally, we prove how every coherent cochain complex gives rise to a
  spectral sequence and illustrate some examples.
\end{abstract}
\maketitle
\setcounter{tocdepth}{1}
\tableofcontents
\section{Introduction}
Recall that a filtered chain complex (resp.\ a filtered spectrum) consists of a
$\bbZ$-indexed sequence
$$\cdots\to F^{n} \to F^{n-1} \to F^{n-2} \to\cdots$$
where each $F^i$ is a chain complex (resp.\ a spectrum) and the morphisms
between them are chain maps\footnote{in the case of chain
  complexes it is common to allow only monomorphisms, but as one can always
  replace a chain map by a monomorphic one up to quasi-isomorphism, this
  extra condition will not be relevant to us}
(resp.\ morphisms of spectra).

Two of the motivating reasons for the study of filtered derived categories
and filtered spectra are the construction of spectral sequences, and
the existence of the Beilinson t-structure (firstly
introduced in \cite{Bei87}; see Definition \ref{defi-fild-beilinson-t-str}).
The objective of this paper is to present a new perspective on such objects
(and more generally, filtered objects in stable \cats) that allows to
generalize at once both constructions in a homotopy coherent fashion,
and to gain some insight on other related constructions, like the obstruction
theory for the realizability of spectra with prescribed homotopy groups and
k-invariants (see Section \ref{section-decomp}).
This perspective will be realized by using coherent cochain complexes
(originally introduced in their bounded flavor in \cite[35.1]{JoyalNotes},
in relation to the $\infty$-categorical Dold-Kan correspondence).

A coherent cochain complex in a stable \cat $\scrC$ is a homotopy coherent
version of an ordinary cochain complex, and consists of a $\bbZ$-indexed
sequence of objects $C^i\in\scrC$ and differentials
$$\cdots\xto{\partial} C^{n} \xto{\partial} C^{n+1}
\xto{\partial} C^{n+2} \xto{\partial}\cdots$$
together with nullhomotopies $\partial^2 \simeq 0$, and further
coherence data making all the nullhomotopies mutually compatible.
Concretely, a coherent cochain complex will be defined as a pointed functor
from (the nerve of) a $1$-category having as objects the integers together with
an extra base point, and being generated by morphisms $\partial\colon n\to n+1$
such that $\partial^2 = 0$ (see Definition \ref{def-ch} for a rigorous
formulation).
For the sake of simplicity, in this introduction we will restrict to the case
of spectra, and the corresponding \cat of coherent cochain complexes
$\coCh(\Sp)$, although the results in the rest of the paper are discussed in the
wider generality of stable \cats equipped with some t-structure.

At their core, both the Beilinson t-structure and the spectral sequence
associated to a filtered spectrum are in a sense ``blind'' to the
information stored at the limit of the relevant filtered object.
As it is customary, we say that a filtered object $F$ is \emph{complete} if
$\varprojlim F \simeq 0$, and we call the cofiber of the canonical map
$\varprojlim F \to \varinjlim F$ the \emph{completion} of $\varinjlim F$.
If $F$ is a filtered spectrum,
the spectral sequence it generates only abuts to the completion of its colimit.
Something similar is true for the Beilinson t-structure, although saying
precisely what this means requires more work;
the failure of it being left complete is a good starting point:
all and only the constant objects are the ones that are $\infty$-connected
with respect to the Beilinson t-structure, and constant objects are in a precise
sense orthogonal to complete ones (see Proposition
\ref{prop-cfild-bousfield}). In Theorem \ref{thm-beil-recoll}, we prove that
the Beilinson t-structure can be glued out of its restriction to the full
subcategory of complete objects and the trivial t-structure on the full
subcategory of constant objects.

As it turns out, the \cat of filtered spectra is equivalent to that
of coherent cochain complexes of spectra:

\begin{thm*}(see Theorems \ref{thm-cohch-fild-equiv},
  \ref{thm-general-equivalence} and \S \ref{section-monoidal})
  There exists a symmetric monoidal adjunction
  $$
  \begin{tikzcd}[column sep=huge]
    \Fun(\bbZ\op,\Sp) \ar[r, shift left=1.1ex, "\aush"] &
    \coCh(\Sp) \ar[l, shift left=1.1ex, "\imp"]
      \ar[l,phantom,"\text{\rotatebox{-90}{$\dashv$}}"]
  \end{tikzcd}
  $$
  that restricts to an equivalence on complete filtered spectra.
  Moreover, this adjunction sends the Day convolution symmetric monoidal
  structure on filtered spectra, to a symmetric monoidal structure given
  componentwise by
  $$(C \tensor D)^n \simeq \tplus_{s+t=n} C^s \tensor D^t .$$
\end{thm*}

As already noted in \cite[1.2]{HA}, given a filtered spectrum $F$ one can
construct a sequence of morphisms on suitable shifts of the graded pieces
\begin{equation}\label{eq-intro-gr}
  \gr^nF[-n]\to\gr^{n+1}F[-n-1]
\end{equation}
such that all pairwise compositions are nullhomotopic, and the homotopies are
suitably compatible (this is in fact the first step for the construction of
the spectral sequence of $F$).
The left adjoint defined in the theorem sends $F$ to the coherent cochain
complex
$$\cdots \to \gr^{-2}F[-2] \to \gr^{-1}F[-1] \to \gr^0 F \to \gr^1 F [1] \to
\gr^2 F [2] \to \cdots$$
where the differentials are precisely the maps of (\ref{eq-intro-gr}).
Although above we only represented the components and the differentials of
the object $\aush F$, it also encodes much more data: $\coCh(\Sp)$ is
defined using a pointed diagram $1$-category and
pointed functors, hence its objects keep track also of the nullhomotopies
for the two-fold compositions of their differentials, and of all the
recursive coherences between them.
The compatibilities between all the nullhomotopies
can be made explicit by means of higher Toda brackets (see Section
\ref{section-toda}); in order to do so, we use an explicit $\bbE_1$ presentation
of exterior algebras over $\bbZ$ due to Achim Krause, presented in
Appendix \ref{appendix-achim}.
The left adjoint above is really a homotopy coherent version
of the construction of homotopy objects in Beilinson's t-structure: such
t-structure corresponds to the pointwise one along the left adjoint functor
discussed above; that is, for a filtered spectrum $F$,
$$\pi_n^B F \cong \pi_n^\text{lvl}\aush F,$$
where $\pi_n^\text{lvl} C$ denotes the functor applying
$\pi_n$ to all the components of $C$.

The right adjoint can be thought of as the functor
iteratively solving all the extension problems needed to reconstruct the
filtered object from its graded pieces and the differentials. Notice that,
in order for such a reconstruction to be possible in general, one really needs
to use all the information stored in the higher morphisms.
In fact, the differentials appearing in a coherent cochain complex encode
precisely the differentials of the $E_1$ page of a spectral sequence
abutting to the \emph{total homology} of the complex (that is, the object
underlying its associated filtered object). The higher pages together with
their differentials can be recovered
from the complex itself, by means of an incarnation of Deligne's
\emph{d\'ecalage} construction, as explained in Section \ref{section-spseq}.

\subsection{Outline of the paper}
In Section \ref{section-prelim}, we put the basis for the rest of the paper, by
recalling
the main definitions and some structural results about the \cats at hand. In
particular, we show that for a stable \cat $\scrC$,
both $\Fun(\bbZ\op,\scrC)$ and $\coCh\scrC$ sit in suitable recollements of
stable \cats.

Section \ref{section-filandch} is the technical core of the paper, in which we
construct the adjunction between filtered spectra
and cochain complexes of spectra, proving the equivalence of the latter
with the \cat of complete filtered spectra. Then, in Section
\ref{section-general-equiv}, we
generalize the result to all stable \cats with sequential (inverse) limits.

In Section \ref{section-t-structures}, we analyze the interplay between the
Beilinson t-structure and the levelwise t-structure on the relevant categories,
when the base \cat $\scrC$ comes equipped with a t-structure. We continue the
analysis of \cats with extra structure in Section \ref{section-monoidal}, where
we study the interplay of the equivalence of Theorem
\ref{thm-general-equivalence} with symmetric monoidal structures, and the
compatibility of the relevant t-structures in the case where the base $\scrC$
is also equipped with a t-structure.

In Section \ref{section-toda}, we use an explicit $\bE_1$ presentation
of exterior algebras over $\bbZ$, proven by Achim Krause in
Appendix \ref{appendix-achim}, to show that $\coCh\scrC$ is the universal \cat
of $\bbZ\op$-indexed sequences of morphisms in $\scrC$ such that any pairwise
composition is trivial, as are all the possible Toda brackets.

In Section \ref{section-decomp}, we show how coherent cochain complexes encode
a form of obstruction theory, and we use the formalism developed there to
recover a spectral sequence from a coherent cochain complex in
Section \ref{section-spseq}.

Finally, in Section \ref{section-examples}, we have a look at a few examples
from the recent literature that we believe benefit from the perspective of
coherent cochain complexes.

Some of the main tools for constructing the adjunction and proving
the equivalence of Theorem \ref{thm-general-equivalence} is
the stable nerve-realization paradigm. We
recollect and prove many results about the $\infty$-categorical incarnation
of this topic in Appendix \ref{appendix-first}.

\subsection{Notational conventions}

Throughout, we will use the following notations and terminology:
\begin{enumerate}
  \item We will denote by
    $$\Pre{\scrC}\coloneqq\Fun(\scrC\op,\spaces)$$
    the \cat of presheaves on $\scrC$.
  \item  We will denote by
    $$\Pst{\scrC}\coloneqq\Fun(\scrC\op,\Sp)$$
    the \cat of spectra-valued presheaves on $\scrC$, and sometimes refer to
    it as the \cat of \emph{stable presheaves} on $\scrC$.
  \item We will denote by
    $$\Yo\colon\scrC\to\Pre{\scrC}$$
    the Yoneda embedding, and by $\Yo_{\kern-0.2emC}$ the functor represented by an
    object $C\in\scrC$.
  \item Given any functor \cat $\Fun(\scrC,\scrD)$, we will denote by
    $$\const\colon\scrD\to\Fun(\scrC,\scrD)$$
    the (fully faithful) functor obtained by precomposition with the terminal
    functor $\scrC\to\Delta^0$, and by $\const_C$ its value at $C\in\scrC$.
  \item When considering $\bbZ$ or $\bbN$ as categories, we will always
    implicitly assume that they are given the structure of a poset category,
    with their usual ordering $\leq$.
    We will use the notations $\bbZ^\delta$ and
    $\bbN^\delta$ to denote
    their underlying discrete categories.
  \item Unless otherwise stated, whenever we refer to $\bbZ$, $\bbZ\op$ or
    $\bbZ^\delta$ as symmetric monoidal categories, we consider them endowed
    with the symmetric monoidal structure induced by addition.
  \item We refer to a presentable \cat endowed with a symmetric monoidal
    structure whose tensor product preserves colimits in each variable as a
    \emph{presentably symmetric monoidal \cat}.
  \item We refer to a stable \cat endowed with a symmetric monoidal
    structure whose tensor product is exact in each variable as a
    \emph{stably symmetric monoidal \cat}.
  \item We refer to a presentable stable \cat endowed with a symmetric monoidal
    structure whose tensor product preserves colimits in each variable as a
    \emph{stable presentably symmetric monoidal \cat}.
  \item We will most of the times omit the Eilenberg--MacLane functor from
    the notations, when considering an Abelian group as a spectrum.
\end{enumerate}

\begin{rem}
  Throughout, when dealing with t-structures, we will use the homological
  grading convention.
  When dealing with spectral sequences, we will use the cohomological
  Serre grading convention.
  For the sake of clarity, we decided to stick with
  the choice of working exclusively with decreasing filtered objects and
  coherent cochain complexes.
  The results of this paper translate immediately
  to the case of increasing filtrations and coherent chain complexes,
  with the caveat that the equivalence of Theorem \ref{thm-general-equivalence}
  always inverts
  the direction of the $\bbZ^\delta$-indexed arrows, hence complete
  \emph{increasing} filtered objects are equivalent to coherent \emph{chain}
  complexes.
\end{rem}

\begin{rem}\label{rem-size-conventions}
  Throughout, we (mostly implicitly)
  work in ZFC+U, where ``U'' is Tarski's axiom: ``For each set $x$, there
  exists a Grothendieck universe $U$ such that $x\in U$''.
  Equivalently we assume the existence of infinitely many strongly inaccessible
  cardinals, and we fix one as the cardinality of our universe of small sets.
  As we are agnostic about the size of the universe fixed, all the arguments
  that do not rely on dealing with a bigger universe hold regardless of
  the actual size of the objects referred to as ``small sets''.
  In those few cases where we need to deal with more than one universe at a
  time, we are going to be explicit about relative sizes.
\end{rem}

\subsection{Related works}
Some of the structural results about filtered objects in stable \cats
contained in Section \ref{section-filandch} have been presented for the
case of filtered spectra in \cite[Sections 2-3]{lurie2015rotation}.
The \cat of filtered objects of a stable \cat has also been studied in
\cite[Section 2]{Gwilliam-Pavlov}, and some results in Section
\ref{section-filandch} and Section \ref{section-monoidal} overlap with
\emph{loc.\ cit.}\footnote{the reader should note that in
\cite{Gwilliam-Pavlov} the authors refer to what we call
\emph{filtered} objects as ``sequences'', and reserve the term ``filtered
objects'' for what we call \emph{complete filtered objects}}.
In \cite{raksit}, Raksit considers an alternative formulation for coherent
chain complexes; although we don't dwelve into a detailed comparison,
the results of Section \ref{section-toda}, should give a clear idea about
the relation
between the formulation in \emph{op.\ cit.}\ and the one in this paper.
In \cite{waldeDK}, Walde proves the equivalence between $\Fun(\bbN,\scrC)$ and
$\Ch_{*\geq 0}(\scrC)$ through their equivalence with
$\Fun(\simplex\kern-.5em\op,\scrC)$.
A similar approach to the construction of the spectral sequence of a filtered
spectrum through the d\'{e}calage functor
that does not use the language of coherent cochain complexes appeared
recently in Hedenlund's Ph.D.\ thesis \cite{hedenlundPhD}.
The d\'{e}calage functor introduced in Section \ref{section-spseq} is
further analyzed and discussed in forthcoming work by
Hedenlund--Krause--Nikolaus.

\subsection{Prerequisites}
We assume the reader is familiar with the theory of \cats as developed
in \cite{HTT} and with the contents of \cite{HA}. In particular, we assume
the reader is thoroughly acquainted with the theory of t-structures as
developed in \cite[1.2.1]{HA} (see also\cite[Appendix A]{AnNik20}).
Some of the terminology and notations in this paper differ from the ones used
in \cite{HTT} and \cite{HA}, but we use the same notations and terminology
of \emph{op.\ cit.}\ for the concepts we do not explicitly recall or introduce
here.

\subsection{Aknowledgements}
This paper is part of my Ph.D. project at the University of M\"{u}nster.
I am very grateful to my advisor, Thomas Nikolaus, for suggesting
this topic and for his support and advice.
I would like to thank Benjamin Antieau, from whom I learned about the relation
between the Beilinson t-structure and Deligne's d\'{e}calage, and Fosco
Loregian, for having introduced me to co/end calculus way before I could
appreciate its value.
I am also grateful to Achim Krause, Edoardo Lanari, Jonas McCandless and
Tashi Walde for the insightful discussions I shared with them while writing this
article.

\section{Definitions and preliminaries}\label{section-prelim}
In this section we introduce and discuss a few key facts about coherent
cochain complexes and filtered objects in \cats.
In particular, in Proposition \ref{prop-coch-pst-recoll} and Proposition
\ref{prop-cfild-bousfield} we prove that both constructions figure in suitable
recollements (see Definition \ref{defi-recollement}) of stable \cats.

\begin{defi}\label{def-ch}\cite[35.1]{JoyalNotes}
  Let $\catCh$ be the pointed (ordinary) category given by
  $$\operatorname{ob}\catCh = \bbZ \cup \{\pt\}$$
  $$\catCh(n,m)=
    \begin{cases}
      \{\id,\zero\} \quad &\text{if } m=n; \\
      \{\de_n,\zero\} \quad &\text{if } m=n-1; \\
      \{\zero\} \quad &\text{otherwise }.
    \end{cases}$$
  where $\pt\in\catCh$ is a zero object, and $\zero$ is the zero map.

  Given any pointed \cat $\scrC$
  \begin{enumerate}
    \item the \cat of \emph{coherent chain complexes in $\scrC$} is the full
      subcategory
      $$\Cha(\scrC)\coloneqq\Fun^{0}(\catCh,\scrC)\subset\Fun(\catCh,\scrC)$$
      spanned by pointed functors.
    \item The \cat of \emph{coherent cochain complexes} is the full
      subcategory
      $$\coCh(\scrC)\coloneqq\Fun^{0}(\catCh\op,\scrC)\subset
        \Fun(\catCh\op,\scrC)$$
      spanned by pointed functors.
    \item An object $C\in\coCh(\scrC)$ is \emph{bounded above} if
      there exists an $n\in\bbZ$ such that
      $$C^k \simeq 0 \text{ for all } k>n.$$
    \item An object $C\in\coCh(\scrC)$ is \emph{bounded below} if
      there exists an $n\in\bbZ$ such that
      $$C^k \simeq 0 \text{ for all } k<n.$$
    \item An object $C\in\coCh(\scrC)$ is \emph{bounded} if it is
      bounded above and bounded below.
  \end{enumerate}
\end{defi}

\begin{rem*}
  Notice that $0\in\bbZ$ is an object of $\catCh$, but it is not its zero
  object.
\end{rem*}

\begin{ex}
  Let $\scrA$ be an Abelian category, then $\coCh(\scrA)$ is the usual category
  of cochain complexes of $\scrA$.
\end{ex}

In what follows, we will exclusively focus our attention on coherent cochain
complexes.

\begin{notat}
  Given a coherent cochain complex $C\in\coCh{\scrC}$, we will denote $C(\de_n)$
  by $\de^n_{C}$, or just by $\de^n$ if there is no risk of confusion.
\end{notat}

\begin{defi}
  Let $\scrC$ be any \cat.
  \begin{enumerate}
    \item The category $\Fild\scrC\coloneqq\Fun(\bbZ\op, \scrC)$ is called
      the \cat of \emph{(decreasing) filtered objects of $\scrC$}.
    \item The category $\Fili\scrC\coloneqq\Fun(\bbZ,\scrC)$ is called the
      \cat of \emph{(increasing) filtered objects of $\scrC$}
      \footnote{the position of the arrows in the notation is meant to remind
        the common convention of using upper indices for decreasing filtrations,
        and lower indices for increasing filtrations}.
    \item Given $F\in\Fild\scrC$, we call $F^{-\infty}\coloneqq\colim_i F^i$
      the \emph{underlying object of $F$}.
    \item Given $F\in\Fild\scrC$, we say that $F$ is \emph{complete} if
      $F^{+\infty}\coloneqq\lim_i F^i\simeq 0$.
      We will denote by $$\rlywdhat{\Fild}\scrC$$ the full subcategory
      of $\Fild\scrC$ spanned by complete objects.
  \end{enumerate}
\end{defi}

Throughout, we will concentrate on the case of descending filtrations.

\begin{defi}\label{defi-gr}
  Let $\scrC$ be an \cat with cofibers and countable coproducts.
  Precomposition with the inclusion map $\iota_n\colon\Delta^{\{n+1, n\}} \to
  \bbZ\op$
  induces functors
  \begin{displaymath}
  \begin{split}
    (\iota_n)^*\colon\Fild\scrC&\to\Fun(\Delta^1,\scrC)\\
    F&\mapsto \left(F^{n+1}\to F^n\right).
  \end{split}
  \end{displaymath}
  \begin{enumerate}
    \item We will define the \emph{$n$-th graded functor} $\gr^n$ as the
      composite functor
      $$\Fild\scrC\xto{(\iota_n)^*}\Fun(\Delta^1,\scrC)\xto{\cofib}\scrC.$$
    \item We define the \emph{associated graded functor} $\gr$ as the
      composite functor
      $$\Fild\scrC\xto{\left(\gr^n\right)_{n\in\bbZ}}
        \prod_{n\in\bbZ}\scrC\xto{\bigoplus}\scrC.$$
  \end{enumerate}
\end{defi}

\begin{defi}
  We say that a map $\alpha\colon F\to G$ in $\Fild\scrC$ is a \emph{graded
  equivalence} if $\alpha\colon F\to G$ is such that $\gr(\alpha)$ is an
  equivalence. Equivalently. if for all $n\in\bbZ$, the dotted map
  $$
  \begin{tikzcd}
    F^{n+1} \ar[d] \ar[r] & F^n \ar[d] \ar[r] & \gr^n F \ar[d, dashed]\\
    G^{n+1} \ar[r] & G^n \ar[r] & \gr^n G
  \end{tikzcd}
  $$
  induced by universality of cofibers is an equivalence.
\end{defi}

\begin{notat}
  Let $\scrC$ be an \cat. For $n,m\in\bbZ\cup\{+\infty,-\infty\}$, with
  $n\leq m$, we use the notation
  $$F^n/F^m\coloneqq\cofib\left(F^m\to F^n\right)$$
  to denote the cofiber of the evident map.
\end{notat}

One of the nice features of ordinary cochain complexes is the possibility to
write them as limits of bounded above (or below) ones. This feature is still
present in the coherent setting, as we now show.

\begin{constr}\label{constr-trunc-ch}
  Given any integer $n$, let $\catCh_{(-\infty,n]}$
  denote the full subcategory of $\catCh$ spanned by
  $\{\pt,n,n-1,n-2,\cdots\}$.
  Let now $\scrC$ denote a pointed complete\footnote{or, more generally, such
  that all the relevant Kan extensions exist and are pointwise}
  \cat. Then, each inclusion $\iota^n\colon\catCh_{(-\infty,n]}\to\catCh$
  induces, by right Kan extension along it, an adjunction
  $$
  \begin{tikzcd}[column sep=huge]
    \coCh(\scrC)
    \ar[r, shift left=1.1ex, "(\iota^n)^*"] &
    \Fun^0\left(\catCh\op_{(-\infty,n]},\scrC\right).
    \ar[l, shift left=1.1ex,"\iota^n_*"]
    \ar[l,phantom,"\text{\rotatebox{-90}{$\dashv$}}"]
    \end{tikzcd}
  $$
  For any given $C\in\coCh(\scrC)$, we can compute explicitly the values of
  $\iota^n_*C$.
  As $\iota^n$ is fully faithful, the only values we need to determine are
  the ones for $M>n$:
  \begin{displaymath}
  \begin{split}
    \iota^n_*C(M)
    \stackrel{\vphantom{(}^{(\ref{rem-end-ran-formula})}}{\simeq}
      &\int_{s\in\catCh\op_{(-\infty,n]}}\left[\Map_{\catCh\op}(M,s),C^s\right]\\
    \stackrel{\vphantom{(}^{\hphantom{(\ref{rem-end-ran-formula})}}}{\simeq}
      &\int_{s\in\catCh\op_{(-\infty,n]}}[\pt,C^s]\\
    \stackrel{\vphantom{(}^{(\ref{cor-dummy-coend})}}{\simeq}
      &\lim_{\catCh\op_{(-\infty,n]}}C\\
    \stackrel{\vphantom{(}^{\hphantom{(\ref{cor-dummy-coend})}}}{\simeq}
      &0
  \end{split}
  \end{displaymath}
  where the last limit is $0$ as the indexing category is pointed, and $C$
  preserves the zero object.
  If we now denote by $(-)^{\leq n} \coloneqq \iota^n_*(\iota^n)^*$, we have
  that the unit of the adjunction $(\iota^n)^*\dashv\iota^n_*$ induces an
  endofunctor of $\coCh(\scrC)$, denoted $(-)^{\leq n}$,
  sending a complex $C$ to the complex $C^{\leq n}$, which in degree $m$
  is given by
  $$ \begin{cases}
      C^m &\text{ if } m\leq n\\
      0   &\text{ else.}
     \end{cases} $$
  Similarly, the inclusions $\iota^{n,n+1}\colon\catCh_{(-\infty,n]}\to
  \catCh_{(-\infty,n+1]}$ induce adjunctions
  $$
  \begin{tikzcd}[column sep=huge]
    \Fun^0\left(\catCh\op_{(-\infty,n+1]},\scrC\right)
    \ar[r, shift left=1.1ex, "(\iota^{n,n+1})^*"] &
    \Fun^0\left(\catCh\op_{(-\infty,n]},\scrC\right).
    \ar[l, shift left=1.1ex,"\iota^{n,n+1}_*"]
    \ar[l,phantom,"\text{\rotatebox{-90}{$\dashv$}}"]
    \end{tikzcd}
  $$
  Notice that we have a natural equivalence
  $$
  (\iota^{n+1}\circ\iota^{n,n+1})^* \simeq
  (\iota^{n,n+1})^*(\iota^{n+1})^* \simeq (\iota^n)^*;
  $$
  by passing to adjoints, we get a natural transformation
  \begin{equation}\label{eq-compose}
    (\iota^{n+1})^* \Rightarrow \iota^{n,n+1}_*(\iota^n)^*;
  \end{equation}
  if we now precompose (\ref{eq-compose})
  with $\iota^{n+1}_*$ (using that adjunctions compose), we get a natural
  transformation
  $$\upsilon^{n+1}\colon(-)^{\leq n+1}\Rightarrow(-)^{\leq n}.$$
  By inspection, the above is given pointwise by
  $$
  (\upsilon^{n+1}_C)^m\simeq
  \begin{cases}
    \id\colon C^m\to C^m &\text{ if } m\leq n \\
    0\colon C^{n+1} \to 0 &\text{ if } m = n+1 \\
    0\colon 0 \to 0 &\text{ if } m \geq n+2.
  \end{cases}
  $$
\end{constr}

\begin{rem}
  Similarly to what we did in Construction \ref{constr-trunc-ch}, one can
  truncate \emph{below} a certan integer. It is also possible to consider
  \emph{left} Kan extensions along the $\iota^n$'s, and the induced counits
  to obtain different truncations with $\cofib \partial^n$ in degree $n+1$,
  for truncations above $n$, or $\fib \partial^n$ in degree $n-1$,
  for truncations below $n$. In what follows, we won't need any of such
  variants.
\end{rem}

\begin{lemma}\label{lemma-coch-limit-of-truncations}
  Let $\scrC$ be a pointed complete \cat. Then, for any $C\in\coCh\scrC$
  $$
  C\simeq\lim \left(\cdots \to C^{\leq n+1}\xto{\upsilon^{n+1}_C}
  C^{\leq n} \to \cdots\right)
  $$
  (where the $\upsilon^n$'s are the natural transformations defined in
  Construction \ref{constr-trunc-ch}).
\end{lemma}
\begin{proof}
  It follows from Proposition \ref{prop-coch-pst-recoll} that limits in
  $\coCh\scrC$ can be computed objectwise. As for any $m\in\bbZ$ the
  sequence
  $$
  \cdots \to \left(C^{\leq n+1}\right)^m\xto{\upsilon^n_C}
      \left(C^{\leq n}\right)^m \to \cdots
  $$
  is eventually constant on the left, the result follows.
\end{proof}

\begin{lemma}\label{lemma-coch-to-prod-conservative}
  Given any pointed \cat $\scrC$,
  the functor $\undch\colon\coCh(\scrC) \to \prod_{\bbZ} \scrC$ induced by
  precomposition with $\bbZ^\delta\to\catCh\op$ is conservative.
\end{lemma}
\begin{proof}
  As equivalences in \cats are detected at the level of homotopy categories,
  this is clear.
\end{proof}

\begin{prop}\label{prop-eval-adj}
  Let $\scrC$ be a complete and cocomplete semiadditive \cat, and let
  $n^*\colon\Fun(\catCh\op,\scrC)\to\Fun(\Delta^0,\scrC)$
  denote precomposition with $n\colon\Delta^0\to\catCh\op$.
  The ``evaluation at $n$'' functor
  $$\ev_n\colon\coCh(\scrC)\subset\Fun(\catCh\op,\scrC)
    \xto{n^*}\scrC,$$
  admits both a left and a right adjoint.
  Such adjoints are given objectwise by
  $$
  (\ev_n)_*X^m\simeq
  \begin{cases}
    X \quad &\text{if } m=n-1, n\\
    0 &\text{else.}
  \end{cases}
  $$
  with the identity as the only nontrivial differential, and
  $$
  (\ev_n)_!X^m\simeq
  \begin{cases}
    X \quad &\text{if } m=n, n+1\\
    0 &\text{else.}
  \end{cases}
  $$
  with the identity as the only nontrivial differential.
\end{prop}
\begin{proof}
  The functor $n^*$ admits both adjoints $n_!$ and $n_*$, given respectively
  by left and right Kan extension.
  As adjoint functors compose,
  it follows from Proposition \ref{prop-coch-pst-recoll} that the right
  Kan extension $(\ev_n)_*$ is given by the composite $\ptize\circ n_*$,
  and the left Kan extension $(\ev_n)_!$ is given by $\ptize\circ n_!$.
  By Remark \ref{rem-end-ran-formula}, $n_*$ is given by
  $$
  n_*X^m= [\Map_{\catCh\op}(m,n),X] =
  \begin{cases}
    X \oplus X \quad &\text{if } m=n-1, n\\
    X &\text{else.}
  \end{cases}
  $$
  with the differentials $n_*X^m\to n_*X^{m+1}$ determined by
  $$\Map_{\catCh\op}(m+1,n) \xto{(\partial_{m}\op)^*} \Map_{\catCh\op}(m,n)$$
  and thus given by
  \begin{equation}\label{eq-complex-ext}
  \begin{tikzcd}[ampersand replacement=\&]
    \cdots \ar[r] \&
    X \ar[r,"{\begin{pmatrix} \id \\ \id \end{pmatrix}}"] \&
    X\oplus X \ar[r,"{\begin{pmatrix} \id & 0 \\ 0 & \id \end{pmatrix}}"] \&
    X\oplus X \ar[r,"{\begin{pmatrix} \id & 0 \end{pmatrix}}"] \&
    X \ar[r] \&
    \cdots
  \end{tikzcd}
  \end{equation}
  in degrees $n-2$ to $n+1$, and by identities elsewhere.

  By Lemma \ref{lemma-coend-formula}, $n_!$ is given by
  $$
  n_*X^m= \Map_{\catCh\op}(n,m) \tensor X =
  \begin{cases}
    X \oplus X \quad &\text{if } m=n,n+1\\
    X &\text{else.}
  \end{cases}
  $$
  with the differentials $n_*X^m\to n_*X^{m+1}$ determined by
  $$\Map_{\catCh\op}(n,m) \xto{(\partial_{m}\op)_*} \Map_{\catCh\op}(n,m+1)$$
  and thus given by
  \begin{equation}
  \begin{tikzcd}[ampersand replacement=\&]
    \cdots \ar[r] \&
    X \ar[r,"{\begin{pmatrix} \id \\ 0 \end{pmatrix}}"] \&
    X\oplus X \ar[r,"{\begin{pmatrix} \id & 0 \\ 0 & \id \end{pmatrix}}"] \&
    X\oplus X \ar[r,"{\begin{pmatrix} \id & \id \end{pmatrix}}"] \&
    X \ar[r] \&
    \cdots
  \end{tikzcd}
  \end{equation}
  in degrees $n-1$ to $n+2$, and by identities elsewhere.

  By the explicit description of $\ptize$ given in Proposition
  \ref{prop-coch-pst-recoll}, we get the formulas for $(\ev_n)_*$ and
  $(\ev_n)_!$.
\end{proof}

We now recall some definitions and facts about recollements,
which we will use extensively in the following sections; we consider only
recollements in the case of stable \cats,
but the theory holds in greater generality; see also
\cite{FLRecoll}, \cite{BGRecoll} and \cite[A.8]{HA}.

\begin{defi}\label{defi-recollement}
  Let $\scrC$ be a stable \cat, and let $i\colon\scrC_0\hookrightarrow\scrC$ and
  $j\colon\scrC_1\hookrightarrow\scrC$ be full subcategories.
  We say that
  $\scrC$ is a \emph{recollement of the essential image of $i$ and
  the essential image of $j$} if:
  \begin{enumerate}
    \item Both $i$ and $j$ admit left adjoints:
      $$
      \begin{tikzcd}
        \scrC_0 \ar[r, hook, "i"'] &
        \scrC \ar[l, shift right=0.6ex, bend right, "i_L"']
          \ar[r, "j_L"]
          \ar[from=r, hook', shift left=0.6ex, bend left, "j"]
          \ar[l,phantom, shift right=1.2ex,
            "\text{\rotatebox{-90}{$\dashv$}}"] &
        \scrC_1
          \ar[l,phantom, shift left=1ex, "\text{\rotatebox{-90}{$\dashv$}}"]
      \end{tikzcd}
      $$
    \item The functor $j_L$, left adjoint to $j$, carries every object of
      $\scrC_0$ to zero;
    \item If $\alpha$ is a morphism of $\scrC$ such that $i_L(\alpha)$ and
      $j_L(\alpha)$ are equivalences, then $\alpha$ is an equivalence.
  \end{enumerate}
\end{defi}

\begin{rem}\label{rem-full-recol}
  It follows from \cite[A.8.5, A.8.19]{HA} that if $\scrC$ is a recollement of
  $\scrC_0$ and $\scrC_1$, then we actually have the following adjunctions
  $$
  \begin{tikzcd}[column sep=huge]
    \scrC_0 \ar[r, hook, "i"' description] &
    \scrC \ar[l, shift right=0.6ex, bend right, "i_L"']
      \ar[from=r, hook, shift right=0.6ex, bend right, "(j_L)_!"']
      \ar[r, "j_L" description]
      \ar[from=r, hook', shift left=0.6ex, bend left, "j"]
      \ar[l, shift left=0.6ex, bend left, "i_R"]
      \ar[l,phantom, shift left=2ex,
        "\text{\scalebox{1}{\rotatebox{-90}{$\dashv$}}}"]
      \ar[l,phantom, shift right=2ex,
        "\text{\scalebox{1}{\rotatebox{-90}{$\dashv$}}}"] &
    \scrC_1
      \ar[l,phantom, shift left=2ex,
        "\text{\scalebox{1}{\rotatebox{-90}{$\dashv$}}}"]
      \ar[l,phantom, shift right=2ex,
        "\text{\scalebox{1}{\rotatebox{-90}{$\dashv$}}}"]
  \end{tikzcd}
  $$
  where $(j_L)_!$ is fully faithful, and $i_R$ is such that
  $$ii_R \to \id_\scrC \to jj_L$$
  is a co/fiber sequence.
\end{rem}

\begin{prop}[]\cite[Proposition A.8.20]{HA}\label{prop-equi-recollement}
  Let $\scrC$ be a stable \cat, and let $i\colon\scrC_0\to\scrC$ be a fully
  faithful functor. The following are equivalent:
  \begin{enumerate}
    \item The functor $i$ admits a left adjoint and a right adjoint;
    \item There exists a full subcategory $j\colon\scrC_1\hookrightarrow\scrC$,
      closed under equivalences, such that $\scrC$ is the recollement of the
      essential images of $i$ and $j$.
  \end{enumerate}
  Moreover, if the conditions above hold, we can identify $\scrC_1$ with
  the full subcategory $\scrC_0^\perp\subseteq\scrC$ spanned by those
  objects $X\in\scrC$ such that for all $C\in \scrC_0$,
  $\Map_{\scrC}(C,X)\simeq\pt.$
\end{prop}

Recollements are strictly related to semiorthogonal decompositions.

\begin{defi}\label{defi-semiorthogonal}
  Let $\scrC$ be a stable \cat. A \emph{semiorthogonal decomposition} of
  $\scrC$ is the datum of two full subcategories $\scrC_0$ and $\scrC_1$ of
  $\scrC$ such that
  \begin{enumerate}
    \item $\scrC_1 \simeq \scrC_0^\perp$;
    \item Every object $C\in\scrC$ sits in a cofiber sequence
      $$C_0\to C\to C_1$$
      where $C_0\in\scrC_0$ and $C_1\in\scrC_1$.
  \end{enumerate}
\end{defi}

\begin{rem}\label{rem-semiorthogonal}
  It follows from Remark \ref{rem-full-recol} that every recollement
  determines a semiorthogonal decomposition.
\end{rem}

We learned the following argument from \cite[Proof of Lemma 3]{BGRecoll}, and
we present it here almost verbatim for the reader's convenience.

\begin{prop}\label{prop-fracture-square}
  The commutative square
  $$
  \begin{tikzcd}
    \id_\scrC \ar[r, "\eta^i"] \ar[d, "\eta^j"'] & i i_L \ar[d, "i i_L \eta^j"] \\
    j j_L \ar[r, "\eta^j j j_L"] & i i_L j j_L
  \end{tikzcd}
  $$
  is Cartesian in $\Fun(\scrC,\scrC)$.
\end{prop}
\begin{proof}
  It follows from Remark \ref{rem-full-recol} that the fiber of the
  vertical maps is given by $\eta^i i i_R \colon i i_R \to i i_L i i_R$,
  which, as $i i_L \simeq \id$, is an equivalence.
\end{proof}

\begin{rem}
  In particular, every recollement determines a ``fracture square''
  $$
  \begin{tikzcd}
    C \ar[r] \ar[d] & i_L C \ar[d] \\
    j_L C \ar[r] & i_L j_L C
  \end{tikzcd}
  $$
  for any object $C\in\scrC$.
\end{rem}

We are now ready to prove the main results of this section.

\begin{prop}\label{prop-coch-pst-recoll}
  Let $\scrC$ be a stable \cat that is both complete and cocomplete.
  The \cat $\Fun(\catCh\op,\scrC)$ is the recollement of the essential images
  of $\const\colon\scrC\to\Fun(\catCh\op,\scrC)$ and $\coCh(\scrC)$:
  $$
  \begin{tikzcd}[column sep=huge]
    \scrC \ar[r, hook, "\const"' description, pos=.7] &
    \Fun(\catCh\op,\scrC)
      \ar[l, shift right=0.6ex, bend right, "\colim" description]
      \ar[from=r, hook, shift right=0.6ex, bend right, "i" description, pos=.4]
      \ar[r, "\ptize" description]
      \ar[from=r, hook', shift left=0.6ex, bend left, "i" description, pos=.4]
      \ar[l, shift left=0.6ex, bend left, "\lim" description]
      \ar[l,phantom, shift left=2ex,
        "\text{\scalebox{1}{\rotatebox{-90}{$\dashv$}}}", pos=.3]
      \ar[l,phantom, shift right=2ex,
        "\text{\scalebox{1}{\rotatebox{-90}{$\dashv$}}}", pos=.3] &
    \coCh(\scrC)
      \ar[l,phantom, shift left=2ex,
        "\text{\scalebox{1}{\rotatebox{-90}{$\dashv$}}}"]
      \ar[l,phantom, shift right=2ex,
        "\text{\scalebox{1}{\rotatebox{-90}{$\dashv$}}}"]
  \end{tikzcd}
  $$
  Moreover, the inclusion $i\colon\coCh(\scrC)\to\Fun(\catCh\op,\scrC)$ is both
  left and right adjoint to a functor $\ptize$, given by
  $$\ptize F(n)\simeq F(n)/F(\pt)$$
  on objects.
  Finally, the \cat $\coCh(\scrC)$ is stable, complete and cocomplete;
  if $\scrC$ is presentable, $\coCh(\scrC)$ is presentable as well.
\end{prop}
\begin{proof}
  The functor $\const$ admits a left and a right adjoint, given respectively
  by left and right Kan extension along the terminal morphism (i.e. the
  $\colim$ and $\lim$ functors). Hence, by Proposition
  \ref{prop-equi-recollement}, $\Fun(\catCh\op,\scrC)$
  is the recollement of (the essential image of) $\const$ and the full
  subcategory $\const(\scrC)^\perp$ orthogonal to it.
  We claim that
  \begin{equation}\label{eq-orth-pstch}
    \const(\scrC)^\perp\simeq\coCh(\scrC).
  \end{equation}
  To see this, recall that $\const(\scrC)^\perp$ is spanned by those objects
  $F\in\Fun(\catCh\op,\scrC)$ such that for all $X\in\scrC$ the mapping space
  $\Map_{\Fun(\catCh\op,\scrC)}(\const_X,F)$ is contractible. But, as
  $$\Map_{\Fun(\catCh\op,\scrC)}(\const_X,F)\simeq\Map_{\scrC}(X,F(\pt)),$$
  we see that (\ref{eq-orth-pstch}) holds.
  We denote by $\ptize$ the left adjoint to the fully faithful inclusion
  $$i\colon\coCh(\scrC)\to\Fun(\catCh\op,\scrC).$$
  Now, $\coCh(\scrC)$
  is stable by \cite[A.8.17]{HA} (or just because of the fact that (co)limits
  in functor categories are computed pointwise)
  and, if $\scrC$ is presentable (as $\Fun(\catCh\op,\scrC)$ is presentable
  by \cite[5.5.3.6]{HTT}) $\coCh(\scrC)$ is presentable by
  \cite[1.4.4.9]{HA}. By \cite[A.8.5]{HA}, we see that
  \begin{equation}\label{eq-pt-cof-cns}
    \ptize F\simeq \cofib\left(\const_{F(\pt)}\to F\right)
  \end{equation}
  and in particular $\ptize F(n)\simeq F(n)/F(\pt)$.
  The existence of a left adjoint for $\ptize$
  follows from Remark \ref{rem-full-recol}.
  To give a description, let's first note that, by (\ref{eq-pt-cof-cns})
  and the fully faithfulness of the inclusion $i$
  we can write
  \begin{displaymath}
  \begin{split}
    \Map_{\coCh(\scrC)}(C,\ptize X)
    &\simeq \Map_{\Fun(\catCh\op,\scrC)}\big(C,\cofib(\const_{F(\pt)}\to X)\big)\\
    &\simeq \Map_{\Fun(\catCh\op,\scrC)}\big(C[-1],\fib(\const_{F(\pt)}\to X)\big)
  \end{split}
  \end{displaymath}
  and, as corepresentable functors commute with limits, the latter is
  equivalent to
  $$\fib\left(\Map_{\Fun(\catCh\op,\scrC)}\left(C[-1],\const_{F(\pt)}\right)
    \to \Map_{\Fun(\catCh\op,\scrC)}\left(C[-1],X\right)\right)$$
  which in turn, as $\const$ is right adjoint to $\colim$, can be computed as
  \begin{equation}\label{eq-in-lemma-ptize}
  \fib\left(\Map_{\scrC}\big(\colim C[-1],F(\pt)\big)
      \to \Map_{\Fun(\catCh\op,\scrC)}\left(C[-1],X\right)\right).
  \end{equation}
  But, as the colimit of a coherent cochain complex is always $0$, we have
  that
  $$\Map_{\scrC}\big(\colim C[-1],F(\pt)\big)\simeq\pt$$
  and thus (\ref{eq-in-lemma-ptize}) is equivalent to
  \begin{displaymath}
  \begin{split}
    \fib\left(\pt\to \Map_{\Fun(\catCh\op,\scrC)}\left(C[-1],X\right)\right)
      \simeq & \ \Omega \Map_{\Fun(\catCh\op,\scrC)}\left(C[-1],X\right)\\
      \simeq & \ \Map_{\Fun(\catCh\op,\scrC)}\left(C,X\right).
  \end{split}
  \end{displaymath}
  In particular, as $\ptize$ is both left and right adjoint to the inclusion,
  we have that $\coCh(\scrC)$ is closed under all limits and colimits
  in $\Fun(\catCh\op,\scrC)$, which is by hypothesis complete and cocomplete.
\end{proof}

\begin{prop}\label{prop-cfild-bousfield}
  Let $\scrC$ be a stable \cat that is both complete and cocomplete.
  The \cat $\Fild\scrC$ is the recollement of the essential images
  of $\const\colon\scrC\to\Fild\scrC$ and $i\colon\cFild\scrC$:
  $$
  \begin{tikzcd}[column sep=huge]
    \scrC \ar[r, hook, "\const"' description] &
    \Fild\scrC \ar[l, shift right=0.6ex, bend right, "\colim"' description, pos=.45]
      \ar[from=r, hook, shift right=0.6ex, bend right, "L_!" description]
      \ar[r, "L" description]
      \ar[from=r, hook', shift left=0.6ex, bend left, "i" description]
      \ar[l, shift left=0.6ex, bend left, "\lim" description, pos=.55]
      \ar[l,phantom, shift left=2ex,
        "\text{\scalebox{1}{\rotatebox{-90}{$\dashv$}}}"]
      \ar[l,phantom, shift right=2ex,
        "\text{\scalebox{1}{\rotatebox{-90}{$\dashv$}}}"] &
    \cFild\scrC.
      \ar[l,phantom, shift left=2ex,
        "\text{\scalebox{1}{\rotatebox{-90}{$\dashv$}}}"]
      \ar[l,phantom, shift right=2ex,
        "\text{\scalebox{1}{\rotatebox{-90}{$\dashv$}}}"]
  \end{tikzcd}
  $$
  The left adjoint to the inclusion $i$
  is given by Bousfield localization at the class of graded equivalences,
  and computed as
  $$LF^n\simeq F^n/F^{+\infty}.$$
  Moreover, the \cat $\cFild\scrC$ is stable, complete and cocomplete; if
  $\scrC$ is presentable, $\cFild\scrC$ is presentable as well.
\end{prop}
\begin{proof}
  The proof is completely analogous to the proof of
  Proposition \ref{prop-coch-pst-recoll}. The only new element is the
  identification of the local maps for the Bousfield localization determined
  by $L$. In order to prove it,
  let us notice that $L\alpha$ is an equivalence if and only
  if $\cofib L\alpha \simeq L (\cofib \alpha) \simeq 0$, which in turn,
  by \cite[A.8.5]{HA}, is the case if and only if $\cofib \alpha$ is
  essentially constant;
  by inspection of the following diagram
  $$
  \begin{tikzcd}
    F^{n+1} \ar[r] \ar[d] & F^n \ar[r] \ar[d] & \gr^n F \ar[d] \\
    G^{n+1} \ar[r] \ar[d] & G^n \ar[r] \ar[d] & \gr^n G \\
    (\cofib \alpha)^{n+1} \ar[r] & (\cofib \alpha)^n &
  \end{tikzcd}
  $$
  where all the rows and columns are co/fiber sequences,
  we see that this is the case if and only if $\alpha$ is a graded equivalence.
\end{proof}

We conclude this section with the following fact about fully faithful
adjoint functors, which we will use later.

\begin{prop}\label{prop-eqv-induces-unit-eqv}
  Let $\scrC$ and $\scrD$ be \cats, and let
  $$
  \begin{tikzcd}[column sep=huge]
    \scrC \ar[r, shift left=1.1ex, "F"] &
    \scrD \ar[l, shift left=1.1ex,"G"]
      \ar[l,phantom,"\text{\rotatebox{-90}{$\dashv$}}"]
  \end{tikzcd}
  $$
  be an adjunction. If there exists a natural isomorphism $\alpha\colon
  \id_{\scrC}\stackrel{\sim}{\Longrightarrow}GF$ then the unit of the
  adjunction is an equivalence (equivalently, $F$ is fully faithful).
\end{prop}
\begin{proof}
  By \cite{RV2}, it is possible to associate to any adjunction a homotopy
  coherent monad on $GF$, whose 1-skeletal part looks as follows
  \begin{displaymath}
  \begin{tikzcd}
  \id_{\scrC} \ar[r,"\eta" description] &[10pt]
  GF \ar[r, shift left=1.5ex, near end, "\eta GF" description]
    \ar[r, shift left=-1.5ex, near end, "GF \eta" description] &[40pt]
  GFGF \ar[l, near end, "G\eps F" description]
    \ar[r, near end, "GF \eta GF" description]
    \ar[r, shift left=3ex, near end, "\eta GFGF" description]
    \ar[r, shift left=-3ex, near end, "GFGF \eta" description] &[80pt]
  GFGFGF \ar[l, shift left=1.5ex, near end, "GFG\eps F" description]
    \ar[l, shift left=-1.5ex, near end, "G\eps FGF" description]
  \end{tikzcd}\cdots
  \end{displaymath}
  as a diagram in $\Fun(\scrC,\scrC)$.
  As, by hypothesis, $\id_{\scrC}\simeq GF$ via some $\alpha$, the homotopy
  coherent monad structure on $GF$ transfers to a homotopy coherent monad
  on $\id_{\scrC}$, whose 1-skeletal part looks as follows
  \begin{displaymath}
  \begin{tikzcd}
  \id_{\scrC} \ar[r,"\wt\eta" description] &[10pt]
  \id_{\scrC} \ar[r, shift left=1.5ex, near end, "\wt\eta" description]
    \ar[r, shift left=-1.5ex, near end, "\wt\eta" description] &[20pt]
  \id_{\scrC} \ar[l, near end, "\mu" description]
    \ar[r, near end, "\wt\eta" description]
    \ar[r, shift left=3ex, near end, "\wt\eta" description]
    \ar[r, shift left=-3ex, near end, "\wt\eta" description] &[20pt]
  \id_{\scrC} \ar[l, shift left=1.5ex, near end, "\mu" description]
    \ar[l, shift left=-1.5ex, near end, "\mu" description]
  \end{tikzcd}\cdots
  \end{displaymath}
  Now, by unitality, $\mu\wt\eta \simeq \id_{\id_{\scrC}}$,
  and as the following diagram commutes by naturality of $\mu$ (or of
  $\wt\eta$),
  $$
  \begin{tikzcd}
    \id_{\scrC} \ar[d, "\mu"'] \ar[r, "\wt\eta"] \ar[dr, "\id"]&
      \id_{\scrC} \ar[d, "\mu"] \\
    \id_{\scrC} \ar[r, "\wt\eta"'] & \id_{\scrC}
  \end{tikzcd}
  $$
  we see that $\wt\eta\mu\simeq\id_{\id_{\scrC}}$. Thus, $\wt\eta$ is an
  equivalence, and so is $\eta$.
\end{proof}

\section{Filtered spectra and cochain complexes of spectra}
\label{section-filandch}

Our goal in this section is to construct an equivalence between the \cats
$\rlywdhat{\Fild} \Sp$ and $\coCh(\Sp)$. In order to do so, we will first
construct a pair of adjoint functors between $\Fild\Sp$ and $\coCh(\Sp)$ using
the machinery of Appendix \ref{appendix-first}, and then prove that the right
adjoint is a fully faithful functor having $\cFild\Sp$ as its essential image.
Along the way, we compute explicitly the values of the pair of adjoints
constructed abstractly.

In order to construct the left adjoint, let us start with the following
observation.

\begin{lemma}\label{lemma-bbz-coch}
  Restriction along $\susinftyp\Yo\colon\bbZ\to\Pst\bbZ$ induces an equivalence
  of \cats
  $$\Fun^{L}(\Fild(\Sp),\coCh(\Sp))\simeq\Fun(\bbZ,\coCh(\Sp))$$
  whose inverse is given by
  associating to every functor $F\colon\bbZ\to\coCh(\Sp)$ its stable
  realization functor (see Definition \ref{defi-stab-nerve-real}).
\end{lemma}
\begin{proof}
  As, by Proposition \ref{prop-coch-pst-recoll}, $\coCh(\Sp)$ is a stable
  \cat, this is just an application of Lemma \ref{st-nerve-real}, toghether
  with the observation that $\Fild\Sp \simeq \Pst{\bbZ}$.
\end{proof}

According to the previous lemma, we will have to provide a functor
$\bbZ\to\coCh(\Sp)$ in order to get the wanted left adjoint. A priori,
constructing such a functor would require to keep track of an infinite
amount of coherences, but it turns out that such coherences are essentially
trivial.
This idea is made precise by the following results.
The content of Proposition \ref{prop-spine-anodyne} is a corollary of
\cite[Tag 00J6]{kerodon}; we report a direct proof here for the reader's
convenience.

\begin{notat}
  For any $m,n\in\bbZ$, let $\mathrm{I}_{[m,n]}$ denote the simplicial set
  $$\DelPar{m,m-1} \cprPar{m-1} \cdots \cpr_{\Delta^{\{n-2\}}}
    \Delta^{\{n-2,n-1\}} \cpr_{\Delta^{\{n-1\}}} \Delta^{\{n-1,n\}}$$
  which can be informally represented as
  $$
  \begin{tikzcd}
    & \cdots \ar[dr] && n. \\
    m \ar[ur] && n-1 \ar[ur] &
  \end{tikzcd}
  $$
  Similarly, we denote by $\mathrm{I}_{[-\infty,n]}$ and
  $\mathrm{I}_{[n,+\infty]}$ respectively, the simplicial sets
  $$\cdots \cpr_{\Delta^{\{n-2\}}} \Delta^{\{n-2,n-1\}} \cpr_{\Delta^{\{n-1\}}}
    \Delta^{\{n-1,n\}}, \quad
  \Delta^{\{n,n+1\}} \cpr_{\Delta^{n+1}} \Delta^{\{n+1,n+2\}}
    \cpr_{\Delta{n+2}} \cdots$$
  and finally by
  $$\spineZ \coloneqq \cdots
    \cpr_{\Delta^{\{n-1\}}} \Delta^{\{n-1,n\}} \cpr_{\Delta^{\{n\}}}
    \Delta^{\{n,n+1\}} \cpr_{\Delta^{\{n+1\}}} \cdots$$
  which can be depicted as
  $$
  \begin{tikzcd}
    & n-1 \ar[dr] && n+1 \ar[dr] \\
    \cdots \ar[ur] && n \ar[ur] && \cdots &
  \end{tikzcd}
  $$
\end{notat}

\begin{prop}\label{prop-spine-anodyne}
  The inclusion $$\spineZ \to \bbZ$$ is inner anodyne.
\end{prop}
\begin{proof}
  Let us begin observing that, as shown in the proof of
  \cite[Proposition 2.13]{JoyThQcat}, for any $s>0$, both
  $$\Delta^{\{0,1\}} \cpr_{\Delta^{\{1\}}} \Delta^{\{1,\ldots,s\}}\to\Delta^s$$
  and
  $$\Delta^{s-1} \cpr_{\Delta^{\{s-1\}}} \Delta^{\{s-1,s\}} \to \Delta^s$$
  are inner anodyne.
  By virtue of the following commutative diagram of simplicial sets
  $$
  \begin{tikzcd}
    \DelPar{0,1} \cprPar{1} \DelPar{1,\ldots,s} \ar[d] \ar[r] &
      \DelPar{0,1} \cprPar{1} \DelPar{1,\ldots,s} \cprPar{s} \DelPar{s,s+1}
      \ar[d] \ar[dr, "\alpha"] & \\
    \Delta^s \ar[r] & \Delta^s \cprPar{s} \DelPar{s,s+1} \ar[r] &
      \Delta^{s+1}
  \end{tikzcd}
  $$
  where the square is a pushout, since the left vertical map is inner anodyne,
  the right vertical map is inner anodyne, and, as the right horizontal
  map is inner anodyne, the composite $\alpha$ is inner anodyne as well.

  Now, for any $m\in\bbZ$, let $P_m$ denote the following pushout of
  simplicial sets
  $$
  \begin{tikzcd}
    \mathrm{I}_{[-m,m]} \ar[d] \ar[r] & \Delta^{\{-m,\ldots,m\}} \ar[d]\\
    \mathrm{I}_{[-\infty,\infty]} \ar[r] & P_m.
  \end{tikzcd}
  $$
  As, by what we have seen above, the left vertical map in the following
  pushout square of simplicial sets
  $$
  \begin{tikzcd}
    \DelPar{0,1} \cprPar{1} \DelPar{1,\ldots,s} \cprPar{s} \DelPar{s,s+1}
      \ar[d] \ar[r] & P_m \ar[d,"\beta_m"] \\
    \Delta^{\{-m,\ldots,m\}} \ar[r] & P_{m+1}
  \end{tikzcd}
  $$
  is inner anodyne, we have that $\beta_m\colon P_m \to P_{m+1}$ is inner
  anodyne for any $m>0$.

  Finally, as the colimit projection $\gamma$ below
  $$
  \begin{tikzcd}
    \spineZ \simeq P_0 \ar[r,"\beta_0"] \ar[drrr,"\gamma"'] &
      P_1 \ar[r,"\beta_1"] \ar[drr]&
      P_2 \ar[r] \ar[dr] & \cdots \\
    &&& \colim_m P_m \simeq \bbZ
  \end{tikzcd}
  $$
  is a transfinite composition of inner anodyne maps, it is inner anodyne.
\end{proof}

\begin{cor}\label{cor-functors-from-Z}
  Let $\scrC$ be an \cat. To specify a functor $\bbZ\to\scrC$ it is sufficient
  to specify its value on objects and on morphisms of the form $n\to n+1$ for
  all $n\in\bbZ$.
\end{cor}

\begin{constr}\label{constr-preaush}
  Let $\preaush\colon\bbZ\to\coCh(\Sp)$ be the functor defined on objects as
  $\preaush(n)=\spch{n}{n}$, where
  \begin{displaymath}
    \begin{split}
      \spch{n}{n}\colon\catCh\op&\to\Sp \\
      m&\mapsto
      \begin{cases}
        \bbS[n] \quad &\text{if } m=n, \\
        0 \quad &\text{if } m\ne n.
      \end{cases}
    \end{split}
  \end{displaymath}
  We define $\preaush$ on morphisms as follows.
  First of all, let us note that, as
  \begin{displaymath}
    \begin{split}
      \Fun(\Delta^{1},\coCh(\Sp))
        &\simeq\Fun(\Delta^{1},\Fun^{0}(\catCh\op,\Sp))\\
        &\subset\Fun(\Delta^{1},\Fun(\catCh\op,\Sp))\\
        &\simeq\Fun(\Delta^{1}\times\catCh\op,\Sp)
    \end{split}
  \end{displaymath}
  it is equivalent to determine a map in $\coCh(\Sp)$ or a map from $\Delta^1
  \times\catCh\op$ landing in spectra, such that both restrictions to
  $\{0\}\times\catCh\op$ and $\{1\}\times\catCh\op$ preserve zero objects
  (recall that $\Fun^0(\catCh\op,\Sp)$ is a full subcategory of $\Fun(\catCh
  \op,\Sp)$).
  Let $\iota_{m}\colon\catCh\op_{[m,m+1]}\to\catCh\op$ be the inclusion functor
  Then, we define $\preaush(\iota_{m})$ as the morphism
  corresponding to the following left Kan extension
  \begin{displaymath}
  \begin{tikzcd}
    \Delta^1\times\catCh\op_{[m,m+1]}
      \ar[d, "\id\times \iota_{m}"'] \ar[r, "\alpha"] & \Sp \\
    \Delta^1\times\catCh\op \ar[ur, dashed, "\beta"'] &
  \end{tikzcd}
  \end{displaymath}
  where $\alpha$ is obtained from the defining square for the suspension
  \begin{displaymath}
  \begin{tikzcd}
    \bbS[m] \ar[r] \ar[d] & 0 \ar[d] \\
    0 \ar[r] & \bbS[m+1]
  \end{tikzcd}
  \end{displaymath}
  and extending this map $\Delta^1\times\Delta^1 \to \Sp$ to
  $\Delta^1\times\catCh\op$ in the only possible way that sends $(0,\pt)$
  and $(1,\pt)$ to the zero spectrum.
  As $\id\times d_m$ is fully faithful, the values of $\beta(0,\pt)$ and
  $\beta(1,\pt)$ are determined by $\alpha$, and are zero by construction.
  By Corollary \ref{cor-functors-from-Z}, this completes the construction of
  $\preaush$.
\end{constr}

\begin{defi}\label{def-aush-imp}
  Let\footnote{where $\aush$ stands for ``Aush\"ulen'', German for \emph{hull
  shelling}} $\aush\coloneqq\lvert - \rvert_{\preaush}^{\rm{st}}
  \colon\Fild(\Sp)\to\coCh(\Sp)$ denote the stable
  $\preaush$-realization
  functor associated, by the equivalence of Lemma \ref{lemma-bbz-coch},
  to the functor $\preaush$ given in
  Construction \ref{constr-preaush}.
  Let us denote by\footnote{where $\imp$ stands for ``Impilare'', Italian
  for \emph{piling up}} $\imp\coloneqq\nerve_{\preaush}^{\rm{st}}$ its right
  adjoint
  $$
  \begin{tikzcd}[column sep=huge]
    \Fild\Sp \ar[r, shift left=1.1ex, "\aush"] &
    \coCh(\Sp) \ar[l, shift left=1.1ex,"\imp"]
      \ar[l,phantom,"\text{\rotatebox{-90}{$\dashv$}}"]
  \end{tikzcd}
  $$
  (the reason for the hat in the notation will be clearified in Proposition
  \ref{prop-imp-is-complete}).
  We will refer to $\aush$ as the \emph{shelling functor} functor
  and to the $\imp$ as the \emph{piling-up functor}; we will use the terms
  \emph{associated shelled complex} and \emph{piled-up filtered object} for
  objects of the form $\aush F$ and $\imp C$, respectively.
\end{defi}

\begin{rem}\label{rem-explicit-imp}
  It follows from Proposition \ref{prop-mapping-nerve} that, for any coherent
  cochain complex $C$, the filtered object $\imp C$ is given by
  $$\imp C^{\istar} \simeq\map_{\coCh(\Sp)}(\spch{\istar}{\istar},C).$$
\end{rem}

Our next goal is to show that $\aush$ factors through the localization
$\Fild\Sp\to\cFild\Sp$.
In order to do so, it will be useful to identify its graded pieces
$\aush(-)^n$.

\begin{lemma}\label{lemma-aush-graded-same}
  There is a natural equivalence
  $$\aush(-)^n\simeq\gr^n(-)[n]$$
  of functors $\Fild(\Sp)\to\Sp$.
\end{lemma}
\begin{proof}
  By virtue of Lemma \ref{st-nerve-real}, it suffices to prove that both
  functors preserve colimits and agree on elements of the form
  $\susinftyp \Yo_{-}$.

  As both $\aush$ and $\ev_n$ admit right adjoints (see Proposition
  \ref{prop-eval-adj}), $\aush(-)^n\simeq\ev_n\circ \ \aush$ preserves
  colimits.
  By Definition \ref{defi-gr}, $\gr^n$ is given by $\cofib\circ(\iota_n)^*$; by
  \cite[1.1.1.8]{HA}, $\cofib$ admits a right adjoint, and since $\Sp$ is
  complete, $(\iota_n)^*$ admits a right adjoint, given by right Kan extension
  along $\iota_n$.

  It follows from Remark \ref{rem-stab-nerve-real} that for all $m\in\bbZ$
  $$\aush(\susinftyp \Yo_m)\simeq\lvert \Yo_m \rvert_{\preaush} \simeq
  \preaush(m)
  \simeq\spch m m$$
  (see Appendix \ref{appendix-first} for a detailed discussion)
  and thus that
  \begin{displaymath}
    \aush\left(\susinftyp \Yo_m\right)^n \simeq
    \begin{cases}
      \bbS[n] \quad &\text{if } m=n, \\
      0 \quad &\text{if } m\ne n,
    \end{cases}
  \end{displaymath}
  whereas a direct check shows that
  $$
  \gr^n\left(\susinftyp \Yo_m\right) \simeq
    \begin{cases}
      \bbS \quad &\text{if } m=n, \\
      0 \quad &\text{if } m\ne n,
    \end{cases}
  $$
  concluding the proof.
\end{proof}

\begin{prop}\label{prop-imp-is-complete}
  The adjunction $\aush\dashv\imp$ given in Definition \ref{def-aush-imp}
  factors as
  $$
  \begin{tikzcd}[column sep=huge]
    \Fild\Sp \ar[r, shift left=1.1ex, "L"] &
    \cFild\Sp \ar[r, shift left=1.1ex, "\caush"]
    \ar[l,phantom,"\text{\rotatebox{-90}{$\dashv$}}"]
    \ar[l, shift left=1.1ex] &
    \coCh(\Sp). \ar[l, shift left=1.1ex,"\imp"]
    \ar[l,phantom,"\text{\rotatebox{-90}{$\dashv$}}"]
  \end{tikzcd}
  $$
  through the localization of Proposition \ref{prop-cfild-bousfield}.

\end{prop}
\begin{proof}
  Let us first show that $\aush$ factors through the localization.
  By virtue of Proposition \ref{prop-cfild-bousfield},
  it is enough to show that it sends maps inducing equivalences on associated
  gradeds to equivalences. But, as per
  Lemma \ref{lemma-coch-to-prod-conservative},
  the functor
  $$\undch\colon\coCh(\Sp)\to\prod_{\bbZ}\Sp$$
  is conservative.
  It is thus enough to know that on each component, $\aush(-)^n$
  sends local maps to equivalences, which is true by virtue of Lemma
  \ref{lemma-aush-graded-same}.
  As $\aush$ preserves colimits, and $\cFild\Sp$ is a full
  subcategory of $\Fild\Sp$ closed under colimits (see Proposition
  \ref{prop-cfild-bousfield}), we get that $\caush$ preserves colimits, and
  hence admits a right adjoint.
  Since adjoints compose, this just means that $\imp$ takes values
  in the full subcategory of complete objects.
\end{proof}

Our next goal is to prove that the induced adjunction $\caush\dashv\imp$
is an equivalence of \cats. To this end, we need to prove a few key lemmata
before, and to get a better understanding of the functor $\imp$.

\begin{lemma}\label{lemma-caush-conservative}
  The functor $\caush$ is conservative.
\end{lemma}
\begin{proof}
  Let $\alpha\colon F\to G$ be a map in $\cFild\Sp$ such that
  $\caush(\alpha)\colon\caush F \to \caush G$
  is an equivalence. As $\coCh(\Sp)$ is a full subcategory of a functor
  category, equivalences are given pointwise, hence we have that, for all
  $n\in\bbZ$
  $$\caush F^n\simeq \caush G^n.$$
  But, by Lemma \ref{lemma-aush-graded-same}
  this is just means that $\alpha$ induces
  an equivalence on associated gradeds, hence, by
  Proposition \ref{prop-cfild-bousfield},
  it is an equivalence in $\cFild\Sp$.
\end{proof}

\begin{lemma}\label{lemma-left-is-right-shifted}
  There is a natural equivalence
  $(\ev_n)_!\simeq(\ev_{n-1})_*$
  (with notation as in Proposition \ref{prop-eval-adj}) of functors
  $\Sp\to\coCh(\Sp)$.
\end{lemma}
\begin{proof}
  It follows from the pointwise description of Proposition \ref{prop-eval-adj}
  that $(\ev_n)_*$ preserves all colimits.
  As $\Sp\simeq\Pst{\Delta^0}\simeq\St(\Pre{\Delta^0})$,
  Lemma \ref{st-nerve-real} implies that it is enough to check that they take
  the same value on $\susinftyp\Yo_{\pt}\simeq\susinftyp\pt\simeq\bbS$. But
  this follows immediately again from description given in Proposition
  \ref{prop-eval-adj}.
\end{proof}

\begin{lemma}\label{lemma-st-yo-is-y}
  We have equivalences
  $$\ptize \circ \susinftyp \Yo_n\simeq(\ev_n)_!\bbS\simeq(\ev_{n+1})_*\bbS$$
  of objects in $\coCh(\Sp)$
\end{lemma}
\begin{proof}
  The second equivalence is an instance of Lemma
  \ref{lemma-left-is-right-shifted}. Regarding the first one, we will show that
  they represent the same functor. In fact:
  \begin{displaymath}
  \begin{split}
    \Map_{\coCh(\Sp)}((\ev_n)_!\bbS,C)
    &\simeq\Map_{\Sp}(\bbS,C^n)\\
    &\simeq\ominfty C^n
  \end{split}
  \end{displaymath}
  and, using Proposition \ref{prop-st-yo-kinda}:
  \begin{displaymath}
  \begin{split}
    \Map_{\coCh(\Sp)}((\ptize\circ\susinftyp\Yo_n,C)
    &\simeq\Map_{\Pst\catCh}(\susinftyp\Yo_n,C)\\
    &\simeq\ominfty\map_{\Pst\catCh}(\susinftyp\Yo_n,C)\\
    &\simeq\ominfty C^n
  \end{split}
  \end{displaymath}
  concluding the proof.
\end{proof}

\begin{notat}\label{notat-evalfun}
  Motivated by Lemma \ref{lemma-st-yo-is-y}, we introduce the notation
  $$\evalfun_{-}\bbS\coloneqq\ptize\circ\susinftyp\Yo\colon\catCh\op\to\coCh(\Sp)$$
  to emphasize the connection with the pointed stabilization of the Yoneda
  embedding.
\end{notat}

We have the following ``density'' result for elements of the form
$\evalfun_n\bbS$ in $\coCh(\Sp)$.
\begin{lemma}\label{lemma-evalfun-dense}
  Given any $C\in\coCh(\Sp)$, there is a natural equivalence
  $$C\simeq\int^{n\in\catCh}\map_{\coCh(\Sp)}(\evalfun_n\bbS,C)
    \tensor\evalfun_n\bbS$$
  where $\tensor$ denotes the canonical tensoring of stable \cats over spectra
  (see \cite[4.8.2.20]{HA}).
\end{lemma}
\begin{proof}
  Using Proposition \ref{prop-st-yo-kinda}
  \begin{displaymath}
  \begin{split}
    \Map_{\coCh(\Sp)}
    &\left(\int^{n\in\catCh}\map_{\coCh(\Sp)}(\evalfun_n\bbS,C)
      \tensor\evalfun_n\bbS, D\right)\\
    &\simeq \int_{n\in\catCh}\Map_{\coCh(\Sp)}
    \left(\map_{\coCh(\Sp)}(\evalfun_n\bbS,C) \tensor\evalfun_n\bbS, D\right)\\
    &\simeq \int_{n\in\catCh}\Map_{\coCh(\Sp)}
    \left(\map_{\coCh(\Sp)}(\evalfun_n\bbS,C),
    \map_{\coCh(\Sp)}(\evalfun_n\bbS,D)\right)\\
    &\simeq \int_{n\in\catCh}\Map_{\coCh(\Sp)}
    \left(\map_{\Pst\catCh}(\susinftyp\Yo_n,C),
    \map_{\Pst\catCh}(\susinftyp\Yo_n,D)\right)\\
    &\simeq \int_{n\in\catCh}\Map_{\coCh(\Sp)}
    \left(C^n,D^n\right)\\
    &\simeq \int_{n\in\catCh}\Map_{\Pst\catCh}
    \left(C^n,D^n\right).
  \end{split}
  \end{displaymath}
  Now, using the end formula for the space of natural transformations
  (see \cite[2.3]{glasman2016spectrum}, \cite[5.1]{GHN17}), we get that
  \begin{displaymath}
  \begin{split}
    \int_{n\in\catCh}
    &\Map_{\Pst\catCh}\left(C^n,D^n\right)\\
    &\simeq \Map_{\Pst\catCh}(C,D)\\
    &\simeq \Map_{\coCh(\Sp)}(C,D)
  \end{split}
  \end{displaymath}
  hence both objects corepresent the same functor.
\end{proof}

\begin{lemma}\label{lemma-cof-seq-repr-coch}
  There is a co/fiber sequence in $\coCh(\Sp)$ given by
  $$\spch{n-1}{n-1} \to \spch n n \to \evalfun_{n-1} \bbS[n]$$
  where the first map is the structure map defined in Construction
  \ref{constr-preaush}.
\end{lemma}
\begin{proof}
  By Proposition \ref{prop-coch-pst-recoll}, the inclusion of $\coCh(\Sp)$ into
  $\Pst\catCh$ preserves colimits, therefore we can understand the colimit
  in the latter \cat, where it can be computed pointwise (see
  \cite[5.1.2.3]{HTT}).
  The result then follows by inspection of the following diagram
  $$
  \begin{tikzcd}
    \spch{n-1}{n-1} \ar[d] &[-25pt] = &[-25pt] \cdots \ar[r] & 0 \ar[r] \ar[d]
     &\bbS[n-1] \ar[r] \ar[d] &  0 \ar[r] \ar[d] &  0 \ar[r] \ar[d] & \cdots \\
    \spch n n \ar[d] &[-25pt] = &[-25pt] \cdots \ar[r] & 0 \ar[r] \ar[d] &
      0 \ar[r] \ar[d] &  \bbS[n] \ar[r] \ar[d, equal] &  0 \ar[r] \ar[d]
      & \cdots \\
    \evalfun_{n-1} \bbS[n] &[-25pt] = &[-25pt] \cdots \ar[r] & 0 \ar[r] &
      \bbS[n] \ar[r, equal] & \bbS[n] \ar[r] &  0 \ar[r] & \cdots .\\
  \end{tikzcd}
  $$
\end{proof}

The previous result lets us get an explicit description of the graded pieces
of $\imp$.

\begin{prop}\label{prop-gr-of-imp}
  Let $C\in\coCh(\Sp)$; then
  $$\gr^n(\imp C) \simeq C^n [-n].$$
\end{prop}
\begin{proof}
  By Lemma \ref{lemma-cof-seq-repr-coch}, we have a co/fiber sequence
  $$\spch{n}{n} \to \spch{n+1}{n+1} \to \evalfun_{n} \bbS[n+1].$$
  By applying $\map_{\coCh(\Sp)}(-,C)$ to it, we get
  $$\map_{\coCh(\Sp)}(\evalfun_{n}\bbS[n+1],C)\to\imp C(n+1)\to\imp C(n),$$
  (see Remark \ref{rem-explicit-imp})
  and, by Lemma \ref{lemma-st-yo-is-y},
  \begin{displaymath}
  \begin{split}
    \map_{\coCh(\Sp)}(\evalfun_{n}\bbS[n+1],C)
    &\simeq\map_{\coCh(\Sp)}(\ptize \circ \susinftyp \Yo_{n},C[-n-1]) \\
    &\simeq\map_{\Pst\catCh}(\susinftyp \Yo_{n},C[-n-1]) \\
    &\simeq C^{n}[-n-1]
  \end{split}
  \end{displaymath}
  (where the last equivalence follows from Proposition \ref{prop-st-yo-kinda}).
\end{proof}

\begin{cor}\label{cor-almost-all-about-imp}
  Let $C\in\coCh(\Sp)$. Then
  \begin{enumerate}
    \item If $C^m \simeq 0$, then $\imp C^m \simeq \imp C^{m+1}$.
      \label{item-trivone}
    \item If there exists an $N$ such that $C^m \simeq 0$ for all $m < N$,
      then $\imp C^m \simeq \imp C^N$ for all $m \leq N$.
      \label{item-trivtwo}
    \item If there exists an $N$ such that $C^m \simeq 0$ for all $m\geq N$,
      then $\imp C^m \simeq \imp C^N$ for all $m \geq N$.
      \label{item-trivthree}
    \item If there exists an $N$ such that $C^m \simeq 0$ for all $m\geq N$,
      then $\imp C^m \simeq 0$ for all $m \geq N$ and
      $\imp C^{N-1} \simeq C^{N-1}[-N+1]$.\label{item-nontriv}
    \item For all $n\in\bbZ$, we have
      that $\imp \spch n n \simeq \susinftyp\Yo_n$.\label{item-repres}
  \end{enumerate}
\end{cor}
\begin{proof}
  The first three points are straightforward consequences of Proposition
  \ref{prop-gr-of-imp}, whereas (\ref{item-nontriv}) follows immediately
  from (\ref{item-trivthree}), together with the completeness of $\imp C$
  (see Proposition \ref{prop-imp-is-complete}).
  Point (\ref{item-repres}) follows from the definition of
  $\spch n n$ together with (\ref{item-trivone}-\ref{item-nontriv}).
\end{proof}

\begin{prop}\label{prop-counit-almost-equiv}
  The composite $\aush \imp$ is equivalent to $\id_{\coCh(\Sp)}$.
\end{prop}
\begin{proof}
  It follows from Remark \ref{rem-explicit-imp} that $\imp$ commutes with
  arbitrary coproducts, and is thus cocontinuous. Hence, $\aush \imp$ is
  cocontinuous as well.
  By Lemma \ref{lemma-evalfun-dense}, we have that every coherent cochain
  complex is canonically the colimit of elements in the essential image of
  $\evalfun_{-}\bbS$.
  In particular, it is enough to check that the two functors agree on elements
  of the form $\evalfun_{n}\bbS$.
  Using Lemma \ref{lemma-cof-seq-repr-coch} and
  Corollary \ref{cor-almost-all-about-imp}.\ref{item-repres}, we have
  that
  \begin{displaymath}
  \begin{split}
    \aush\imp\evalfun_n\bbS
    &\simeq\cofib\left(\aush\imp\spch n n \to \aush\imp\spch{n+1}{n+1}\right)\\
    &\simeq\cofib\left(\aush\susinftyp\Yo_n \to \aush\susinftyp\Yo_{n+1}\right)\\
    &\simeq\cofib\left(\spch n n \to \spch{n+1}{n+1}\right)\\
    &\simeq\evalfun_n\bbS
  \end{split}
  \end{displaymath}
  concluding the proof.
\end{proof}

We can now prove the following theorem.

\begin{thm}\label{thm-cohch-fild-equiv}
  The adjunction
  $$
  \begin{tikzcd}[column sep=huge]
    \cFild\Sp \ar[r, shift left=1.1ex, "\caush"] &
    \coCh(\Sp). \ar[l, shift left=1.1ex,"\imp"]
      \ar[l,phantom,"\text{\rotatebox{-90}{$\dashv$}}"]
  \end{tikzcd}
  $$
  is an equivalence.
\end{thm}
\begin{proof}
  Putting together Proposition \ref{prop-counit-almost-equiv}, and Proposition
  \ref{prop-imp-is-complete}, we get that $\caush\imp\simeq\id_{\coCh(\Sp)}$.
  This, together with (the dual of) Proposition \ref{prop-eqv-induces-unit-eqv}
  implies that the counit
  $\eps\colon\caush\imp\Longrightarrow \id_{\coCh{\Sp}}$ of the adjunction
  $\caush\dashv\imp$ is an equivalence.
  If we now consider the triangular identity
  $$
  \begin{tikzcd}
    \caush \ar[r, Rightarrow, "\caush \eta"]
      \ar[dr, Rightarrow, "\id_{\caush}"'] &
      \caush\imp\caush
      \ar[d, Rightarrow, "\eps_{\caush}"] \\
    & \caush
  \end{tikzcd}
  $$
  we see that, as both $\id_\caush$ and $\eps_\caush$
  are equivalences, $\caush\eta$ is an equivalence as well;
  but, as by Lemma \ref{lemma-caush-conservative} $\caush$ is
  conservative, then $\eta$ has to be an
  equivalence, concluding the proof.
\end{proof}

We conclude the section giving more information about the functor $\imp$;
we can leverage from Lemma \ref{lemma-cof-seq-repr-coch} our
understanding of it, by getting
a recursive description of its components, which in turn gives a complete
description of its values in the case of bounded above cochain complexes.

In the rest of the section, we will make free use of some results about
cubic diagrams in stable \cats, as presented in \cite{DJW}, and refer to
\emph{op.\ cit.}\ for all the related concepts, notations and terminology we
use and do not introduce here. The following fact is somewhat implicit in
\cite{DJW}, but as we will use it crucially, we report it
here for convenience.

\begin{prop}\label{prop-totcof-itercof}
  Let $\scrC$ be a stable \cat, let $a\in\bbN$ and let
  $C:(\Delta^a)\op\to\scrC$ be a functor.
  Moreover, if $(\Delta^1)^a$ denotes the $a$-fold product of $\Delta^1$ with
  itself, and we denote by
  $\overrightharp{v} =
  (\overrightharp{v}_{a}, \overrightharp{v}_{a-1}, \ldots, \overrightharp{v}_1)$
  the objects of $(\Delta^1)^a$ (where each $\overrightharp{v}_i$ can either
  be $0$ or $1$),
  let $F$ be the $a$-cube having for vertices:
  $$
  F (\overrightharp{v}) =
  \begin{cases}
    C(a) \quad &\text{if } \overrightharp{v} = (0,\ldots,0)
      \\[10pt]
      C(a-b)
      \quad &\text{if } 0 < b \leq a
      \text{, } \overrightharp{v}_i = 0 \text{ for } i > b
      \text{ and } \overrightharp{v}_i = 1 \text{ for } i \leq b
      \\[10pt]
      0 &\text{else}
  \end{cases}
  $$
  where all the nonzero maps are determined by $C$
  (that is, $F$ is an $a$-cube having the $C(i)$'s on its
  ``spine'' and zero objects elsewhere).
  Then,
  $$
  \totcof F\simeq\cofib^a (C) \simeq
  \underbrace{
  \cofib \kern-.2em\Big(\kern-.4em\cdots
  \kern-.2em\cofib \kern-.2em\Big(
  \kern-.2em\cofib
  }_{a \text{ times}}
  \kern-.2em\Big(C(a) \to C(a-1) \Big)
  \to C(1) \Big) \cdots \to C(0) \Big).
  $$
\end{prop}
\begin{proof}
  By \cite[A.24]{DJW}, one can extend $F$ to a coCartesian $(a+1)$-cube
  $\widetilde F$
  such that $\widetilde{F}|_{\{0\}\times\Delta^a} \simeq F$,
  and with $\widetilde{F}|_{\{1\}\times\Delta^a}$
  having $\totcof F$
  as its terminal vertex, and $0$ elsewhere
  (see \emph{op.\ cit.}\ for a precise statement).
  As $\widetilde F$ is coCartesian, by iteratively applying \cite[A.11]{DJW},
  $\cofib^a(\widetilde{F})$ is a coCartesian $1$-cube, i.e.\ an equivalence.
  We conclude by observing that its source is given by $\cofib^a(F)$, and its
  target is just $\totcof F$.
\end{proof}

\begin{rem}\label{rem-recursive-imp}
  By virtue of Lemma \ref{rem-explicit-imp} and
  Lemma \ref{lemma-cof-seq-repr-coch}, we see that
  \begin{displaymath}
  \begin{split}
    \imp C^n
      &\simeq \map(\spch n n , C) \\
      &\simeq \map  \Big(\fib\big(\evalfun_{n-1}\bbS[n]\to\spch{n}{n-1}\big), C\Big) \\
      &\simeq \cofib\Big(\map\big(\spch{n}{n-1},C\big)\to
        \map\big(\evalfun_{n-1}\bbS[n],C\big)\Big) \\
      &\simeq \cofib\Big(\map\big(\spch{n-1}{n-1},C\big)[-1]\to C^{n-1}[-n]\Big) \\
      &\simeq \cofib\Big(\map\big(\spch{n-1}{n-1},C\big)\to C^{n-1}[-n+1]\Big)[-1] \\
      &\simeq \fib  \Big(\imp C^{n-1}\to C^{n-1}[-n+1]\Big) \\
  \end{split}
  \end{displaymath}
  hence, by iteratively applying the above, we get that, for all $a\in\bbN$,
  $\imp C^n$ is naturally equivalent to
  \begin{displaymath}
  \begin{split}
  \overbrace{
  \fib \kern-.2em\Big(\kern-.4em\cdots
  \kern-.2em\fib \kern-.2em\Big(
  \kern-.2em\fib
  }^{a \text{ times}}
  \kern-.2em\Big(&\imp C^{n-a} \to C^{n-a}[-n+a] \Big)\\
  &\to C^{n-a+1}[-n+a-1] \Big) \cdots \to C^{n-1}[-n+1] \Big).
  \end{split}
  \end{displaymath}
  By stability, the above is equivalent to
  \begin{displaymath}
  \begin{split}
  \overbrace{
  \cofib \kern-.2em\Big(\kern-.4em\cdots
  \kern-.2em\cofib \kern-.2em\Big(
  \kern-.2em\cofib
  }^{a \text{ times}}
  \kern-.2em\Big(&\imp C^{n-a} \to C^{n-a}[-n+a] \Big)[-1]\\
  &\to C^{n-a+1}[-n+a-1] \Big)[-1] \cdots \to C^{n-1}[-n+1] \Big)[-1]
  \end{split}
  \end{displaymath}
  which, in turn, is equivalent to
  \begin{displaymath}
  \begin{split}
  \overbrace{
  \cofib \kern-.2em\Big(\kern-.4em\cdots
  \kern-.2em\cofib \kern-.2em\Big(
  \kern-.2em\cofib
  }^{a \text{ times}}
  \kern-.2em\Big(&\imp C^{n-a}[-a] \to C^{n-a}[-n] \Big)\\
  &\to C^{n-a+1}[-n] \Big) \cdots \to C^{n-1}[-n] \Big).
  \end{split}
  \end{displaymath}
  By \ref{prop-totcof-itercof},
  the latter is equivalent to
  the total cofiber of a suitable cube $F_a$, that is
  $$\imp C^n \simeq \totcof F_a$$
  and hence, by \cite[A.31]{DJW},
  $$\imp C^n \simeq \left(\totfib F_a\right)[a].$$
  We can describe $F_a$ explicitly as follows;
  if $(\Delta^1)^a$ denotes the $a$-fold product of $\Delta^1$ with itself,
  and we denote by
  $\overrightharp{v} =
  (\overrightharp{v}_{a-1}, \overrightharp{v}_{a-2}, \ldots, \overrightharp{v}_0)$
  the objects of $(\Delta^1)^a$ (where each $\overrightharp{v}_i$ can either
  be $0$ or $1$)
  $F_a$ is the $a$-cube having for vertices:
  $$
  F_a (\overrightharp{v}) =
  \begin{cases}
    \imp C^{n-a}[-a] \quad &\text{if } \overrightharp{v} = (0,\ldots,0)
      \\[10pt]
    C^{n-a+b}[-n]
      \quad &\text{if } 0 \leq b < a
      \text{, } \overrightharp{v}_i = 0 \text{ for } i > b
      \text{ and } \overrightharp{v}_i = 1 \text{ for } i \leq b
      \\[10pt]
      0 &\text{else.}
  \end{cases}
  $$
  As an example, in the case $a=3$, the formula above looks as follows
  $$
  \imp C^n \simeq
  \totcof \left(
  \begin{tikzcd}
    \imp C^{n-3}[-3] \ar[rr]\ar[dd]\ar[dr]&& C^{n-3}[-n] \ar[dr] \ar[dd] & \\
    & 0 \ar[rr, crossing over] && C^{n-2}[-n] \ar[dd] \\
    0 \ar[rr]\ar[dr]&& 0 \ar[dr]& \\
    & 0 \ar[rr] \ar[from=uu, crossing over]&& C^{n-1}[-n] \\
  \end{tikzcd}
  \right).
  $$
\end{rem}

The previous remark takes a particularly pleasant form when the complex is
bounded.

\begin{cor}\label{cor-imp-closed-formula-for-bounded}
  Let $C\in\coCh(\Sp)$, and assume there exists an $N$ such that
  $C^m \simeq 0$ for all $m>N$; then for all $a>0$
  \begin{displaymath}
  \begin{split}
  \imp C^{N-a} \simeq
  \overbrace{
  \cofib \kern-.2em\Big(\kern-.4em\cdots
  \kern-.2em\cofib \kern-.2em\Big(
  \kern-.2em\cofib
  }^{a \text{ times}}
  \kern-.2em\Big( &C^{N-a}[-N] \to C^{N-a+1}[-N] \Big)\\
  &\to C^{N-a+2}[-N] \Big) \cdots \to C^{N}[-N] \Big).
  \end{split}
  \end{displaymath}
\end{cor}
\begin{proof}
  It follows from Remark \ref{rem-recursive-imp}, Corollary
  \ref{cor-almost-all-about-imp}.\ref{item-nontriv}
  and \cite[Proposition A.24]{DJW} that the $(a+1)$-cube
  defined by
  $$
  F (\overrightharp{v}) =
  \begin{cases}
    \imp C^{N-a}[-a] \quad &\text{if } \overrightharp{v} = (0,\ldots,0)
      \\[10pt]
    C^{N-a+b}[-N]
      \quad &\text{if } 0 \leq b < a
      \text{, } \overrightharp{v}_i = 0 \text{ for } i > b
      \text{ and } \overrightharp{v}_i = 1 \text{ for } i \leq b
      \\[10pt]
    \imp C^N \simeq C^N[-N]
      \quad &\text{if } \overrightharp{v} = (1,\ldots,1)
      \\[10pt]
    0 &\text{else}
  \end{cases}
  $$
  (where $\overrightharp{v} =
  (\overrightharp{v}_{a}, \overrightharp{v}_{a-1}, \ldots,
  \overrightharp{v}_0)$)
  is coCartesian.
  In particular, again by \cite[Proposition A.24]{DJW},
  $\imp C^{N-a}[-a]$ is the total fiber of $F|_{\Delta^a \times \{1\}}$;
  the thesis follows from \cite[Proposition A.31]{DJW}.
\end{proof}

\begin{ex}
  In the cases $a=2$ and $a=3$, the previous Corollary specializes to
  $$
  \imp C^{N-2} \simeq \totcof \left(
  \begin{tikzcd}
    C^{N-2} \ar[r] \ar[d] & C^{N-1} \ar[d] \\
    0 \ar[r] & C^{N} \\
  \end{tikzcd}
\right)[-N]
  $$
  and
  $$
  \imp C^{N-3} \simeq \totcof \left(
  \begin{tikzcd}
    C^{N-3} \ar[rr]\ar[dd]\ar[dr]&& C^{N-2} \ar[dr] \ar[dd] & \\
    & 0 \ar[rr, crossing over] && C^{N-1} \ar[dd] \\
    0 \ar[rr]\ar[dr]&& 0 \ar[dr]& \\
    & 0 \ar[rr] \ar[from=uu, crossing over]&& C^{N}\\
  \end{tikzcd}
\right)[-N]
  $$
  respectively.
\end{ex}

The following remark, inspired by the ``Gap objects'' considered
in \cite[1.2.2]{HA}, will be useful later, and sheds some light on the
relation between coherent cochain complexes and the objects called
$\bbZ$-complexes in \emph{op.\ cit.}

\begin{rem}\label{rem-intermediate-subquotients}
  Proposition \ref{prop-gr-of-imp} and the calculus of total cofibers of cubic
  diagrams allow us to understand also all the intermediate subquotients of
  $\imp C$.
  We can consider the diagram consisting of co/Cartesian squares
  $$
  \begin{tikzcd}[column sep = 0.7em]
    \imp C^n \ar[r] \ar[d]
            & \imp C^{n-1} \ar[r] \ar[d]
            & \imp C^{n-2} \ar[r] \ar[d]
            & \imp C^{n-3} \ar[r] \ar[d]
            & \imp C^{n-4} \ar[d] \ar[r] & \cdots\\
    0 \ar[r]
            & C^{n-1}[-n+1] \ar[r] \ar[d]
            & \imp C^{n-2}/\imp C^{n} \ar[r] \ar[d]
            & \imp C^{n-3}/\imp C^{n} \ar[r] \ar[d]
            & \imp C^{n-4}/\imp C^{n} \ar[d] \ar[r] & \cdots\\
    & 0 \ar[r]
            & C^{n-2}[-n+2] \ar[r] \ar[d]
            & \imp C^{n-3}/\imp C^{n-1} \ar[r] \ar[d]
            & \imp C^{n-4}/\imp C^{n-1} \ar[d] \ar[r] & \cdots\\
    && 0 \ar[r]
            & C^{n-3}[-n+3] \ar[r]\ar[d]
            & \imp C^{n-4}/\imp C^{n-2} \ar[r]\ar[d] & \cdots\\
    &&&\vdots & \vdots
  \end{tikzcd}
  $$
  from which we can deduce that, for any $n\in\bbZ$
  $$\imp C^{n-2}/\imp C^{n}\simeq \cofib(\partial^{n-2})[-n+1];$$
  moreover, we can proceed inductively and apply
  \cite[A.26]{DJW} to the cofiber sequences
  $$\imp C^{n-k+1}/\imp C^n \to \imp C^{n-k}/\imp C^n \to C^{n-k}[-n+k]$$
  to identify
  $\imp C^{n-k}/\imp C^n$ with the total cofiber of a cube having
  the truncation
  $$\left(C^{n-k}\to C^{n-k+1}\to \cdots \to C^{n-1}\right)[-n+1]$$
  of $C[-n+1]$ on its ``spine'' and zeroes elsewhere; thus, by Proposition
  \ref{prop-totcof-itercof} we have
  \begin{displaymath}
  \begin{split}
  \imp C^{n-k}/\imp C^n \simeq
  \overbrace{
  \cofib \kern-.2em\Big(\kern-.4em\cdots
  \kern-.2em\cofib \kern-.2em\Big(
  \kern-.2em\cofib
  }^{k-1 \text{ times}}
  \kern-.2em\Big( &C^{n-k} \to C^{n-k+1} \Big)\\
  &\to C^{n-k+2} \Big) \cdots \to C^{n-1} \Big)[-n+1]
  \end{split}
  \end{displaymath}
  for any $k\geq 2$.
\end{rem}

\section{The general equivalence}\label{section-general-equiv}
In this section, we extend the result of Theorem \ref{thm-cohch-fild-equiv}
to all stable \cats with sequential limits. Along the way, we prove that
the explicit formulas given for $\caush$ and $\imp$ in the previous section
hold also in the general case.

\begin{lemma}\label{lemma-huge-presheaves}
  Let $\scrC$ be a small stable \cat. Then, the following
  equivalences hold:
  $$\Fun\left(\scrC,\coCh\left(\Sp\right)\right)\simeq
  \coCh\left(\Fun\left(\scrC,\Sp\right)\right)$$
  and
  $$\Fun\left(\scrC,\cFild\left(\Sp\right)\right)
  \simeq\cFild\left(\Fun\left(\scrC,\Sp\right)\right).$$
\end{lemma}
\begin{proof}
  We prove this for $\coCh\left(\Fun\left(\scrC,\Sp\right)\right)$, the
  case of
  $\cFild\left(\Fun\left(\scrC,\Sp\right)\right)$ being entirely
  analogous.
  We have that, by definition of $\coCh\left(\Sp\right)$,
  \begin{equation}\label{eq-funcats}
    \Fun\left(\scrC,\coCh\left(\Sp\right)\right)
    \simeq\Fun\left(\scrC,\Fun^0\left(\catCh\op,\Sp\right)\right).
  \end{equation}
  Now, as
  $$\Fun\left(\scrC,\Fun\left(\catCh\op,\Sp\right)\right)
    \simeq\Fun\left(\scrC\times\catCh\op,\Sp\right)$$
  we have that the \cats of (\ref{eq-funcats})
  are equivalent to the full subcategory
  of $\Fun(\scrC\times\catCh\op,\Sp)$ spanned by functors that are pointed in
  the second variable (that is, sending any pair
  $(C,0)\in\scrC\times\catCh\op$ to $0$). Thus, as
  $$\Fun\left(\scrC\times\catCh\op,\Sp\right)
    \simeq\Fun\left(\catCh\op,\Fun\left(\scrC,\Sp\right)\right)$$
  the \cats of (\ref{eq-funcats}) are in turn equivalent to
  the full subcategory
  $$\Fun^0\left(\catCh\op,\Fun\left(\scrC,\Sp\right)\right)
    \subset\Fun\left(\catCh\op,\Fun\left(\scrC,\Sp\right)\right)$$
  spanned by
  pointed functors, which, by definition, is $\coCh\left(\Fun\left(\scrC,
  \Sp\right)\right)$.
\end{proof}

\begin{cor}\label{cor-eq-for-psh}
  Let $\scrC$ be a small stable \cat. Then, the equivalence of
  Theorem \ref{thm-cohch-fild-equiv} extends to an equivalence between
  $\coCh(\Fun(\scrC,\Sp))$ and $\cFild(\Fun(\scrC,\Sp))$
  which we will again denote by
  $$
  \begin{tikzcd}[column sep=huge]
    \cFild(\Fun(\scrC,\Sp)) \ar[r, shift left=1.1ex, "\caush"] &
    \coCh(\Fun(\scrC,\Sp)). \ar[l, shift left=1.1ex,"\imp"]
      \ar[l,phantom,"\sim" description]
  \end{tikzcd}
  $$
\end{cor}

\begin{rem}\label{rem-rephrase}
  One way to re-phrase Lemma \ref{lemma-huge-presheaves} and Corollary
  \ref{cor-eq-for-psh} is to say that objects in
  $\coCh(\Fun(\scrC,\Sp))$ can be described as bifunctors
  $(C,n)\mapsto G_C(n)$ such that $G_C(\pt)\simeq 0$ for any $C\in\scrC$,
  and similarly
  objects in $\cFild(\Fun(\scrC,\Sp))$ can be described as bifunctors
  $(C,n)\mapsto F_C(n)$ such that $\lim_n F_C(n) \simeq 0$ for any $C\in\scrC$.
  The equivalence of Corollary
  \ref{cor-eq-for-psh}
  is then given on objects by pointwise (in $\scrC$) postcomposition with
  $\caush$ and $\imp$; that is, we have
  $$\caush\Big((C,n)\mapsto F_C^n\Big) \simeq \Big((C,n) \mapsto
  \caush\big(F_C^\bullet\big)^n\Big)$$
  and
  $$\imp\Big((C,n)\mapsto G_C^n\Big) \simeq \Big((C,n) \mapsto
  \imp\big(G_C^\bullet\big)^n\Big).$$
\end{rem}

\begin{rem}\label{rem-ess-images}
  Let $\scrC$ be a small stable \cat, and let $\scrA\subseteq\Fun(\scrC,\Sp)$
  be a full stable subcategory closed under sequential limits. Postcomposition
  with the inclusion $\iota\colon\scrA\to\Fun(\scrC,\Sp)$ induces fully
  faithful functors
  $$\iota_{\mathrm{Fil}}\colon\cFild(\scrA)\to\cFild(\Fun(\scrC,\Sp))$$
  and
  $$\iota_{\mathrm{Ch}}\colon\coCh(\scrA)\to\coCh(\Fun(\scrC,\Sp)).$$
  By inspection, an object $F$ lies in the essential image of
  $\iota_{\mathrm{Fil}}$ if and only if $F^n$ lies in $\scrA$ for all
  integers $n$, and an object $C$ lies in the essential image
  of $\iota_{\mathrm{Ch}}$ if and only if $C^n$ lies in $\scrA$ for all
  integers $n$.
\end{rem}

\begin{lemma}\label{lemma-factorization-full-subcats}
  Let $\scrA\subseteq\Fun(\scrC,\Sp)$ be a full, stable subcategory closed
  under sequential limits.
  \begin{enumerate}
    \item If $F$ is a coherent cochain complex in $\cFild(\Fun(\scrC,\Sp))$ that
      lies in the essential image of $\cFild(\scrA)$, then $\caush F$
      lies in the essential image of $\coCh(\scrA)$.
      \label{l-e-nptwisea}
    \item If $C$ is a coherent cochain complex in $\coCh(\Fun(\scrC,\Sp))$ that
      lies in the essential image of $\coCh(\scrA)$, then $\imp C$
      lies in the essential image of $\cFild(\scrA)$.
      \label{l-e-nptwisei}
  \end{enumerate}
\end{lemma}
\begin{proof}
  (\ref{l-e-nptwisea}):
  Let $F\colon (C,n)\mapsto F_C^n$ lie in (the essential image of)
  $\cFild(\scrA)$.
  By Lemma \ref{lemma-aush-graded-same} together with Remark
  \ref{rem-rephrase}, we have that, for any integer $m$,
  $(\caush F)^m$ is given by $\cofib\left(F_\bullet^{m+1}\to F_\bullet^m\right)$.
  By Remark \ref{rem-ess-images}, each $F_\bullet^n$ lies in $\scrA$,
  hence, as $\scrA$ is a stable subcategory each $(\caush A)^m$ lies in
  $\scrA$, and thus $\caush F$ lies in the essential image of $\coCh(\scrA)$.

  (\ref{l-e-nptwisei}):
  Let $G\colon (C,n)\mapsto G_C^n$ lie in (the essential image of)
  $\coCh(\scrA)$.
  By Lemma \ref{lemma-coch-limit-of-truncations}, $G$ can be expressed as the
  limit of bounded above complexes:
  $$
  G\simeq\lim_{j\in\bbZ\op}G^{\leq j};
  $$
  notice that, by Remark \ref{rem-ess-images} and by the definition of
  the $G^{\leq j}$'s given in Construction \ref{constr-trunc-ch}, all
  the objects of the form $(G^{\leq j})^k_\bullet$ belong to $\scrA$.
  By Remark \ref{rem-recursive-imp} together with Remark \ref{rem-rephrase},
  we have that, for any choice of integers $j$ and $m$,
  $(\imp G^{\leq j})^m_\bullet$ is either zero or can be expressed as a
  suitable total fiber for a finite diagram $K^{j,m}$ whose entries are either
  zeroes or of the form $(G^{\leq j})^k_\bullet[j]$ for different values of $k$;
  in particular, $G^{\leq j}$ is a finite limit of a diagram with values in
  $\scrA$.
  Putting everything together (and keeping in mind that by Proposition
  \ref{prop-eval-adj} evaluation commutes with limits),
  we have that
  \begin{displaymath}
  \begin{split}
    (\imp G)^n &\simeq \left(\imp(\lim_j G^{\leq j})\right)^n\\
               &\simeq \left(\lim_j (\imp G^{\leq j})\right)^n \\
               &\simeq \lim_j (\imp G^{\leq j})^n \\
               &\simeq \lim_j (\totcof K^{j,n})
  \end{split}
  \end{displaymath}
  is given by a sequential limit of finite colimits of elements of $\scrA$,
  and thus lies in $\scrA$ for any given $n$. This in turn proves that
  $\imp G$ lies in the essential image of $\cFild(\scrA)$.
\end{proof}

\begin{lemma}\label{lemma-factored-equivalence}
  Let $\scrA\subseteq\Fun(\scrC,\Sp)$ be a full, stable subcategory closed
  under sequential limits. Then the equivalence between
  $\cFild(\Fun(\scrC,\Sp))$ and $\coCh(\Fun(\scrC,\Sp))$ restricts to an
  equivalence
  $$
    \cFild(\scrA) \simeq \coCh(\scrA).
  $$
\end{lemma}
\begin{proof}
  It follows from Lemma \ref{lemma-factorization-full-subcats} that (in the
  notation of Remark \ref{rem-ess-images}) $\caush\circ\iota_{\mathrm{Fil}}$
  factors through $\iota_{\mathrm{Ch}}$ and $\imp\circ\iota_{\mathrm{Ch}}$
  factors thorugh $\iota_{\mathrm{Fil}}$, i.e.\ there exist functors (which
  we'll temporarily denote $A$ and $I$) such that
  $\caush\circ\ifil\simeq\ich\circ A$ and $\imp\circ\ich\simeq\ifil\circ I$.
  In particular, as
  \begin{displaymath}
  \begin{split}
    \ich &\simeq \caush\circ\imp\circ\ich\\
         &\simeq \caush\circ\ifil \circ I\\
         &\simeq \ich \circ A \circ I
  \end{split}
  \quad
  \text{ and }
  \quad
  \begin{split}
    \ifil &\simeq \imp\circ\caush\circ\ifil\\
          &\simeq \imp\circ\ich\circ A\\
          &\simeq \ifil\circ I\circ A
  \end{split}
  \end{displaymath}
  $A$ and $I$ are mutually inverse.
\end{proof}

\begin{thm}\label{thm-general-equivalence}
  Let $\scrC$ be a stable \cat having sequential limits. Then there exists
  an equivalence of stable \cats
  $$
  \begin{tikzcd}[column sep=huge]
    \cFild(\scrC) \ar[r, shift left=1.1ex, "\caush"] &
    \coCh(\scrC). \ar[l, shift left=1.1ex,"\imp"]
      \ar[l,phantom,"\sim" description]
  \end{tikzcd}
  $$
\end{thm}
\begin{proof}
  Let $U_0$ be our universe of small sets, and let $U_1$ denote a
  universe containing $U_0$ as an element. Let us denote by $\Sp_{U_1}$ the
  stabilization of the \cat of $U_1$-small spaces. The stable Yoneda embedding
  (see Definition \ref{defi-st-yo}) provides a fully faithful
  (see \cite[Section 6]{NikStableOperads}) exact functor
  $$\stYo\colon\scrC\to\Fun\left(\scrC\op,\Sp_{U_1}\right).$$
  The results now follows from Lemma \ref{lemma-factored-equivalence} applied
  to $\stYo$.
\end{proof}

\begin{rem}\label{rem-general-adjunction}
  In particular, by composing the equivalence of Theorem
  \ref{thm-general-equivalence} with the adjunction $L\dashv i$ of Proposition
  \ref{prop-cfild-bousfield}, we get an induced adjunction
  $$
  \begin{tikzcd}[column sep=huge]
    \Fild\scrC \ar[r, shift left=1.1ex, "\aush"] &
    \coCh\scrC \ar[l, shift left=1.1ex,"\imp"]
      \ar[l,phantom,"\text{\rotatebox{-90}{$\dashv$}}"]
  \end{tikzcd}
  $$
  for any stable \cat $\scrC$ with sequential limits.
\end{rem}

\begin{rem}\label{rem-explicit-description-general-case}
  In the proof of Lemma \ref{lemma-factorization-full-subcats} we also showed
  that the pointwise descriptions already given for $\caush$ and $\imp$ in
  the case $\scrC=\Sp$ hold also in the general case of
  Theorem \ref{thm-general-equivalence};
  that is, for $F\in\cFild\scrC$
  $$\caush F^n\simeq\gr^n F[n]$$
  and, for $C\in\coCh\scrC$,
  $$\imp C^n \simeq \totfib F_a [a]$$
  for a suitable cube $F_a$, given explicitly by
  $$
  F_a (\overrightharp{v}) =
  \begin{cases}
    \imp C(n-a)[-a] \quad &\text{if } \overrightharp{v} = (0,\ldots,0)
      \\[10pt]
    C^{n-a+b}[-n]
      \quad &\text{if } 0 \leq b < a
      \text{, } \overrightharp{v}_i = 0 \text{ for } i > b
      \text{ and } \overrightharp{v}_i = 1 \text{ for } i \leq b
      \\[10pt]
      0 &\text{else.}
  \end{cases}
  $$
  (see \ref{rem-explicit-imp} for details about the notation).
  In particular, Remark \ref{rem-intermediate-subquotients} generalizes as well,
  giving
  \begin{displaymath}
  \begin{split}
  \imp C^{n-k}/\imp C^n \simeq
  \overbrace{
  \cofib \kern-.2em\Big(\kern-.4em\cdots
  \kern-.2em\cofib \kern-.2em\Big(
  \kern-.2em\cofib
  }^{k-1 \text{ times}}
  \kern-.2em\Big( &C^{n-k} \to C^{n-k+1} \Big)\\
  &\to C^{n-k+2} \Big) \cdots \to C^{n-1} \Big)[-n+1]
  \end{split}
  \end{displaymath}
  for any $n\in\bbZ$ and $k\geq 2$.
\end{rem}

\section{Coherent cochain complexes and Beilinson $\mathrm{t}$-structures}
\label{section-t-structures}

In this section, we study the connection between the
pointwise t-structure on coherent cochain complexes
and the Beilinson t-structure on filtered objects and show how the
former can in some sense be interpreted as an easier to understand version of
the latter.

In fact, if $\Fild\scrC$ is equipped with the Beilinson t-structure,
then the full subcategory of complete filtered objects inherits a t-structure
from it\footnote{this fact can easily be proved directly, but will be an immediate
consequence of Theorem \ref{thm-general-equivalence}}.
It turns out that such inherited t-structure
is equivalent to the one obtained by carrying the pointwise t-structure
on $\coCh\scrC$ along the equivalence of Theorem \ref{thm-general-equivalence};
moreover, the Beilinson t-structure on $\Fild\scrC$
is in some sense characterized by this property and by carrying ``trivial
information'' on essentially constant objects
(see Theorem \ref{thm-beil-recoll} for a precise statement).
In particular, $\Fild\scrC$ and $\cFild\scrC$ have the same heart (Remark
\ref{rem-recollement-of-hearts}).

\begin{defi}\label{defi-coch-belinison-t-str}
  Let $\scrC$ be a stable \cat equipped with a t-structure. We define
  the \emph{pointwise t-structure} on $\coCh\scrC$ to be the one defined
  by
  \begin{displaymath}
  \begin{split}
    (\coCh\scrC)_{\geq 0} &= \{C\in\coCh\scrC \ | \ \forall n \ C^n \in
      \scrC_{\geq 0}\}\\
    (\coCh\scrC)_{\leq 0} &= \{C\in\coCh\scrC \ | \ \forall n \ C^n \in
      \scrC_{\leq 0}\}.
  \end{split}
  \end{displaymath}
  We will denote the truncation functors for this t-structure by
  $\tau^\text{lvl}_{\geq n}$ and $\tau^\text{lvl}_{\leq n}$, and the
  homotopy objects by $\pi^\text{lvl}_n$ for all $n\in\bbZ$.
\end{defi}

\begin{rem}\label{rem-coch-same-completeness}
  It follows immediately from the definitions that $\coCh\scrC$ has precisely
  the same separatedness and completeness properties that $\scrC$ has.
\end{rem}

\begin{defi}\label{defi-transferred}
  Let $\scrC$ be a stable \cat with sequential limits.
  The \emph{transferred t-structure} on $\cFild\scrC$ is the t-structure
  $(\imp(\coCh\scrC)_{\geq 0},\imp(\coCh\scrC)_{\leq 0})$ transferred
  along the equivalence of Theorem \ref{thm-general-equivalence}.
\end{defi}

\begin{rem}\label{rem-heart-complete}
  As a direct consequence of the definitions, we have that
  $$\left(\cFild\scrC\right)^\heartsuit\simeq(\coCh\scrC)^\heartsuit
  \simeq\coCh\left(\scrC^\heartsuit\right).$$
\end{rem}

The following definition is a straightforward generalization of the
t-structure first introduced by Beilinson in \cite{Bei87}.

\begin{defi}\label{defi-fild-beilinson-t-str}
  Let $\scrC$ be a stable \cat having all sequential limits\footnote{this
  assumption is likely superfluous; see Remark \ref{rem-seq-lims}},
  equipped with a right separated\footnote{unlike the previous one, this
  assumption is crucial; see Remark \ref{rem-right-assumption}}
  t-structure $(\scrC_{\geq 0}, \scrC_{\leq 0})$.

  The \emph{Beilinson t-structure} $(\Fild_{\geq 0}\scrC, \Fild_{\leq 0}\scrC)$
  on $\Fild\scrC$ is defined as follows:
  \begin{enumerate}
    \item[$\bullet$] $\Fild_{\geq 0}\scrC$ is the full subcategory spanned by
      the objects $F\in\Fild\scrC$ such that
      $$\gr^i(F)\in \scrC_{\geq -i} \text{ for all } i.$$
    \item[$\bullet$] $\Fild_{\leq 0}\scrC$ is the full subcategory spanned by
      the objects $F\in\Fild\scrC$ such that
      $$F^i\in \scrC_{\leq -i} \text{ for all } i.$$
  \end{enumerate}
  (Note the asymmetry in the definition). This t-structure appeared first,
  in a slightly different setting, in \cite{Bei87}.
  Its existence in this generality will be a consequence
  of \cite[Theorem 2.19]{FLRecoll} together with Theorem \ref{thm-beil-recoll}.
  We will denote the truncation functors for this t-structure by
  $\tau^B_{\geq n}$ and $\tau^B_{\leq n}$, and the
  homotopy objects by $\pi^B_n$ for all $n\in\bbZ$.
\end{defi}

\begin{rem}\label{rem-seq-lims}
  The hypothesis of $\scrC$ having
  all sequential limits in Definition \ref{defi-fild-beilinson-t-str} is
  there just because we will use Theorem \ref{thm-general-equivalence} to prove
  its existence. We believe that the t-structure can exist even
  without this assumption on $\scrC$, but as all the examples that arise in
  practice satisfy this extra hypothesis, we didn't bother finding a proof that
  does not use it.
\end{rem}

\begin{rem}
  We can infer a few properties of the Beilinson t-structure from its
  definition:
  \begin{enumerate}
    \item Even if we are assuming $\scrC$ to be right separated, $\Fild\scrC$
          need not be so; in fact, the full subcategory of
          $\infty$-coconnective objects $\cap_n (\Fild\scrC)_{\leq n}$ consists
          of all filtered objects whose associated graded is trivial, hence of
          all the essentially constant objects.
    \item Since the full subcategory of $\infty$-connective objects consists of
          the levelwise $\infty$-connective ones, $\Fild\scrC$ is left separated
          if and only if $\scrC$ is so.
  \end{enumerate}
\end{rem}

We learned the following fact from \cite[5.4]{BMS2}; although the result in
\emph{loc.\ cit.}\ is stated in less generality, the proof carries verbatim
to the general case. We report the argument here for the reader's convenience.

\begin{prop}\label{prop-gr-and-trunc}
  Let $\scrC$ be a stable \cat having all sequential limits,
  equipped with a right separated t-structure
  $(\scrC_{\geq 0}, \scrC_{\leq 0})$, and let $\wtrunc{n}$ denote its
  Whitehead truncation functors.
  Let $\Fild(\scrC)$ be equipped with the Beilinson t-structure, and
  let $\wtrunc{n}^B$ denote its Whitehead truncation functors.
  Then, there is a natural equivalence of functors $\Fild\scrC\to\scrC$
  $$
  \gr^i\circ\wtrunc{n}^B \simeq \wtrunc{n-i}\circ\gr^i
  $$
  for all $i\in\bbZ$.
\end{prop}
\begin{proof}
  Notice that for any $i$, the exact functor $\gr^i\colon\Fild\scrC\to\scrC$
  carries $(\Fild\scrC)_{\geq 0}$ to $\scrC_{\geq -i}$. Moreover, as by
  \cite[1.2.1.16]{HA} each $\scrC_{\geq -i}$ is closed under extensions,
  the fiber sequence
  $$F^i\to\gr^i F\to F^{i+1}[1]$$
  proves that $\gr^i F\in\scrC_{\leq -i}$, and thus $\gr^i$ carries
  also $(\Fild\scrC)_{\leq 0}$ to $\scrC_{\leq -i}$. That is, each $\gr^i$
  is \emph{t-exact} with respect to the Beilinson t-structure on $\Fild\scrC$
  and the shifted t-structure $(\scrC_{\geq -i},\scrC_{\leq -i})$ on $\scrC$.
  As any exact and t-exact functor between stable \cats equipped with
  t-structures commutes with the truncation functors associated to the
  t-structures, the result follows.
\end{proof}

\begin{cor}\label{cor-gr-and-pi}
  In particular, in the hypotheses of Proposition \ref{prop-gr-and-trunc},
  if we denote by $\pi_n$ the t-structure homotopy object functors of $\scrC$
  and by $\pi_n^B$ the t-structure homotopy object functors of $\Fild\scrC$,
  we have the equivalence (natural in $F$)
  $$
  \gr^i\pi_n^B F \simeq (\pi_{n-i}(\gr^i F))[-i]
  $$
  for all $i,n\in\bbZ$.
\end{cor}

We recall the following theorem from \cite{FLRecoll} (which is an
$\infty$-categorical generalization of \cite[1.4.10]{BBD83}).
\begin{thm}\cite[Theorem 2.19]{FLRecoll}\label{thm-glued-t-str}
  Given a recollement $\scrC_0\hookrightarrow\scrC\hookleftarrow\scrC_1$,
  suppose both $\scrC_0$ and $\scrC_1$ are equipped with t-structures, then
  there exists a t-structure on $\scrC$, called the \emph{glued t-structure},
  given by
  \begin{displaymath}
  \begin{split}
    \scrC_{\geq 0} &= \{X\in\scrC \ | \ j_L X \in (\scrC_1)_{\geq 0}
    \text{ and } i_L X \in (\scrC_0)_{\geq 0}\}\\
    \scrC_{\leq 0} &= \{X\in\scrC \ | \ j_L X \in (\scrC_1)_{\leq 0}
    \text{ and } i_R X \in (\scrC_0)_{\leq 0}\}
  \end{split}
  \end{displaymath}
  (with notations as in Remark \ref{rem-full-recol}).
\end{thm}

\begin{thm}\label{thm-beil-recoll}
  Let $\scrC$ be a stable \cat with all sequential limits equipped with a
  right separated t-structure $(\scrC_{\geq 0},\scrC_{\leq 0})$.
  Then the glued t-structure on $\Fild\scrC$
  (via the recollement of Remark \ref{rem-full-recol})
  obtained by considering
  \begin{enumerate}
    \item the trivial t-structure\footnote{that
      is, the one given by $(\scrC,\{0\})$, where all objects are connective,
      and only the zero object is coconnective}
      on $\scrC$,
    \item the transferred t-structure on $\cFild\scrC$
  \end{enumerate}
  is the Beilinson t-structure (with respect to
  $(\scrC_{\geq 0},\scrC_{\leq 0})$) on $\Fild\scrC$, as per Definition
  \ref{defi-fild-beilinson-t-str}.
\end{thm}

Before proving Theorem \ref{thm-beil-recoll}, we state one immediate
consequence of it (and the definition of glued t-structure given in Theorem
\ref{thm-glued-t-str}).

\begin{cor}\label{cor-transferred}
  The Beilinson t-structure on $\Fild\scrC$ induces a t-structure on the
  full subcategory of complete filtered objects $\cFild\scrC$ that is
  equivalent to the transferred t-structure of Definition
  \ref{defi-transferred}.
\end{cor}

\begin{proof}[Proof of Theorem \ref{thm-beil-recoll}]
  Let us start by identifying the subcategory of connective objects for
  the transferred t-structure on $\cFild\scrC$.
  We have that $F\in(\cFild\scrC)_{\geq 0}$ if and only if
  $\caush F \in \left(\coCh\scrC\right)_{\geq 0}$.
  By Definition \ref{defi-coch-belinison-t-str},
  this is the case if and only if
  $$
  (\caush F)^n \simeq \gr^n F[n] \in \scrC_{\geq 0}
  \text{ for all } n
  $$
  hence
  $$
  (\cFild\scrC)_{\geq 0} = \left\{F\in\cFild\scrC \ | \ \forall n \ \gr^n F \in
  \scrC_{\geq -n}\right\}.
  $$
  Now, according to Theorem \ref{thm-glued-t-str}, the connective
  objects in the glued t-structure on $\Fild\scrC$ are given by all the
  $G\in\Fild\scrC$ such that
  $$LG\in\left(\cFild\scrC\right)_{\geq 0}$$
  (the condition on $G^{-\infty}$ being empty, as we are considering the
  trivial t-structure on $\scrC$);
  but, as for all $n$ we have $\gr^n LG \simeq \gr^n G$, the above is
  equivalent to the condition
  $$
  \gr^n G \in \scrC_{\geq -n} \text{ for all } n
  $$
  which in turn determines exactly the class of connective objects for the
  Beilinson t-structure of Definition \ref{defi-fild-beilinson-t-str};
  as, by \cite[1.2.1.3]{HA}, the class of connective objects completely
  determines the t-structure, provided its existence, the only thing left
  is to chech is that the description of coconnective objects given in
  Theorem \ref{thm-glued-t-str} coincides with the one given in Definition
  \ref{defi-fild-beilinson-t-str}; that is, we have to prove that
  $$
  \left(\forall n \ \gr^n F \in \scrC_{\leq -n} \text{ and } F^{+\infty} \simeq 0\right)
  \Longleftrightarrow
  \left(\forall n \ F^n \in \scrC_{\leq -n}\right).
  $$

  For the ``only if'' direction: \cite[1.2.1.16]{HA}
  implies that $\scrC_{\leq -n}$ is closed under
  extensions, hence if $F^{n+1} \in \scrC_{\leq -n-1}$ and
  $F^n \in \scrC_{\leq -n}$ the fiber sequence
  $$F^n\to\gr^{n}F\to F^{n+1}[1]$$
  proves that $\gr^n F\in\scrC_{\leq -n}$.
  To check that $F^{+\infty}\simeq 0$, observe that any subset of the form
  $\bbZ\op_{\geq n}$ is an initial subcategory of $\bbZ\op$, hence for any
  $n\in\bbZ$ we have
  $$
  F^{+\infty}\coloneqq\lim\left(\cdots\to F^{n+1}\to F^n \to \cdots\right)
  \simeq\lim\left(\cdots\to F^{n+1}\to F^n\right);
  $$
  as (by \cite[1.2.1.6]{HA}), each $\scrC_{\leq n}$ is closed under all limits
  existing in $\scrC$,
  $$F^{+\infty}\in\scrC_{\leq n} \ \forall n\in\bbZ$$
  (recall that $\scrC_{\leq m}\subseteq\scrC_{\leq n}$ for all pairs of integers
  $m\leq n$) and thus $F^{+\infty}\simeq 0$ by the right separatedness
  hypothesis.

  For the ``if'' direction, let us start by noticing that as
  $\gr^{n+1}F[1]\in\scrC_{\leq -n}$, the fiber
  sequence
  $$F^n/F^{n+2}\to\gr^n F\to\gr^{n+1}F$$
  proves $F^n/F^{n+2}\in\scrC_{\leq -n}$ (again, as the latter \cat is closed
  under limits in $\scrC$). We can now proceed inductively for $m\geq 2$ to
  show that (as $\gr^{n+m}F[1]\in\scrC_{\leq -n-m+1}\subseteq\scrC_{\leq -n})$
  the fiber sequence
  $$F^n/F^{n+m+1}\to F^n/F^{n+m}\to \gr^{n+m}F[1]$$
  implies all objects $F^n/F^k$ for $k>n$ lie in $\scrC_{\leq -n}$.
  Since (again, by \cite[1.2.1.6]{HA}) we know we can compute colimits in
  $\scrC_{\leq -n}$ just by computing them in $\scrC$ and then reflecting
  along the left adjoint to the inclusion (in particular, the colimit
  is the same in both categories if the object already happened to land
  in $\scrC_{\leq -n}$ when computed in $\scrC$), we have that
  \begin{displaymath}
  \begin{split}
    \lim_k F^n/F^k
    &\simeq \lim_k \cofib\left(F^k\to F^n\right)\\
    &\simeq \cofib\left(\lim_k F^k\to F^n\right)\\
    &\simeq F^n/F^{+\infty}
  \end{split}
  \end{displaymath}
  lies in $\scrC_{\leq -n}$. But, as by hypothesis $F^{+\infty}\simeq 0$,
  we have that $F^n\in\scrC_{\leq -n}$ as desired.
\end{proof}

\begin{notat}
  Motivated by the previous results, we will refer to the transferred
  t-structure on $\cFild\scrC$ also as the \emph{Beilinson t-structure}.
\end{notat}

\begin{rem}\label{rem-recollement-of-hearts}
  In the situation of Theorem \ref{thm-glued-t-str}, passing to hearts one
  gets a ``recollement'' of Abelian categories:
  $$
  \begin{tikzcd}[column sep=huge]
    \scrC_0^\heartsuit \ar[r, hook, "i"' description] &
    \scrC^\heartsuit \ar[l, shift right=0.6ex, bend right, "i_L"']
      \ar[from=r, hook, shift right=0.6ex, bend right, "(j_L)_!"']
      \ar[r, "j_L" description]
      \ar[from=r, hook', shift left=0.6ex, bend left, "j"]
      \ar[l, shift left=0.6ex, bend left, "i_R"]
      \ar[l,phantom, shift left=2ex,
        "\text{\scalebox{1}{\rotatebox{-90}{$\dashv$}}}"]
      \ar[l,phantom, shift right=2ex,
        "\text{\scalebox{1}{\rotatebox{-90}{$\dashv$}}}"] &
    \scrC_1^\heartsuit
      \ar[l,phantom, shift left=2ex,
        "\text{\scalebox{1}{\rotatebox{-90}{$\dashv$}}}"]
      \ar[l,phantom, shift right=2ex,
        "\text{\scalebox{1}{\rotatebox{-90}{$\dashv$}}}"]
  \end{tikzcd}
  $$
  where $i$, $j$ and $(j_L)_!$ are fully faithful. As shown already in
  \cite[1.4.18]{BBD83}, in this situation $j_L$ characterizes
  $\scrC_1^\heartsuit$ as the quotient category
  $\scrC^\heartsuit/\scrC_0^\heartsuit$.

  In particular, in the hypothesis of Theorem \ref{thm-beil-recoll}, we have
  that (as $\scrC$ is endowed with the trivial t-structure)
  $\scrC^\heartsuit\simeq 0$ and thus (using Remark \ref{rem-heart-complete})
  $$(\Fild\scrC)^\heartsuit\simeq\left(\cFild\scrC\right)^\heartsuit
  \simeq\coCh\left(\scrC^\heartsuit\right).$$
\end{rem}

\begin{rem}\label{rem-comparison-formulae}
  It follows from Definition \ref{defi-transferred} together with Corollary
  \ref{cor-transferred} that for any $F\in\cFild\scrC$ and $C\in
  \coCh\scrC$ we have
  $$
  \caush \left(\tau^B_{\leq n} F\right) \simeq
  \tau^\text{lvl}_{\leq n} (\caush F) \quad
  \text{and} \quad
  \imp \left(\tau^\text{lvl}_{\leq n} C \right)\simeq
  \tau^B_{\leq n} (\imp C)
  $$
  for all $n\in\bbZ$, and similar formulae for the truncations above $n$.
  Moreover, by Remark \ref{rem-heart-complete}, we also have
  $$\pi^B_n F \simeq \pi^\text{lvl}_n \caush F \quad
  \text{and} \quad
  \pi^\text{lvl}_n C \simeq \pi^B_n \imp C
  $$
  for all $n\in\bbZ$.
\end{rem}

\begin{rem}\label{rem-right-assumption}
  One can construct on $\Fild\scrC$ the glued t-structure as in
  Theorem \ref{thm-glued-t-str} regardless of the right separatedness
  hypothesis;
  if $\scrC$ is not right separated, the class of coconnective objects will be
  the full subcategory
  $$
  \{F\in\Fild_{\leq 0} \ | \ \aush F \in (\coCh\scrC)_{\leq 0}
  \text{ and } F^{-\infty} \simeq 0\}
  $$
  but this class will in general not coincide with the one described in
  Definition \ref{defi-fild-beilinson-t-str}
  (and, as by \cite[1.2.1.3]{HA} the class of connective objects completely
  determines the t-structure, there cannot exist
  a t-structure exactly as in Definition \ref{defi-fild-beilinson-t-str} if
  the two classes of ``candidate coconnective objects'' do not coincide).
  We are not aware of any application for this t-structure.
\end{rem}

\section{Symmetric monoidal structures}\label{section-monoidal}
The $\infty$-categorical Day convolution (introduced in \cite{GlasmanDay} and
further developed in \cite[2.2.6]{HA}) provides a way
to equip with a symmetric monoidal structure any functor category
$\Fun(\scrC,\scrD)$, provided that both $\scrC$ and $\scrD$ are symmetric
monoidal, and $\scrD$ is presentably so.
In particular (see Remark \ref{rem-day-structures}),
we can endow $\Fild\scrC$ with a symmetric monoidal structure whenever
$\scrC$ is presentably symmetric monoidal.
As it turns out, such monoidal structure induces one on $\cFild\scrC$, and
thus on $\coCh\scrC$, whenever Theorem \ref{thm-general-equivalence} applies.
In this section, we analyze these induced symmetric monoidal structures,
and their interaction with the t-structures introduced
in the previous section. In particular, we prove that the Day convolution
structure on both $\Fild\scrC$ and $\cFild\scrC$ is compatible with
Beilinson t-structures, and that the symmetric monoidal structure on
$\coCh\scrC$ provides a homotopy coherent refinement of the usual tensor
product of cochain complexes.

\begin{rem}\label{rem-day-structures}
  Let $\scrC$ be a presentably symmetric monoidal \cat.
  By \cite[2.11]{GlasmanDay} (see also \cite[2.2.6.17]{HA}, with $\kappa$
  chosen to be the strongly inaccessible cardinal determining the size of our
  universe of small sets) we can endow $\Fild(\scrC)$ with the structure of a
  symmetric monoidal \cat, given by Day convolution (where the symmetric
  monoidal structure on $\bbZ\op$ is given by addition).
  Again, by \cite[2.2.6.17]{HA}, if $F$ and $G$ are filtered objects in
  $\scrC$, their Day convolution product is given by Kan extension of
  $\tensor\circ(F,G)$ along $+$, hence by \cite[4.3.3.2]{HTT} and
  \cite[4.3.2.2]{HTT} it is pointwise given by
  \begin{equation}\label{eq-day-formula}
    (F\tensor_{\text{Day}}G)^n \simeq \colim_{
    \substack{(s,t)\in\bbZ\op\times\bbZ\op \\
    s+t\geq n}}
  F^s \tensor G^t
  \end{equation}
  where the shape of the colimit follows from inspection
  of the comma category $(+\downarrow n)$ (see Definition \ref{defi-comma-cat}).
  In particular, $\Fild(\scrC)$ is presentably symmetric monoidal.
  Similarly, we can endow $\prod_\bbZ\scrC\simeq\Fun(\bbZ^\delta,\scrC)$ with
  a presentably symmetric monoidal structure given by Day convolution, whose
  product is pointwise given by
  $$
  \left((X_u)_{u\in\bbZ}\tensor_{\text{Day}}(Y_v)_{v\in\bbZ}\right)_n \simeq
  \bigoplus_{s+t=n} X_s \tensor Y_t
  $$
  where once again the shape of the colimit follows from inspection
  of the comma category $(+\downarrow n)$.
\end{rem}

\begin{prop}\label{prop-unit-day}
  Let $\scrC$ be a presentably symmetric monoidal \cat, whose unit we'll denote
  by $\mathbbm{1}$. Then, the unit object for $(\Fild\scrC,
  \tensor_{\text{Day}})$ is given by the filtered object
  $$\mathbbm{1}_{\langle\leq 0\rangle}\coloneqq
  \cdots \to 0 \to 0 \to \mathbbm{1}\xto{\id}\mathbbm{1}\xto{\id}\cdots$$
  consisting of copies of $\mathbbm{1}$ and identity morphisms for $n\leq 0$,
  and of copies of $0$ for $n>0$.
\end{prop}
\begin{proof}
  First, let us note that it is enough to prove the following claim:
  \begin{enumerate}
    \item[($\spadesuit$)] For any $F\in\Fild\scrC$ and any $n\in\bbZ$
    we have
  $$(F\tensor_{\text{Day}}\mathbbm{1}_{\langle\leq 0\rangle})^n\simeq F^n.$$
  \end{enumerate}
  In fact, the existence of the unitor equivalence
  $\mathbbm{1}_{\text{Day}}\tensor\mathbbm{1}_{\langle\leq 0\rangle}\xto{\sim}
  \mathbbm{1}_{\langle\leq 0\rangle}$
  together with the claim implies the existence of equivalences
  $$\mathbbm{1}_{\text{Day}}^n\simeq
  \left(\mathbbm{1}_{\text{Day}}\tensor
  \mathbbm{1}_{\langle\leq 0\rangle}\right)^n\simeq
  \mathbbm{1}_{\langle\leq 0\rangle}^n$$
  for all $n\in\bbZ$.
  As equivalences in $\Fild\scrC$ can be cheched pointwise, this is enough to
  conclude
  $\mathbbm{1}_{\text{Day}}\simeq \mathbbm{1}_{\langle\leq 0\rangle}$.

  We now turn to the proof of claim ($\spadesuit$).
  By (\ref{eq-day-formula}), this boils down to proving that the colimit
  of the following diagram
  $$
  \begin{tikzcd}
    &\cdots \ar[d] & \cdots \ar[d] & \cdots \ar[d] & \cdots \ar[d] &
      \cdots \ar[d] & \cdots \ar[d]\\
    \cdots \ar[r] & 0 \ar[r]\ar[d] & 0 \ar[r]\ar[d]&
      0 \ar[r]\ar[d]& 0 \ar[r]\ar[d]&
      0 \ar[r]\ar[d]& 0 \\
    \cdots \ar[r] & 0 \ar[r]\ar[d] & 0 \ar[r]\ar[d]&
      0 \ar[r]\ar[d]& 0 \ar[r]\ar[d]& 0 & \\
    \cdots \ar[r] & F^{n+3} \ar[r]\ar[d, equal] & F^{n+2} \ar[r]\ar[d, equal]&
      F^{n+1} \ar[r]\ar[d, equal]& F^n && \\
    \cdots \ar[r] & F^{n+3} \ar[r]\ar[d, equal] & F^{n+2} \ar[r]\ar[d, equal]&
      F^{n+1} &&& \\
    \cdots \ar[r] & F^{n+3} \ar[r]\ar[d, equal] & F^{n+2} &&&& \\
    &\cdots &&&&& \\
  \end{tikzcd}
  $$
  is equivalent to $F^n$. By finality, it is enough to check that the colimit
  of the following diagram
  \begin{equation}\label{eq-zigzag}
  \begin{tikzcd}
    &&&&& \cdots \ar[d]\\
    &&&& 0 \ar[r]\ar[d]& 0 & \\
    &&& 0 \ar[r]\ar[d]& 0 && \\
    && F^{n+1} \ar[r]\ar[d, equal]& F^{n} &&& \\
    & F^{n+2} \ar[r]\ar[d, equal]& F^{n+1} &&&& \\
    F^{n+3} \ar[r]\ar[d, equal] & F^{n+2} &&&&& \\
    \cdots &&&&& \\
  \end{tikzcd}
  \end{equation}
  is equivalent to $F^n$.
  If we denote by $A$ the colimit of the diagram
  \begin{equation}\label{eq-zero-zig}
  \begin{tikzcd}
    &&&&& \cdots \ar[d]\\
    &&&& 0 \ar[r]\ar[d]& 0 & \\
    &&& 0 \ar[r]& 0 && \\
  \end{tikzcd}
  \end{equation}
  and by $B$ the colimit of the diagram
  \begin{equation}\label{eq-zag}
  \begin{tikzcd}
    &&& 0 \ar[d]&&& \\
    && F^{n+1} \ar[r]\ar[d, equal]& F^{n} &&& \\
    & F^{n+2} \ar[r]\ar[d, equal]& F^{n+1} &&&& \\
    F^{n+3} \ar[r]\ar[d, equal] & F^{n+2} &&&&& \\
    \cdots &&&&& \\
  \end{tikzcd}
  \end{equation}
  we have that, by \cite[4.2.3.10]{HTT}, we can decompose the colimit of
  (\ref{eq-zigzag}) as the coproduct $A\oplus B$. As (\ref{eq-zero-zig})
  consists only of zero objects, its colimit is zero. Thus, the colimit of
  (\ref{eq-zigzag}) is equivalent to the colimit of (\ref{eq-zag}).
  By finality, we can omit the top right arrow $0\to F^n$ from the diagram in
  order to compute its colimit.
  By applying inductively \cite[4.2.3.10]{HTT}, we see that $B$ can be computed
  as the iterated pushout
  $$F^n\cpr_{F^{n+1}}F^{n+1}\cpr_{F^{n+2}}F^{n+2}\cdots.$$
  As each of the pushouts is taken along an equivalence, we have $B\simeq F^n$,
  as desired.
\end{proof}

Recall the following definition.
\begin{defi} \label{defi-loc-compatible-with-tensor}
  \cite[2.2.1.7]{HA}
  Let $\scrC$ be a symmmetric monoidal \cat, and let $L\colon\scrC\to\scrC$
  be a localization functor. The functor $L$ is \emph{compatible with the
  symmetric monoidal structure} if for every $L$-equivalence
  $X\to Y$, and every $Z\in\scrC$, the induced $X\tensor Z\to Y\tensor Z$
  is an $L$-equivalence.
\end{defi}

The following proposition already appeared as \cite[2.25]{Gwilliam-Pavlov},
we present an alternative proof.

\begin{prop}\label{prop-loc-comp-with-day}
  Let $\scrC$ be a presentably symmetric monoidal \cat. Then,
  the localization functor $L\colon\Fild\scrC\to\cFild\scrC$ is compatible with
  Day convolution.
\end{prop}
\begin{proof}
  Note that, by \cite[3.11]{NikStableOperads}, $\Fild\scrC$ admits an internal
  mapping object given by
  $$n\mapsto\innmap(F,G)(n)\simeq\int_{m\in\bbZ\op}\map_{\scrC}(F(m),G(m+n))$$
  (where $\map_{\scrC}$ denotes the internal mapping object of $\scrC$)
  By \cite[2.12 (3)]{NikStableOperads}, it suffices to prove that, given any
  $F\in\Fild\scrC$ and any $G\in\cFild\scrC$, the internal mapping object is
  complete; by the following computation
  \begin{displaymath}
  \begin{split}
    \lim_{n\in\bbZ\op} \ \innmap(F,G)(n)
    &\simeq\lim_{n\in\bbZ\op}\int_{m\in\bbZ\op}\map_{\scrC}\bigg(F(m),G(m+n)\bigg)\\
    &\simeq\int_{m\in\bbZ\op}\lim_{n\in\bbZ\op}\map_{\scrC}\bigg(F(m),G(m+n)\bigg)\\
    &\simeq\int_{m\in\bbZ\op}\map_{\scrC}\bigg(F(m),\lim_{n\in\bbZ\op}G(m+n)\bigg)\\
    &\simeq\int_{m\in\bbZ\op}\map_{\scrC}\bigg(F(m),0\bigg)\\
    &\simeq \vphantom{\int}0.
  \end{split}
  \end{displaymath}
  this is in fact the case.
\end{proof}

\begin{rem}\label{rem-day-localized}
  \cite[2.25]{Gwilliam-Pavlov}
  It follows from Proposition \ref{prop-loc-comp-with-day} and
  \cite[2.2.1.9]{HA} that if $\scrC$ is presentably symmetric monoidal,
  we have an induced presentably symmetric monoidal structure on
  $\cFild(\scrC)$, which we'll refer to as the \emph{completed Day convolution}
  monoidal structure and
  will denote by $\widehat{\tensor}$. In particular, we have that
  $$F \ \widehat{\tensor} \ G \simeq L \Big(F\tensor_{\text{Day}}G\Big).$$
\end{rem}

\begin{rem}\label{rem-localized-unit}
  From Proposition \ref{prop-loc-comp-with-day} and
  Remark \ref{rem-day-localized}, we have
  that the unit for $\widehat\tensor$ is given by
  $$L\mathbbm{1}_{\langle\leq 0\rangle}\simeq\mathbbm{1}_{\langle\leq 0\rangle}.$$
\end{rem}

\begin{defi}\label{defi-coh-tensor-product-coch}
  Let $\scrC$ be a stable presentably symmetric monoidal \cat.
  We refer to the symmetric monoidal structure induced on $\coCh(\scrC)$ by the
  equivalence of Theorem \ref{thm-general-equivalence} as the \emph{coherent
  cochains tensor product}, and we will denote it just by $\tensor$.
\end{defi}

\begin{rem}\label{rem-square-aush-gr-shift}
  It follows from Theorem \ref{thm-general-equivalence} together with
  Remark \ref{rem-explicit-description-general-case}
  that the functor $\undch$ defined in Lemma
  \ref{lemma-coch-to-prod-conservative}
  fits into the following commutative diagram
  $$
  \begin{tikzcd}
    \Fild\scrC\ar[r,"\aush"] \ar[d, "\gr"'] &[15pt] \coCh(\scrC) \ar[d,"\undch"] \\
    \prod_\bbZ\scrC\ar[r,"(\Sigma^n)_{n\in\bbZ}"] &[15pt] \prod_\bbZ\scrC.
  \end{tikzcd}
  $$
  that is, $\undch \circ \aush F \simeq \left(\gr^n F [n]\right)_{n\in\bbZ}$
  naturally in $F\in\Fild\scrC$.
\end{rem}

The symmetric monoidal stucture given in Definition
\ref{defi-coh-tensor-product-coch} is really a homotopy coherent generalization
of the usual tensor product of cochain complexes, in a sense made precise
by the following results.

\begin{prop}\label{prop-undch-aush-symm}
  The functor $\undch \circ \aush$ is symmetric monoidal.
\end{prop}
\begin{proof}
  This is a direct consequence of \cite[2.26]{Gwilliam-Pavlov} together
  with Remark \ref{rem-square-aush-gr-shift} and $(\Sigma^n)_{n\in\bbZ}$
  being an equivalence.
\end{proof}

\begin{cor}\label{cor-coh-tensor-formula}
  Let $\scrC$ be a stable presentably symmetric monoidal \cat.
  Let $C$ and $D$ be elements of $\coCh(\scrC)$. Then, we have that
  $$(C \tensor D)^n \simeq \bigoplus_{s+t=n} C^s \tensor_\scrC D^t .$$
\end{cor}
\begin{proof}
  By definition of $\undch$ (see Lemma
  \ref{lemma-coch-to-prod-conservative}),
  $(C\tensor D)^n \simeq \undch(C\tensor D)^n$. The result follows from
  the following computation
  \begin{displaymath}
  \begin{split}
    \undch(C\tensor D)^n
    &\stackrel{\vphantom{(}^{(\ref{thm-cohch-fild-equiv})}}{\simeq}
      \undch \Big(\caush \imp C \tensor \caush \imp D\Big)^n \\
    &\stackrel{\vphantom{(}^{(\ref{defi-coh-tensor-product-coch})}}{\simeq}
      \undch \caush \Big(\imp C \ \widehat{\tensor} \ \imp D\Big)^n \\
    &\stackrel{\vphantom{(}^{(\ref{rem-day-localized})}}{\simeq}
      \undch \caush L \Big(\imp C \tensor_{\text{Day}} \imp D\Big)^n \\
    &\stackrel{\vphantom{(}^{(\ref{prop-imp-is-complete})}}{\simeq}
      \undch \aush \Big(\imp C \tensor_{\text{Day}} \imp D\Big)^n \\
    &\stackrel{\vphantom{(}^{(\ref{prop-undch-aush-symm})}}{\simeq}
      \Big(\undch \aush \imp C \tensor_{\text{Day}} \undch
      \aush \imp D\Big)^n \\
    &\stackrel{\vphantom{(}^{(\ref{prop-counit-almost-equiv})}}{\simeq}
    \vphantom{\Big(}\Big(uC \tensor_{\text{Day}} uD\Big)^n \\
    &\stackrel{\vphantom{(}^{(\ref{rem-day-structures})}}{\simeq}
    \bigoplus_{s+t=n} C^s \tensor_{\scrC} D^t.
  \end{split}
  \end{displaymath}
\end{proof}

We now analyze the interaction between the symmetric monoidal structures
introduced in this section, and the t-structures introduced in Section
\ref{section-t-structures}. We start by recalling the following definition
(see \cite[2.2.1]{HA} and \cite[A.2]{AnNik20} for more details about the
general theory of the interaction between t-structures and symmetric monoidal
structures).

\begin{defi}
  Let $\scrC$ be a stably symmetric monoidal \cat equipped with a t-structure
  $(\scrC_{\leq 0}, \scrC_{\geq 0})$.
  The t-structure is said to be \emph{compatible} with the symmetric monoidal
  structure if the following conditions hold:
  \begin{enumerate}
    \item The unit object for $\tensor$ lies in $\scrC_{\geq 0}$;
      \label{item-unit}
    \item Given any pair of connective objects $X,Y\in\scrC_{\geq 0}$,
      their product $X\tensor Y$ lies in $\scrC_{\geq 0}$ as well.
      \label{item-product}
  \end{enumerate}
\end{defi}

\begin{rem}
  Conditions (\ref{item-unit}) and (\ref{item-product}) guarantee that
  $\scrC_{\geq 0}$ inherits a symmetric monoidal structure from $\scrC$
  such that the fully faithful inclusion
  $\scrC_{\geq 0}\hookrightarrow\scrC$ is a (strong) symmetric monidal functor.
\end{rem}

\begin{prop}
  Let $\scrC$ be a presentably symmetric monoidal \cat equipped with
  a t-structure that is compatible with the monoidal structure. Then:
  \begin{enumerate}
    \item The Beilinson t-structure on $\Fild\scrC$ is compatible with
      the Day convolution product $\tensor_{\text{Day}}$;
      \label{item-day-t}
    \item The Beilinson t-structure on $\cFild\scrC$ is compatible
      with the completed Day convolution product $\widehat\tensor$;
      \label{item-complete-day-t}
    \item If $\scrC$ is moreover stable, the pointwise
      t-structure on $\coCh\scrC$ is compatible with the coherent cochains
      tensor product $\tensor$
      \label{item-coch-tensor-t}
  \end{enumerate}
\end{prop}
\begin{proof}
  By Proposition \ref{prop-unit-day}, we have that
  $$
  \gr^i\mathbbm{1}_{\langle\leq 0\rangle}\simeq
  \begin{cases}
    0 &\text{ for } i\ne 1;\\
    \mathbbm{1} &\text{ for } i=1.
  \end{cases}
  $$
  and thus the unit is Beilinson connective.
  Let us now consider $F, G\in\Fild\scrC_{\geq 0}$.
  By \cite[2.26]{Gwilliam-Pavlov} we have that
  $$
  \gr^i(F \tensor_{\text{Day}} G)\simeq \bigoplus_{s+t=i} \gr^s F\tensor\gr^t G
  $$
  from which we immediately have that $F\tensor_{\text{Day}}G$ lies in
  $(\Fild\scrC)_{\geq 0}$. This completes the proof of (\ref{item-day-t}).
  By Proposition \ref{prop-cfild-bousfield}, applying the localization
  functor $L$ has no effect on associated gradeds, hence (\ref{item-day-t})
  immediately implies (\ref{item-complete-day-t}). Finally,
  (\ref{item-coch-tensor-t}) is a trivial consequence of
  (\ref{item-complete-day-t}) and the definition of $\tensor$ on $\coCh\scrC$.
\end{proof}

\begin{rem}\label{rem-refinement-usual-tensor}
  It follows from \cite[2.2.1.10]{HA} that Day convolution, and hence also the
  tensor product of coherent chain complexes, induce a symmetric monoidal
  structure on the hearts of the respective Beilinson t-structures.
  From Corollary \ref{cor-coh-tensor-formula}, we have
  that the induced symmetric monoidal structure on
  $$\cFild(\Sp)^\heartsuit\simeq\coCh(\Sp)^\heartsuit\simeq\coCh(\bbZ)$$
  is the usual one.
\end{rem}

\section{Coherent cochain complexes and Toda brackets}
\label{section-toda}

In this section, we will have a closer look at the relation between
Toda brackets and coherent cochain complexes.
Our main result will be characterizing coherent cochain complexes in suitable
pointed \cats as being precisely the sequences of objects
$$\cdots\to X^n \xto{f^n} X^{n+1}\xto{f^{n+1}} X^{n+2}\cdots$$
such that all pairwise composable morphisms are nullhomotopic (i.e. for all
$n\in\bbZ$ we have $f^{n+1}\circ f^n\simeq 0$) and all possible Toda brackets
are compatibly trivial (see Remark \ref{rem-informal-toda}).

Let us first start with a short recollection about Toda brackets.
Given a pointed \cat $\scrC$, let
$f^0\colon X^0 \to X^1$, $f^1\colon X^1 \to X^2$ and $f^2\colon X^2 \to X^3$
be morphisms in $\scrC$, such that $f^1 f^0\simeq f^2 f^1 \simeq 0$, and let
$\alpha\colon 0\Rightarrow f^1 f^0$ and $\beta\colon 0 \Rightarrow f^1 f^2$ be
two choices of nullhomotopies:
$$
\begin{tikzcd}
  X^0 \ar[r, "f^0"] \ar[rr, bend right=40, "0"']
  \ar[rr, phantom, bend right=20, "\rotatebox{90}{$\Rightarrow$} \alpha"]
& X^1 \ar[r, "f^1"]
& X^2 \ar[r, "f^2"]
& X^3. \ar[from=ll, bend left=40, "0"]
  \ar[from=ll, phantom, bend left=20, "\rotatebox{90}{$\Leftarrow$} \beta"]
\end{tikzcd}
$$
In such a situation, the two whiskerings $f^2 \alpha$ and $\beta f^0$ determine
two paths $0\Rightarrow f^2 f^1 f^0$ in $\Map_*(X^0,X^3)$.
By gluing them along their endpoints, we obtain a map (pointed at zero)
$T\colon S^{1}\to\Map_*(X^0,X^3)$, which in turns determines a class
$\langle f^2,f^1,f^0 \rangle_{(\alpha, \beta)} \in \pi_1(\Map(X^0,X^3))$.
The homotopy class $\langle f^2,f^1,f^0 \rangle_{(\alpha, \beta)}$ is known as
the Toda bracket determined by $\alpha$ and $\beta$.

The customary approach to
Toda brackets would be not fix $\alpha$ and $\beta$ as part of the datum,
and to instead define the Toda bracket to be a subset of $\pi_1(\Map(X^0,X^3))$,
whose elements are given by all the possible choices of homotopy classes of
paths $([\alpha], [\beta])$. According to this definition, the object
referred above as $\langle h,g,f \rangle_{(\alpha, \beta)}$ is a specified
element of this subset.

The notion can be generalized to longer sequences of maps, provided that the
Toda brackets of the shorter sub-sequences are nullhomotopic; to illustrate how
the generalization works, let us consider the case of a sequence of length 4
$$
\begin{tikzcd}
  X^0 \ar[r, "f^0"]
& X^1 \ar[r, "f^1"]
& X^2 \ar[r, "f^2"]
& X^3 \ar[r, "f^3"]
& X^4.
\end{tikzcd}
$$
and choices of nullhomotopies $\alpha$, $\beta$ and $\gamma$ for all consecutive
pairs of maps, for which $\langle f^2,f^1,f^0 \rangle_{(\alpha, \beta)}=0$
and $\langle f^3,f^2,f^1 \rangle_{(\beta, \gamma)}=0$; in such a situation,
we obtain a pair of nullhomotopies for a pair of maps $X^0\to\Sigma X^4$, one
for the composite
$$
\begin{tikzcd}[column sep=6em]
  X^0 \ar[r, "f^0"]
& X^1 \ar[r, "{\langle f^3,f^2,f^1 \rangle_{(\beta, \gamma)}}"]
& \Sigma X^4
\end{tikzcd}
$$
and the other one for composite
$$
\begin{tikzcd}[column sep=6em]
  X^0
  \ar[r, "{\langle f^2,f^1,f^0 \rangle_{(\alpha, \beta)}}"]
& \Sigma X^3 \ar[r, "\Sigma f^3"]
& \Sigma X^4;
\end{tikzcd}
$$
putting together these two as before, we obtain a pointed map
$S^1\to\Map_*(X^0,\Sigma X^4)$, and thus a class
$\langle f^3,f^2,f^1,f^0 \rangle\in\pi_2(\Map_*(X^0,X^3))$,
the \emph{4-fold Toda bracket} of the
sequence\footnote{again, we are working with a coherent notion, depending also
on the choices of all the five nullhomotopies involved}.
This idea clearly generalizes to longer sequences of maps,
provided that the shorter subsequences have nullhomotopic brackets.

Although it is possible to work with these notions using the approach sketched
above, in the rest of the section we will use a
slightly different perspective to define and generalize Toda brackets.
Our approach will be to define these classes by means
of the actions of certain algebra objects in graded pointed spaces;
the main advantage for us will be the ease of working
at once with all the possible $n$-fold Toda brackets in a $\bbZ$-indexed sequence
(provided that all the possible $(n-1)$-fold Toda brackets are
trivial);
of course, one can recover the case of finite sequences
by considering a $\bbZ$-indexed sequence where only finitely many
objects are non-zero (see Definition \ref{defi-toda-bracket} for the details).
One other pleasant feature
of our approach will be to obtain an inherently coherent notion of Toda brackets
(that is, encoding also choices for all the involved homotopies, rather than
defining the notion up to a coset of ``indeterminacies'').

Our strategy for proving that coherent chain complexes are precisely those
$\bbZ$-indexed sequences of morphisms where consecutive maps compose to zero
and all Toda brackets are coherently nullhomotopic requires a few steps:
in \S \ref{subsection-algebras},
we will prove that for a pointed \cat $\scrC$
(which we will for simplicity assume to be also presentable) both
$\Fun(\bbZ\op,\scrC)$\footnote{in the present section we will use the notation
  $\Fun(\bbZ\op,\scrC)$ instead of the previously introduced $\Fild(\scrC)$
  to stress that here we want to think not of structured
  objects of independent interest, but of sequences
  of morphisms upon which we intend to put relations;
  in particular, the reader should
  keep in mind that the appearence of this \cat here has nothing to do with
  the equivalence of Theorem \ref{thm-general-equivalence}}
and $\coCh(\scrC)$ are
\cats of modules for suitable $\bbE_1$-algebra objects in $\Gr\spaces_*$.
Then, in \S \ref{subsection-inductive-step}
we will construct a sequence of $\bbE_1$-algebras in $\Gr\spaces_*$
\begin{equation}\label{eq-sequence-of-algebras}
  R_1\to R_2\to R_3 \to \cdots
\end{equation}
such that the \cat of modules over $R_1$ will be equivalent to
$\Fun(\bbZ\op,\scrC)$, the \cat of modules over $\colim_n R_n$ will be
equivalent to $\coCh(\scrC)$, and such that the sequence induced
by (\ref{eq-sequence-of-algebras}) on
homology will be equivalent to the resolution of $\Lambda(e)$ described in
Theorem \ref{thm:e1presentation-ungraded}, with $e$ of degree 1.
Modules over $R_n$ for a fixed $n\geq 2$ will
be sequences where all pairwise composable morphisms will be nullhomotopic,
all possible $m$-fold Toda brackets for $m\leq n$ will be defined and trivial,
and all the relevant nullhomotopies will be encoded in the module structure
(see Remark \ref{rem-recursive-module-structure} for the details).
In order to prove the existence of the sequence (\ref{eq-sequence-of-algebras}),
we will use the results of Appendix \ref{appendix-achim}.
As it turns out, the construction of the above sequence up to $R_3$ is
slightly more subtle than its extension to all the other $R_n$'s,
hence we will first construct the sequence inductively for $n\geq 4$,
and defer the construction of the beginning of the sequence
to \S \ref{subsection-hard-cases}.

\subsection{Algebras in graded pointed spaces}\label{subsection-algebras}
The main goal of this section is to prove that, given a pointed \cat $\scrC$,
both $\Fun(\bbZ\op,\scrC)$ and $\coCh(\scrC)$ can be expressed as \cats of
modules over suitably defined $\bbE_1$-algebras in graded pointed spaces.
For simplicity, we will assume that $\scrC$ is also presentable, but with
some extra effort the results can be generalized further, using methods
along the lines of those employed in Section \ref{section-general-equiv}.

The stable analogue of the following proposition appeared as
\cite[2.4.4]{lurie2015rotation}.

\begin{prop}\label{prop-action-of-Z}
  Let $\scrC$ be a pointed \cat. The following data are equivalent:
  \begin{enumerate}
    \item self-equivalences $\scrC\to\scrC$;
      \label{item-selv-eq}
    \item (left or right) actions of $\bbZ^\delta$ on $\scrC$;
      \label{item-action}
    \item monoidal functors $\bbZ^\delta\to\Fun(\scrC,\scrC)$;
      \label{item-monoidal}
    \item monoidal functors $\bbZ^\delta\to\Fun^{\text{Rex}}(\scrC,\scrC)$;
      \label{item-mon-two}
    \item right exact monoidal functors
      $\Prefinpt{\bbZ^\delta}\to\Fun^{\text{Rex}}(\scrC,\scrC)$;
      \label{item-rex}
    \item (left or right) actions of $\Prefinpt{\bbZ^\delta}$ on $\scrC$
      such that the action map commutes with finite colimits in each variable.
      \label{item-act-two}

    {\phantom{x}}

    {\noindent If $\scrC$ is pointed presentable, then the above are equivalent to:}

    {\phantom{x}}

    \item left adjoint monoidal functors
      $\Prept{\bbZ^\delta}\to\Fun^{\text{Rex}}(\scrC,\scrC)$;
      \label{item-presentable-ladj}
    \item actions of $\Prept{\bbZ^\delta}$ on $\scrC$ such that the action
      map commutes with small colimits in each variable.
      \label{item-presentable-action}
  \end{enumerate}
  Here $\Prept{\bbZ^\delta}$ and $\Prefinpt{\bbZ^\delta}$ are defined as in \S
  \ref{subsec-appendix-stable-nerve}.
\end{prop}
\begin{proof}
  The equivalence (\ref{item-selv-eq}) $\Leftrightarrow$
  (\ref{item-monoidal}) follows
  from \cite[2.4.3]{lurie2015rotation}: \emph{loc.\ cit.}\ gives an equivalence
  $$\Fun^\otimes\kern-0.2em\left(\bbZ^\delta,\Fun(\scrC,\scrC)\right)
  \simeq\Aut(\scrC)$$
  where the monoidal structure on $\Fun(\scrC,\scrC)$ is given by composition,
  and $\Aut(\scrC)\subset\Fun(\scrC,\scrC)$ denotes the full subcategory spanned
  by self-equivalences.
  (\ref{item-monoidal})$\Leftrightarrow$(\ref{item-action}) and
  (\ref{item-rex})$\Leftrightarrow$(\ref{item-act-two}) are just
  reformulations of the definition of action of a
  symmetric monoidal \cat.
  (\ref{item-monoidal})$\Leftrightarrow$(\ref{item-mon-two}) holds trivially, as
  $\bbZ^\delta$ is a discrete category.
  (\ref{item-mon-two})$\Leftrightarrow$(\ref{item-rex}) follows from
  Proposition \ref{prop-pointed-rex-univ} and the monoidality of the pointed
  Yoneda embedding (see \cite[\S 6]{NikStableOperads}).
  (\ref{item-presentable-ladj})$\Leftrightarrow$(\ref{item-rex}) and
  (\ref{item-presentable-action})$\Leftrightarrow$(\ref{item-act-two})
  follow at once from Proposition \ref{prop-pointed-rex-univ}.
\end{proof}

\begin{rem}\label{rem-terminology-action-tensoring}
  In what follows, unless otherwise specified, we will implicitly work with
  left actions whenever we apply Proposition \ref{prop-action-of-Z}.
  Recall that, in the terminology of \cite{HA}, a left action of a monoidal \cat
  $\scrC$ on some \cat $\scrM$ is equivalent to exhibiting $\scrM$ as
  left-tensored over $\scrC$, and to exhibiting $\scrM$ as a left $\scrC$-module
  in (a suitably sized) \cat of \cats (see \cite[p. 8]{lurie2015rotation}).
\end{rem}

\begin{ex}\label{ex-pfinpt-action}
  Let $\scrC$ be a pointed \cat. Then, precomposition with
  $-1\colon\bbZ\to\bbZ$ induces an automorphism
  $$(-)\{1\}\colon\coCh\scrC\to\coCh\scrC \ \colon
  C^\bullet\mapsto C\{1\}^\bullet \simeq C^{\bullet -1};$$
  by virtue of Proposition
  \ref{prop-action-of-Z}, this endowes $\coCh\scrC$ with the structure of
  an \cat left-tensored over $\Prefinpt{\bbZ^\delta}\simeq\Gr(\spaces_*^{\text{fin}})$
  (and when $\scrC$ is also presentable, the left-tensoring extends over
  $\Gr(\spaces_*)$);
  we will use the notation:
  $$(-)\{n\}\colon\coCh\scrC\to\coCh\scrC \quad \forall n\in\bbZ$$
  to denote the (right exact) endofunctor induced by $n$.
\end{ex}

\begin{ex}\label{ex-gr-action-on-fil}
  Let $\scrC$ be a pointed \cat.
  Similarly to the previous example, precomposition with
  $-1\colon\bbZ\op\to\bbZ\op$ induces an automorphism
  $$(-)\{1\}\colon\Fild\scrC\to\Fild\scrC;$$
  hence, $\Fild\scrC$ is also left-tensored over
  $\Gr(\spaces_*^{\text{fin}})$
  (and when $\scrC$ is also presentable, the left-tensoring extends over
  $\Gr(\spaces_*)$);
  also in this case, we will use the notation:
  $$(-)\{n\}\colon\Fild\scrC\to\Fild\scrC \quad \forall n\in\bbZ$$
  to denote the automorphism induced by $n$.
\end{ex}

\begin{notat}
  Given any $n\in\bbN$ and any $t\in\bbZ$, we will denote by
  $S^{n,t}\in\coCh\spaces_*$ the graded pointed space consisting of a copy of
  $S^n$ in degree $t$, and copies of $\pt$ in all other degrees.
\end{notat}

\begin{prop}\label{prop-graded-mapping-space}
  Let $C\in\coCh\scrC$ be a coherent cochain complex in a complete\footnote{or,
  more generally, in a pointed \cat where all the relevant right Kan extensions
  exist and are pointwise}, pointed \cat.
  The functor $-\tensor C \colon \Gr\spaces_* \to \coCh\scrC$ induced by
  the left-tensoring of Example \ref{ex-pfinpt-action} admits a right adjoint,
  denoted $\innmap{}_*(C,-)$,
  whose value on any $D\in\coCh\scrC$ is pointwise given by
  $$\innmap{}_*(C,D)^n \simeq\ptMap\left(C\{n\},D\right).$$
\end{prop}
\begin{proof}
  By unraveling the definitions, the $\Gr\spaces_*$-action
  on $\Fun^{\text{Rex}}(\coCh\scrC,\coCh\scrC)$ is obtained by Kan
  extension of the functor
  $\{\}\colon\bbZ^\delta\to\Fun^{\text{Rex}}(\coCh\scrC,\coCh\scrC)\colon
  n\mapsto\{n\}$ along the functor $y\colon\bbZ^\delta\to\Gr\spaces_*\colon
  n\mapsto S^{0,n}$. As the Kan extension is pointwise,
  $$-\tensor C\simeq \Lan_y(\ev_C \circ\{\})$$
  where $\ev_C\colon\Fun^{\text{Rex}}(\coCh\scrC,\coCh\scrC)\to\coCh\scrC$
  denotes the evaluation at $C$ functor.
  In particular, if we put $F\coloneqq\ev_C\circ\{\}$, we have that
  $-\tensor C$ is, in the language of Appendix \ref{appendix-first},
  the pointed $F$-realization functor. Thus, it admits as a right adjoint
  given by the pointed $F$-nerve functor $\nerve_F^{\text{pt}}$,
  which by Proposition \ref{prop-pt-mapping-nerve} is given by
  $$\nerve_F^{\text{pt}}(D)\simeq \ptMap(F(-),D)\simeq
  \ptMap\left(C\{-\},D\right)$$
  as desired.
\end{proof}

\begin{defi}
  We will refer to the right adjoint functor defined in Proposition
  \ref{prop-graded-mapping-space}
  as the \emph{graded pointed mapping space} functor.
\end{defi}

\begin{thm}\label{thm-ch-is-right-modules}
  Let $\scrC$ be a pointed presentable \cat. Then,
  there exists an $\bbE_1$-algebra in graded pointed spaces
  $A\in\Alg_{\Gr\spaces_*}$
  such that the \cat of coherent cochain
  complexes in $\scrC$ is equivalent to the \cat of right $A$-modules
  in graded objects of $\scrC$:
  $$\coCh\scrC\simeq\operatorname{RMod}_A(\Gr\scrC).$$
  Moreover, the underlying graded pointed space of $A$ is given by
  $S^{0,0}\amalg S^{0,1}$.
\end{thm}
\begin{proof}
  Let us first notice that the claim for a generic pointed presentable
  \cat $\scrC$ can be deduced from the case $\scrC\simeq\spaces_*$,
  by means of Lurie's tensor product (see \cite[\S 4.8]{HA}):
  by \cite[4.8.1.17]{HA},
  $$\Fun(\catCh\op,\scrC)\simeq\Fun(\catCh\op,\spaces_*)\tensor\scrC;$$
  since preserving the zero object is a property, the equivalence above induces
  an equivalence $\coCh(\scrC)\simeq\coCh(\spaces_*)\tensor\scrC$.
  Now, by \cite[4.8.4.6]{HA},
  $$\operatorname{RMod}_A(\Gr\spaces_*)
  \tensor_{\Gr\spaces_*}\scrC
  \simeq \operatorname{RMod}_A(\Gr\scrC);$$
  it thus suffices to prove our claim for the case of pointed spaces.

  By Example \ref{ex-pfinpt-action}, together with Proposition
  \ref{prop-action-of-Z}, we have that $\coCh\spaces_*$ is left-tensored
  over $\Gr\spaces_*$ (see Remark \ref{rem-terminology-action-tensoring}).
  By \cite[4.8.5.8]{HA}, it is thus enough to prove that:
  \begin{enumerate}
    \item $\coCh(\spaces_*)$ admits geometric realizations of simplicial
      objects.
      \label{item-proof-graded-modules-one}
    \item The action map $\Gr\spaces_* \times \coCh(\spaces_*) \to
      \coCh(\spaces_*)$ preserves geometric realizations of simplicial objects.
      \label{item-proof-graded-modules-two}
    \item There exists an $M\in\coCh(\spaces_*)$ such that
      the functor $-\tensor M \colon \Gr\spaces_* \to \coCh(\spaces_*)$
      admits a right adjoint, denoted $\underline{\Map}{}_*(M,-)$.
      \label{item-proof-graded-modules-three}
    \item The functor $\underline{\Map}{}_*(M,-)$ preserves geometric realizations
      of simplicial objects.
      \label{item-proof-graded-modules-four}
    \item The functor $\underline{\Map}{}_*(M,-)$ is conservative.
      \label{item-proof-graded-modules-five}
    \item For every coherent chain complex $C$ and every graded pointed space
      $X$, the map
      $$X\tensor \underline{\Map}{}_*(M,C)\tensor M \xto{X\tensor\eps_C}
      X\tensor C$$
      is adjoint to an equivalence
      $$X\tensor \underline{\Map}{}_*(M,C)\to\underline{\Map}{}_*(M,X\tensor C).$$
      \label{item-proof-graded-modules-six}
  \end{enumerate}

  We take $M$ to be the coherent cochain complex consisting of a copy of
  $S^0$ sitting in degree $0$, a copy of $S^0$ sitting in degree $1$,
  copies of $\pt$ elsewhere, and
  the identity as its only possibly nontrivial differential.
  (\ref{item-proof-graded-modules-one}) follows from the cocompleteness of
  $\spaces_*$.
  (\ref{item-proof-graded-modules-two}) follows from presentability of
  $\spaces_*$ and Proposition \ref{prop-action-of-Z}.
  (\ref{item-proof-graded-modules-three}) and
  (\ref{item-proof-graded-modules-four}) hold for any choice of $M$, and
  follow from Proposition \ref{prop-graded-mapping-space}.
  Using the explicit description for the right adjoint given in Proposition
  \ref{prop-graded-mapping-space}, we get that for our choice of
  $M$, $\underline{\Map}{}_*(M,C)^n\simeq\ptMap(M\{n\},C)$
  is equivalent to the space of choices of $f$ and $g$'s making the square
  $$
  \begin{tikzcd}
    S^0 \ar[r, "\id"]\ar[d, "f"'] & S^0 \ar[d, "g"] \\
    C^n \ar[r, "\partial^n"]& C^{n+1}
  \end{tikzcd}
  $$
  commute; this space is in turn equivalent to $\ptMap(S^0,C^n)\simeq C^n$,
  and thus the functor $\underline{\Map}{}_*(M,-)$ is equivalent to the
  forgetful functor
  $\undch\colon\coCh\spaces_*\to\Gr\spaces_*$ introduced in Lemma
  \ref{lemma-coch-to-prod-conservative},
  implying (\ref{item-proof-graded-modules-five}).

  For (\ref{item-proof-graded-modules-six}), by definition of adjoint map,
  the adjoint map of interest can be factored as
  $$ X\tensor\underline{\Map}{}_*(M,C)
  \xto{\eta_{X\tensor\underline{\Map}{}_*(M,C)}}
  \underline{\Map}{}_*(M,X\tensor \underline{\Map}{}_*(M,C)\tensor M)
  \xto{\underline{\Map}{}_*(M,X\tensor\eps_C)}$$
  $$\xto{\underline{\Map}{}_*(M,X\tensor\eps_C)}
  \underline{\Map}{}_*(M,X\tensor C).$$
  By a pointwise check, we have that
  $\eta_{X\tensor\underline{\Map}{}_*(M,C)}\simeq
  \underline{\Map}{}_*(M,\eta_{X\tensor C})$,
  and $X\tensor\eps_c\simeq\eps_{X\tensor C}$;
  by the conservativity of $\underline{\Map}{}_*(M,-)$,
  it is sufficient to show that the composite
  $$X\tensor C\xto{\eta_{X\tensor C}}
  \underline{\Map}{}_*(M,X\tensor C)\tensor M
  \xto{\eps_{X\tensor C}} X\tensor C$$
  is an equivalence;
  but, as the latter is one of the triangle identities, the desired condition
  holds.

  Finally, in the proof of \cite[4.8.5.8]{HA} is shown that $A$ can be
  identified with the internal hom object
  $\innmap{}_*(M,M)$,
  thence we get the desired characterization for the graded pointed
  space underlying $A$.
\end{proof}

\begin{rem}
  As, by \cite[5.5]{ggn15}, the stabilization functor is symmetric
  monoidal on presentable \cats, we can stabilize Theorem
  \ref{thm-ch-is-right-modules}, to obtain an equivalence
  $$\coCh(\scrD)\simeq\operatorname{RMod}_{\bbS^{0,0}\oplus\bbS^{0,1}}
    (\Gr\scrD)$$
  for any stable presentable \cat $\scrD$.
\end{rem}

\begin{thm}\label{thm-fild-is-right-modules}
  Let $\scrC$ be a pointed presentable \cat. Then,
  the \cat of filtered objects of $\scrC$ is equivalent to the \cat of right
  $\rr{Free}_{\bbE_1}(S^{0,1})$-modules in graded objects:
  $$\Fun(\bbZ\op,\scrC)\simeq
  \operatorname{RMod}_{\rr{Free}_{\bbE_1}(S^{0,1})}(\Gr\scrC).$$
\end{thm}
\begin{proof}
  Let us first prove the claim for the case $\scrC\simeq\spaces_*$.
  This is basically an adaptation of \cite[3.1.6]{lurie2015rotation} to
  the case of pointed spaces.

  First of all, notice that, as the functor
  $u\colon\Fun(\bbZ\op,\spaces_*)\to\Gr\spaces_* $
  given by precomposition with $\bbZ^\delta\to\bbZ\op\colon n\mapsto -n$
  is symmetric monoidal, it induces a functor
  $$
  \theta\colon
  \Fun(\bbZ\op,\spaces_*)\simeq\Mod_{\mathbbm{1}_{\text{Day}}}
  (\Fun(\bbZ\op,\spaces_*))\to\Mod_{u\mathbbm{1}_{\text{Day}}}(\Gr\spaces_*);
  $$
  by Proposition \ref{prop-unit-day} together with \cite[4.1.1.18]{HA}, we have
  that $u\mathbbm{1}_{\langle\leq 0\rangle}$ is an $\bbE_\infty$-algebra
  whose underlying $\bbE_1$-algebra is equivalent to
  $\rr{Free}_{\bbE_1}(S^{0,1})$.

  It is thus sufficient to show that the functor $\theta$ is an equivalence.
  Let us start with fully faithfulness.
  That is, we want to prove that the maps
  $$\phi_{X,Y}\colon\ptMap(X,Y)\to\ptMap(\theta X, \theta Y)$$
  are all equivalences.
  As the functor $\phi_{-,Y}\colon\Fun(\bbZ\op,\spaces_*)\to
  \Fun(\Delta^1,\spaces_*)$ sending a graded pointed space $X$ to the map
  $\phi_{X,Y}$ sends colimits to limits, the full subcategory spanned by those
  $X$ such that $\phi_{X,Y}$ is an equivalence is a pointed subcategory of
  $\Fun(\bbZ\op,\spaces_*)$ closed under colimits.
  Thus, as $\Fun(\bbZ\op,\spaces_*)$ is generated under colimits by objects
  of the form $\ptYo_n$ for $n\in\bbZ$ (see Proposition
  \ref{prop-pointed-rex-univ}), it suffices to show that
  $\phi_{-,Y}$ is an equivalence for such generators;
  this is indeed the case, as for $X\simeq\ptYo_n$ we have
  $\phi_{X,Y}\simeq\id_{Y^n}$.
  Essential surjectivity now follows from the observation that the essential
  image of $\theta$ is closed under colimits, as $\theta$ is fully faithful
  and commutes with colimits.

  Similarly to the proof of Theorem \ref{thm-ch-is-right-modules},
  the claim for general pointed presentable \cats follows from
  the properties of Lurie's tensor product.
  By \cite[4.8.1.17]{HA},
  $$\Fun(\bbZ\op,\scrC)\simeq\Fun(\bbZ\op,\spaces_*)\tensor\scrC$$
  and by \cite[4.8.4.6]{HA},
  $$\operatorname{RMod}_{\rr{Free}_{\bbE_1}(S^{0,1})}(\Gr\spaces_*)
  \tensor_{\Gr\spaces_*}\scrC
  \simeq \operatorname{RMod}_{\rr{Free}_{\bbE_1}(S^{0,1})}(\Gr\scrC)$$
  we thence have the desired claim.
\end{proof}

\begin{rem}\label{rem-mod-to-fil}
  We can describe more explicitly the equivalence of Theorem
  \ref{thm-fild-is-right-modules}.
  Let us first observe that the action of
  the free generator of $R_1 \coloneqq \rr{Free}_{\bbE_1}(S^{0,1})$ determines
  the action of all other spaces of $R_1$.
  By unraveling the definitions, specifying a map
  $t\colon S^{0,1}\tensor X\to X$ on a graded object $X$ is equivalent to
  specifying maps $X^n\to X^{n+1}$ for all $n\in\bbZ$.
  Hence, given an $R_1$-module, the associated filtered object is the one
  having as maps between consecutive objects the maps determined by $t$,
  and the actions of the $s$-th powers of the generator
  $t^s\colon S^{0,s}\tensor X\to X$ correspond to
  choices for all the possible $s$-fold composites for the structure
  maps of the resulting filtered object (i.e.\ maps
  $X^n\to X^{n+s}$ for all $n>1$). This description makes
  also clear the behavior of the functor going in the opposite direction.
\end{rem}

\subsection{The resolution sequence}\label{subsection-inductive-step}
We will now move to the construction of the sequence
(\ref{eq-sequence-of-algebras}). By Theorem
\ref{thm-fild-is-right-modules}, we can take
$R_1\simeq\rr{Free}_{\bbE_1}(S^{0,1})$.
We will define the first map of the sequence to be the one obtained by
killing the square of the free generator of $R_1$.

\begin{constr}\label{constr-r2}
  By \cite[4.1.1.18]{HA}, we have
  $\left(\rr{Free}_{\bbE_1}(S^{0,1})\right)^n\simeq S^n$ for all $n\geq 0$.
  In particular, the inclusion of pointed spaces
  $S^{0,2}\to\rr{Free}(S^{0,1})$ induces a map of $\bbE_1$-algebras
  $t^2\colon\rr{Free}_{\bbE_1}(S^{0,2})\to\rr{Free}_{\bbE_1}(S^{0,1})$.
  We define $R_2$ to be the $\bbE_1$-algebra in graded pointed spaces
  obtained as the pushout of $t^2$ along the augmentation map for the
  free algebra $\rr{Free}_{\bbE_1}(S^{0,2})$:
  \begin{equation}\label{eq-r2}
  \begin{tikzcd}
    \rr{Free}_{\bbE_1}(S^{0,2})\ar[r]\ar[d,"t^2"']& S^{0,0} \ar[d] \\
    \rr{Free}_{\bbE_1}(S^{0,1})\ar[r]& R_2
  \end{tikzcd}
  \end{equation}
  We pick the bottom horizontal map $\rr{Free}_{\bbE_1}(S^{0,1})
  \simeq R_1 \to R_2$ to be the first map in the sequence
  (\ref{eq-sequence-of-algebras}).
\end{constr}

\begin{rem}
  Notice that, as passing to homology preserves colimits, the pushout
  square (\ref{eq-r2}) induces on homology the square of Theorem
  \ref{thm:e1presentation} for the case $n=2$.
\end{rem}

In Lemma \ref{lemma-toda-massey-r2}, we will show that there exists a map
$S^{1,3}\to R_2$, inducing on homology the map
$\mathrm{Free}_{\bbE_1}(r_3)\to \Lambda^{(2)}$ of Theorem \ref{thm:e1presentation}.
Given such a map, we can consider the pushout diagram of $\bbE_1$-algebras
$$
\begin{tikzcd}
  \rr{Free}_{\bbE_1}(S^{1,3})\ar[r]\ar[d]& S^{0,0} \ar[d] \\
  R_2\ar[r]& R_3
\end{tikzcd}
$$
and pick the bottom horizontal map to be the second map in the sequence
(\ref{eq-sequence-of-algebras}).

Similarly, in Lemma \ref{lemma-toda-massey-r3},
we will show that there exists a map $S^{2,4}\to R_2$, inducing on homology
the map $\mathrm{Free}_{\bbE_1}(r_4)\to \Lambda^{(3)}$ of Theorem
\ref{thm:e1presentation}.

Given such maps, we can construct the rest of the sequence by induction.

\begin{constr}\label{constr-inductive-rn}
  Let us assume we are given a sequence
  $R_1\to R_2\to\cdots\to R_{n}$ with $n\geq 3$, such that every
  $R_j$ comes equipped with a map $S^{j-1,j+1}\to R_j$
  inducing on homology the map
  $\mathrm{Free}_{\bbE_1}(r_{j+1})\to \Lambda^{(j)}$ of
  Theorem \ref{thm:e1presentation},
  and that each map $R_{j-1}\to R_{j}$ for $2\leq j\leq n$
  fits in a commutative square of $\bbE_1$-algebras
  $$
  \begin{tikzcd}
    \rr{Free}_{\bbE_1}(S^{j-2,j})\ar[r]\ar[d]& S^{0,0} \ar[d] \\
    R_{j-1}\ar[r]& R_{j}.
  \end{tikzcd}
  $$
  Our goal is to produce inductively
  an $\bbE_1$-algebra $R_{n+1}$ and a map
  $S^{n,n+2}\to R_{n+1}$ inducing on homology the map
  $\mathrm{Free}_{\bbE_1}(r_{n+2})\to \Lambda^{(n+1)}$ of
  Theorem \ref{thm:e1presentation}.

  Let us denote by $R_{n+1}$ the object obtained pushing out
  the map $\rr{Free}_{\bbE_1}(S^{n-1,n+1})\to R_n$ we have by inductive
  hypothesis along the augmentation map of the free algebra:
  $$
  \begin{tikzcd}
    \rr{Free}_{\bbE_1}(S^{n-1,n+1})\ar[r]\ar[d]& S^{0,0} \ar[d] \\
    R_n\ar[r]& R_{n+1}.
  \end{tikzcd}
  $$
  We pick the bottom map of the above square to be the one for the
  sequence (\ref{eq-sequence-of-algebras}).
  By passing to homology, we get precisely the pushout square
  of Theorem \ref{thm:e1presentation}.
  By Lemma \ref{lem:pnhomology}, we have an element $r_{n+1}$
  of bidegree $(n+1,n-1)$ in
  $\Lambda^{(n)}\simeq\Lambda(e)\tensor\bbZ[r_{n+1}]$.
  By Hurewicz's Theorem, the cohomology class $r_{n+1}$ determines
  in an essentially unique way a map $S^{n,n+2}\to R_{n+1}$, which we will
  denote $\langle t\rangle^{n+2}$ and call the
  \emph{$(n+2)$-fold Toda power of $t$}.
  This completes the induction step needed to extend the sequence for
  $m\geq 4$.
\end{constr}

\begin{rem}\label{rem-recursive-module-structure}
  We can now understand the relations between $R_n$-modules for varying $n$:
  \begin{enumerate}
    \item It follows from Construction \ref{constr-r2} together with
      Remark \ref{rem-mod-to-fil} that the $R_1$-module associated to
      a filtered space $X$ can be given the structure of an $R_2$-module
      if and only if all the composites of pairs of consecutive maps of $X$
      are nullhomotopic, and specifying the structure of an
      $R_2$-module is equivalent to choosing a nullhomotopy for each
      such pair.
    \item It follows from Construction \ref{constr-inductive-rn} together with
      Remark \ref{rem-mod-to-fil} that given any $R_{n-1}$-module $C$,
      it can be given the structure of an $R_n$-module if and only if
      the map $\langle t\rangle^n\colon S^{n-2,n}\to\innmap{}_*(C,C)$ of
      Construction \ref{constr-inductive-rn} \footnote{or,
        in the case of $n=3$ (resp. $n=4$), the map
        $\langle t \rangle^3$ defined in Construction \ref{constr-t3}
        (resp. the map $\langle t\rangle^4$ defined in Construction
        \ref{constr-r3})}
      is trivial.
      By unraveling the definitions, we see that if we denote
      by $X$ the filtered space associated to $C$ (obtained by restriction of
      scalars along $R_1\to R_{n-1}$), the components of $\langle t\rangle^n$
      are maps
      $$S^{n-2}\to\Map_*(X^m,X^{m+n})$$
      for all $m\in\bbZ$; specifying the structure of an $R_n$-module on $C$
      is equivalent to choosing nullhomotopies for each such map.
  \end{enumerate}
\end{rem}

\begin{defi}\label{defi-toda-bracket}
  Let $X^0\to X^1\to\cdots\to X^{n}$ be a sequence of pointed spaces and
  maps between them, and assume that the filtered pointed space $X$ obtained
  by extending the given sequence by zeroes
  can be given the structure of an $R_{n-1}$-module.
  Let us moreover fix one such structure and denote the resulting object by
  $C_X$.
  We define the \emph{n-fold Toda bracket} of $C_X$ to be the class
  $$S^{n-2}\to\Map_*(X^0,X^n)$$
  given by the only nontrivial component of the map
  $\langle t\rangle^n\colon S^{n-2,n}\to\innmap{}_*(C_X,C_X)$ of Construction
  \ref{constr-inductive-rn}, or,
  in the case of $n=3$ (resp. $n=4$), the map
  $\langle t \rangle^3$ defined in Construction \ref{constr-t3}
  (resp. the map $\langle t\rangle^4$ defined in Construction
  \ref{constr-r3}).
\end{defi}

\begin{defi}
  Given any filtered pointed space $X\in\Fun(\bbZ\op,\spaces_*)$, we say
  it is a \emph{naive cochain complex} if its corresponding $R_1$-module
  can be given a structure of $R_2$-module. Given a naive cochain complex
  $X$, we say that it has \emph{uniformly trivial n-fold Toda
  brackets} if it can be recursively given the structure of an $R_n$-module
  after choosing $R_m$-module structures for all $m<n$.
\end{defi}

\begin{rem}
  As all pointed \cats are canonically enriched over pointed spaces, the
  definition of Toda brackets generalizes in an obvious way to sequences of
  maps in an arbitrary pointed \cat.
  In particular, it extends to sequences of maps in any stable \cat.
\end{rem}

\begin{prop}\label{lemma-colim-is-A}
  The $\bbE_1$-algebra $A$ of Theorem \ref{thm-ch-is-right-modules}
  is equivalent to the colimit
  $\colim_n R_n$ of the sequence (\ref{eq-sequence-of-algebras});
  i.e.\ for $\scrC$ pointed presentable with a special extremal separator,
  $$\coCh(\scrC)\simeq\operatorname{RMod}_{\colim_{\bbZ\op} R_n}
  (\Gr\scrC).$$
\end{prop}
\begin{proof}
  It follows from the resolution sequence (\ref{eq-sequence-of-algebras}),
  together with the associated sequence in homology, that the underlying
  graded space of $\colim_n R_n$ is equivalent to $S^{0,0}\amalg S^{0,1}$, as
  is also the underlying space of the $\bbE_1$-algebra $A$ of Theorem
  \ref{thm-ch-is-right-modules}. As $S^{0,0}\amalg S^{0,1}$ is 0-truncated,
  the $\bbE_1$-algebra structures on it are determined by the associative
  ring structures on the discrete object $S^{0,0} \amalg S^{0,1}$
  (living in the 1-category $\coCh(\mathbf{Top}_*)$).
  But, by a direct check, there exists only
  one nontrivial associative ring structure on $S^{0,0} \amalg S^{0,1}$.
  As $\colim_n R_n$ is augmented over $S^{0,0}$, it cannot have trivial
  multiplication; likewise, $A$ cannot be trivial
  by Theorem \ref{thm-ch-is-right-modules}. Therefore $A$ and $\colim_n R_n$
  must be equivalent as $\bbE_1$ algebras.
\end{proof}

\begin{rem}\label{rem-informal-toda}
  The above proposition can be informally rephrased by saying that the datum
  of a coherent cochain complex is equivalent to the datum of a naive
  chain complex together with choices for all the nullhomotopies of pairs
  of consecutive maps, and recursively defined choices of nullhomotopies for
  all possible $n$-fold Toda brackets, for all $n\geq 3$.
\end{rem}

\subsection{The resolution sequence for $n\leq 3$}\label{subsection-hard-cases}
\begin{notat}
Let $\scrC$ be a pointed \cat. Let $X=(X^n)_{n\in\bbZ}\in\Gr\scrC$.
Given any $s\in\bbZ$ and any $t\in\bbN$ will denote by
$\Omega^{t,s}$ the composite functor
$\Omega_\scrC^t\circ\{-s\} \simeq \{-s\}\circ\Omega_\scrC^t.$
That is,
$$(\Omega^{t,s}X)^n\simeq\Omega^t(X^{n+s}).$$
Notice that, for $S^{0,0}\in\Gr\spaces_*$, we have
$\Omega^{0,s}\left(S^{0,0}\right)\simeq S^{0,-s}$.
\end{notat}

\begin{rem}
  Let $X\in\coCh\spaces_*$; we have that:
  \begin{enumerate}
    \item For all $t\in\bbN$, $s\in\bbZ$,
      $$\innmap{}_*(S^{t,s},X)\simeq(\Omega^t X)\{-s\}\simeq \Omega^{t,s}X;$$
    \item For all $t\in\bbN$, $s\in\bbZ$,
      $$S^{t,s}\tensor_{\text{Day}}X \simeq (\Sigma^t X)\{s\}.$$
  \end{enumerate}
\end{rem}
  \begin{rem}\label{rem-r2-toda-situation}
    The structure map $\rr{Free}_{\bbE_1}(S^{0,1})\to R_2$ defined in
    Construction \ref{constr-r2} determines a
    degree $1$ square-zero element
    $t\in(\pi_0(R_2))^1$, whose action
    induces the following diagram of right $R_2$-modules:
    $$
    \begin{tikzcd}
      R_2 \ar[r, "t"] \ar[rr, bend right=40, "0"']
      \ar[rr, phantom, bend right=20, "\rotatebox{90}{$\Rightarrow$}"]
    & \Omega^{0,1}R_2 \ar[r, "t"]
    & \Omega^{0,2}R_2 \ar[r, "t"]
    & \Omega^{0,3}R_2. \ar[from=ll, bend left=40, "0"]
      \ar[from=ll, phantom, bend left=20, "\rotatebox{90}{$\Leftarrow$}"]
    \end{tikzcd}
    $$
  \end{rem}

  \begin{constr}\label{constr-t3}
    In the situation of Remark \ref{rem-r2-toda-situation},
    let us denote by $F_2$ the fiber (in $R_2$-modules) of multiplication
    by $t$, considered as a map $R_2\to\Omega^{0,1}R_2$.
    Choosing a nullhomotopy $\alpha\colon 0\Rightarrow t^2$
    is equivalent to chosing a factorization of $t$
    through $\Omega^{0,1}F_2$.
    By the following diagram (where all squares are Cartesian)
    $$
    \begin{tikzcd}
      R_2 \ar[r] \ar[rd, dotted, "\langle t\rangle^3"']
      & \Omega^{0,1}F_2 \ar[r]\ar[d]
      & \Omega^{0,1}R_2 \ar[d]\ar[from=ll, bend left=40, "t"]& \\
      & \Omega^{0,3} R_2\ar[r]\ar[d]
      & \Omega^{0,2}F_2\ar[d]\ar[r] & \pt\ar[d] \\
      & \pt \ar[r]& \Omega^{0,2}R_2\ar[r]& \Omega^{0,3}R_2 \\
    \end{tikzcd}
    $$
    we see that there exists a map
    $R_2\to\Omega^{1,3}R_2$ induced by such a factorization.
    We call this map the \emph{3-fold Toda power} of $t$,
    and denote it by $\langle t \rangle^3$.
    With a little abuse of terminology, we will reserve the same name
    for the map $S^{1,3}\to\innmap{}_*(R_2,R_2)$ obtained by adjoining over
    twice:
    $$
    \cfrac{R_2\to\innmap{}_*(S^{1,3},R_2)}
    {\cfrac{R_2\tensor S^{1,3}\to R_2}
    {S^{1,3}\to\innmap{}_*(R_2,R_2)}}
    $$
    As $R_2$ is a free $R_2$-module of rank $1$, we have that
    $\innmap{}_*(R_2,R_2)\simeq R_2$, hence the triple Toda power
    can be equivalently expressed as a map $S^{1,3}\to R_2$.
  \end{constr}

  \begin{lemma}\label{lemma-toda-massey-r2}
    The map $S^{1,3}\to R_2$ constructed above induces on homology the map
    $\rr{Free}_{\bbE_1}(r_3)\to \Lambda^{(2)}$ defined in
    Theorem \ref{thm:e1presentation}.
  \end{lemma}
  \begin{proof}
    Upon passing to homology degree-wise, using that
    $H(R_2)\simeq \Lambda^{(2)}$ in the notation of Theorem
    \ref{thm:e1presentation},
    and that, by Remark \ref{rem:explicitdga}, we can use
    $\wt{P}_2 \coloneqq
    \bbZ\langle e_1,e_2\rangle/(\partial e_2 = e_1^2)\simeq \Lambda^{(2)}$
    as a representative
    for the homology in the $1$-category of dg-modules,
    the defining diagram can be represented by the strictly commuting
    diagram of dg-modules
    $$
    \begin{tikzcd}[ampersand replacement =\&]
      \wt P_2 \ar[r, "\binom{e_1}{- e_2}"] \ar[rd]
      \& \Sigma^{0,-1}\wt P_2 \oplus \Sigma^{-1,-2}\wt P_2
        \ar{r}{\begin{pmatrix} 1 & 0 \end{pmatrix}}
        \ar{d}{\begin{pmatrix}e_2 & e_1\end{pmatrix}}
        \& \Sigma^{0,-1}\wt P_2 \ar[d, "\binom{e_1}{- e_2}"]\& \\[3.5em]
      \& \Sigma^{-1,-3}\wt P_2\ar[r,"\binom{0}{1}"]\ar[d]
      \& \Sigma^{0,-2}\wt P_2\oplus\Sigma^{-1,-3}\wt P_2
        \ar{d}{\begin{pmatrix}1 & 0\end{pmatrix}}\ar[r] \& 0\ar[d] \\[3.5em]
      \& 0 \ar[r]\& \Sigma^{0,-2}\wt P_2\ar[r,"e_1"]\& \Sigma^{0,-3}\wt P_2 \\
    \end{tikzcd}
    $$
    (where the negative signs are due to the direction of the homotopy,
    going \emph{from} $0$ \emph{to} $e_2$).
    Thus, the diagonal map is given by the action of $e_1e_2-e_2e_1$,
    which by Remark \ref{rem:explicitdga}
    represents exactly the homology class
    pointed by $r_3$ in the defining diagram for $\Lambda^{(3)}$.
  \end{proof}

  \begin{constr}\label{constr-r3}
    Let us now consider the pushout square
    $$
    \begin{tikzcd}
      \rr{Free}_{\bbE_1}(S^{1,3})\ar[r]\ar[d]& S^{0,0} \ar[d] \\
      R_2\ar[r]& R_3
    \end{tikzcd}
    $$
    where the vertical map is the one induced by the 3-fold Toda power
    costructed above.
    We pick the bottom horizontal map to be the second map in the sequence
    (\ref{eq-sequence-of-algebras}).

    By the commutativity of the following diagram of $R_3$-modules
    (where we know that $\langle t\rangle^3\simeq 0$ by the defining pushout for
    $R_3$)
    $$
    \begin{tikzcd}
      R_3 \ar[r] \ar[rd, "t"']
      & \Omega^{0,1}F_3 \ar[r]\ar[d]
      & \Omega^{1,3}R_3 \ar[d]\ar[from=ll, bend left=40, "0"] \\
      & \Omega^{0,1}R_3\ar[r]
      & \Omega^{0,2}F_3 \\
    \end{tikzcd}
    $$
    we see that $t$ factors through
    $L_3\coloneqq\fib(\Omega^{0,1}R_3\to\Omega^{0,2}F_3)$.
    If we denote by $G_3\coloneqq\fib(\Omega^{0,1}F_3\to \Omega^{1,3} R_3)$,
    we have the following diagram (where all squares are Cartesian)
    $$
    \begin{tikzcd}
      R_3 \ar[r] \ar[rd, dotted, "\langle t\rangle^4"']
      & L_3 \ar[r]\ar[d]
      & \Omega^{0,1}R_3 \ar[d]\ar[from=ll, bend left=40, "t"]& \\
      & \Omega^{2,4} R_3\ar[r]\ar[d]
      & \Omega^{0,1} G_3\ar[d]\ar[r] & \pt\ar[d] \\
      & \pt \ar[r]& \Omega^{0,2}F_3\ar[r]& \Omega^{1,4} R_3 \\
    \end{tikzcd}
    $$
    and we define the composite
    $$R_3\to L_3\to \Omega^{2,4}R_3$$ to be the
    \emph{4-fold Toda power} of $t$, denoted $\langle t\rangle^4$.
    As for the 3-fold case, the datum of a 4-fold Toda bracket
    uniquely determines a map $S^{2,4}\to R_3$ that we will refer to by the
    same name.
  \end{constr}

  \begin{lemma}\label{lemma-toda-massey-r3}
    The map $S^{2,4}\to R_3$ constructed above induces on homology the map
    $\rr{Free}_{\bbE_1}(r_4)\to \Lambda^{(3)}$ defined in
    Theorem \ref{thm:e1presentation}.
  \end{lemma}
  \begin{proof}
    As in Lemma \ref{lemma-toda-massey-r2}, we can consider the diagram
    of dg-modules obtained from the defining diagram of the $4$-fold Toda
    power by passing to homology:
    $$
    \begin{tikzcd}[ampersand replacement =\&, row sep=4em, column sep=1.3em]
      \wt{P}_3 \ar{d}{\begin{pmatrix}e_1 \\ - e_2 \\ e_3\end{pmatrix}}
        \ar[dd, bend right=80] \&\& \\
      \Sigma^{0,-1}\wt P_3 \oplus \Sigma^{0,-2}\wt P_2\oplus\Sigma^{-1,-3}\wt P_2
      \ar{r}{\begin{pmatrix}1 & 0 & 0\end{pmatrix}}
      \ar{d}{\begin{pmatrix}e_3 & e_2 & e_1\end{pmatrix}}
      \& \Sigma^{0,-1}\wt P_3
        \ar{d}{\begin{pmatrix}e_1 \\ - e_2 \\ -e_3\end{pmatrix}}\& \\
        \Sigma^{-2,-4} \wt{P}_3\ar{r}[swap]{\begin{pmatrix}0\\ 0\\ 1\end{pmatrix}}\ar[d]
      \& \Sigma^{0,-2}\wt P_2\oplus\Sigma^{-1,-3}\wt P_2 \oplus \Sigma^{-2,-4} \wt P_3
      \ar{d}{\begin{pmatrix}1 & 0 & 0 \\ 0 & 1 & 0\end{pmatrix}}\ar[r] \& 0\ar[d] \\
      0 \ar[r]
      \&\Sigma^{0,-2}\wt P_2\oplus\Sigma^{-1,-3}\wt P_2
      \ar{r}{\begin{pmatrix}e_2 & e_1\end{pmatrix}}\& \Sigma^{-1,-4} \wt P_3 \\
    \end{tikzcd}
    $$
    (again, the signs are consistent with the choice of having the nullhomotopies
    \emph{starting} at $0$)
    and we see that the $4$-fold Toda power induces on homology the map
    given by the action of
    $$e_1e_3 - e_2^2 + e_3e_1$$
    which by Remark \ref{rem:explicitdga}
    represents exactly the homology class
    pointed by $r_4$ in the defining diagram for $\Lambda^{(4)}$.
  \end{proof}

\section{Total homology and k-decompositions}
\label{section-decomp}

The Postnikov tower construction makes clear how a spectrum is
uniquely determined by its homotopy groups together with its k-invariants.
In particular, any spectrum $X$ uniquely determines co/fiber sequences of the
form
$$\tau_{[n,n+1]}X\to\pi_{n}X[n]\xto{k_{n}}\pi_{n+1}X[n+2]$$
for each $n\in\bbZ$, and it's easy to verify that the $k_n$'s are such that
$$k_{n}\circ k_{n-1}[-1]\simeq 0$$
for all $n\in\bbZ$.
It is widely known among the experts
that the above observation admits a converse precisely when
the nullhomotopies for such pairwise compositions of
``truncated k-invariants'' are suitably compatible
(see e.g. \cite[Section 4]{SagaveToda}, where an instance of this fact
is discussed for objects with finite Postnikov filtrations in the setting of
triangulated categories);
to be precise, a spectrum is uniquely determined by the datum of its homotopy
groups together with a collection of maps
$\{k_{n}\colon\pi_{n}X[n]\to\pi_{n+1}X[n+2]\}_{n\in\bbZ}$ such
that all pairs of composable maps compose to a nullhomotopic one,
and all the possible higher Toda brackets are trivial.

In this section we provide a rigorous formulation of the idea discussed above
using the language of coherent cochain complexes developed in the previous
sections, generalizing it to the context of stable \cats satisfying some
mild hypothesis\footnote{that is, admitting sequential limits, sequential
colimits and equipped with a right complete t-structure}
and general homotopy objects in (ordinary) Abelian categories.
The process of recostructing an object from its decomposition in homotopy
objects and maps between their shifts will be given by the \emph{total homology}
construction (see \ref{defi-tothom}), a concept further investigated in the next
section, in relation to the spectral sequence generated by a coherent cochain
complex.

Let us begin by constructing an \cat that encapsulates
the data needed to reconstruct the objects from their homotopy groups and their
k-invariants.

\begin{defi}
  Let $\scrC$ be a stable \cat equipped with a
  t-structure.
  We will denote by $\Ch_{\scrC}^{+}(\scrC^\heartsuit)$ the full
  subcategory of $\coCh\scrC$ spanned by those objects $C$ such that
  $C^n\in\scrC^\heartsuit[2n]$ for all $n\in\bbZ$, and refer to it as the
  \emph{\cat of degenerate cochain complexes} of $\scrC$.
\end{defi}

\begin{rem}
  The notation $\Ch_\scrC^+(\scrC^\heartsuit)$ is slightly abusive, as the
  object depends also on the t-structure $\scrC$ is equipped with.
\end{rem}

We can now construct a functor that incarnates the idea discussed in the
introduction of this section.

\begin{rem}\label{rem-factor-tower}
  Let $\scrC$ be a stable \cat equipped with a right complete t-structure.
  By the definition of right completeness, there exists an equivalence
  $$
  \scrC\simeq\lim\left(\cdots\to
  \left(\scrC\right)_{\geq n}\xto{\tau_{\geq n+1}}
  \left(\scrC\right)_{\geq n+1}\to\cdots\right).
  $$
  By the right completeness of $\cFild\scrC$ (see Remark
  \ref{rem-coch-same-completeness}, together with Theorem
  \ref{thm-general-equivalence}), the fully faithful
  inclusion $\scrC\hookrightarrow\Fild(\scrC)$ factors through
  the subcategory of complete objects:
  $$
  \wtrunc{\istar}\colon\scrC\hookrightarrow\cFild(\scrC).
  $$
\end{rem}

\begin{defi}\label{defi-k-decomp}
  Let $\scrC$ be a stable presentable \cat equipped with a right complete
  t-structure.
  Notice that, by Remark \ref{rem-factor-tower},
  the composite $\caush\circ\wtrunc\istar\colon\scrC\to\coCh(\scrC)$
  factors through the inclusion
  $\Ch_\scrC^+(\scrC^\heartsuit)\subset\coCh(\scrC)$.
  We define the \emph{k-decomposition} functor
  \begin{displaymath}
  \begin{split}
    \kinv\colon\scrC&\to\Ch_\scrC^+(\scrC^\heartsuit)\\
    C&\mapsto(\cdots\to\pi_{n-1}X[2n-2]\to\pi_nX[2n]\to\pi_{n+1}[2n+2]\to\cdots)
  \end{split}
  \end{displaymath}
  to be such factorization of
  $\caush\circ\wtrunc{\istar}$ through $\Ch_\scrC^+(\scrC^\heartsuit)$.
  We refer to the differentials of the complex $\kinv C$ as the
  \emph{k-invariants} of $C$.
\end{defi}

\begin{ex}
  Let $R$ be an ordinary commutative ring, and let $\rrD(R)$ denote its
  derived category (considered as a stable \cat equipped with the usual
  t-structure).
  Given any object $X\in\rrD(R)$, its k-decomposition $\kinv X$ is given
  by the coherent cochain complex
  $$\cdots\to (H^{-n+1}X)[2n-2]\to(H^{-n} X)[2n]\to(H^{-n-1})X[2n+2]\to\cdots$$
  with $H^{-n}X \in \Ab$.
\end{ex}

\begin{ex}\label{ex-ch-spectra}
  Given any $X\in\Sp$, its k-decomposition $\kinv X$ is given by the
  coherent cochain complex
  $$\cdots\to(\pi_{n-1}X)[2n-2]\to(\pi_{n}X)[2n]
  \to(\pi_{n+1}X)[2n+2]\to\cdots$$
  with $\pi_n X \in \Ab$.
  The equivalence $\Ch_\Sp^+(\Ab)\simeq\Sp$ can be interpreted as saying
  that any spectrum can be expressed as a coherent chain complex
  of (suitably shifted) Abelian groups, and, viceversa, one can construct
  a spectrum from the datum of objects in $\bigoplus_\bbZ \Ab[2n]$ and maps
  between consecutive objects $A_n[2n]\to A_{n+1}[2n+2]$ such that all possible
  Toda brackets between them are trivial.
\end{ex}

\begin{rem}
  The previous two examples make clear how, from the point of view of
  coherent chain complexes, the only difference between an object in
  $\rrD(\bbZ)$ and a spectrum lies in the $\bbZ$-linearity of the k-invariants:
  $$\rrD(\bbZ)\simeq\Ch_{\rrD(\bbZ)}^+(\Ab)\subset\Ch_\Sp^+(\Ab)\simeq\Sp.$$
  Of course, this is just a consequence of a theorem of Shipley's
  (see \cite{shipleyHZ}) showing that the derived \cat $\rrD(R)$ of a discrete
  commutative ring $R$ is equivalent to the \cat of $HR$-modules
  $\Mod_{HR}$.
\end{rem}

We now turn our attention to the functor inverse to $\kinv$. It turns out that
the process reconstructing an object of $\scrC$ from its k-decomposition
is a particular case of the more general construction assigning to every
coherent cochain complex $C$ the object underlying its piled-up filtration
$\imp C$.

\begin{defi}\label{defi-tothom}
  Let $\scrC$ be a stable \cat with sequential limits and sequential colimits.
  We define the \emph{total homology} functor to be the composite functor
  $\tothom\coloneqq\colim\circ\imp$
  \begin{displaymath}
  \begin{split}
    \tothom\colon\coCh\scrC&\to\scrC\\
    C&\mapsto(\imp C)^{-\infty}.
  \end{split}
  \end{displaymath}
\end{defi}

\begin{thm}\label{thm-k-inv-tothom-inverses}
  Let $\scrC$ be a stable \cat with sequential limits and sequential colimits,
  equipped with a right complete t-structure.
  The restriction of $\tothom$ to $\Ch_\scrC^+(\scrC^\heartsuit)$
  induces an equivalence
  $$\Ch_\scrC^+(\scrC^\heartsuit)\xto{\sim}\scrC$$
  whose inverse is given by the k-decomposition functor of Definition
  \ref{defi-k-decomp}.
\end{thm}
\begin{proof}
  By the definition of right completeness, there exists an equivalence
  $$
  \scrC\simeq\lim\left(\cdots\to
  \left(\scrC\right)_{\geq n}\xto{\tau_{\geq n+1}}
  \left(\scrC\right)_{\geq n+1}\to\cdots\right).
  $$
  By the discussion in \cite{HA} right before Proposition 1.2.1.17, we can
  give an alternative description of the right hand side as the full
  subcategory of $\Fild(\scrC)$ spanned by those filtered objects for
  which
  \begin{enumerate}
    \item For each $n\in\bbZ$, $F^n\in\scrC_{\geq n}$;
      \label{item-defi-whitehead-one}
    \item For each $m\geq n$, the associated map $F^m\to F^n$ induces an
      equivalence $F^m\xto{\sim}\wtrunc{m}F^n$.
      \label{item-defi-whitehead-two}
  \end{enumerate}
  We first notice that, as long as (\ref{item-defi-whitehead-one}) holds,
  we can replace (\ref{item-defi-whitehead-two})
  with the weaker assumption that
  \begin{enumerate}
    \item[($\spadesuit$)] For each $n\in\bbZ$, the associated map
      $F^{n+1}\to F^n$ induces an equivalence
      $F^{n+1}\xto{\sim}\wtrunc{n+1}F^n$.
  \end{enumerate}
  By induction, assume that for a fixed $n$ and some $m\geq n$
  $F^m\to F^n$ induces an equivalence $F^m\to \wtrunc{m} F^n$. We want to prove
  that $F^{m+1}\to F^n$ induces an equivalence $F^{m+1}\to\wtrunc{m+1}F^n$.
  By ($\spadesuit$), $F^{m+1}\to F^m$
  induces an equivalence $F^{m+1}\to\wtrunc{m+1}F^m$, and by our inductive
  hypothesis
  $$\wtrunc{m+1}F^m\to\wtrunc{m+1}\circ\wtrunc{m}F^n\simeq\wtrunc{m+1}F^n$$
  is an equivalence, hence the composite map $F^{m+1}\to\wtrunc{m+1}F^m\to
  \wtrunc{m+1}F^n$
  (which, by (\ref{item-defi-whitehead-one}), is precisely the image
  under $\wtrunc{m+1}$ of the composite $F^{m+1}\to F^m\to F^n$)
  is an equivalence as desired.

  As $\wtrunc{\istar}$ is fully faithful, it induces an equivalence with its
  essential image in $\cFild(\scrC)$ (see Remark \ref{rem-factor-tower}).
  As $\caush$ is an equivalence, it is sufficient to check that
  $\imp\left(\Ch_\scrC^+\scrC^\heartsuit\right)\simeq\wtrunc{\istar}\scrC$.
  By the definition of $\Ch_\scrC^+\scrC^\heartsuit$ and Lemma
  \ref{lemma-aush-graded-same}, all objects in
  $\imp\left(\Ch_\scrC^+\scrC^\heartsuit\right)$ satisfy
  (\ref{item-defi-whitehead-one}) and
  (\ref{item-defi-whitehead-two}).
  On the other hand, again by Lemma \ref{lemma-aush-graded-same},
  any filtered object satisfying conditions
  (\ref{item-defi-whitehead-one}) and
  (\ref{item-defi-whitehead-two})
  is such that its shelled complex lies
  in $\Ch_\scrC^+\scrC^\heartsuit$. Hence,
  $\imp\left(\Ch_\scrC^+\scrC^\heartsuit\right)$ and
  $\wtrunc{\istar}\scrC$ are two full subcategories of $\cFild\scrC$ spanned
  by the same objects.
  To identify its inverse, it now suffices to notice that
  \begin{displaymath}
  \begin{split}
    \tothom\circ\kinv
    &\simeq \colim\circ\kern.25em\imp\circ\caush\circ\wtrunc{\istar}\\
    &\simeq \colim\circ\kern.25em\id\circ\wtrunc{\istar}\\
    &\simeq \id
  \end{split}
  \end{displaymath}
  as desired.
\end{proof}

We now turn our attention to the behavior of $\tothom$ on objects coming from
the (ordinary) category of chain complexes in the heart.

\begin{prop}\label{prop-tothom-eilenberg}
  Let $\scrC$ be as in Theorem \ref{thm-k-inv-tothom-inverses}.
  Then the t-structure homotopy objects of the total homology of an ordinary
  chain complex in $\oC \in \coCh(\scrC^\heartsuit)$ (considered as an
  element of $(\coCh\scrC)^\heartsuit$) are given by the cohomology of
  $\oC$:
  $$\pi_{-n}\tothom\oC\simeq \rrH^n\oC.$$
\end{prop}
\begin{proof}
  By Proposition \ref{prop-gr-of-imp},
  together with the hypothesis that $\oC$ lies in
  $\coCh(\scrC^\heartsuit)$, we have that
  $\gr^n\imp\oC\simeq \oC^n[-n]$; in particular, (as $\imp\oC$
  is complete) we see that for any integer $n$,
  $\imp\oC^n$ is $(-n)$-coconnective, and that
  \begin{equation}\label{eq-pi-tothom}
    \pi_m\left(\imp\oC^n\right) \cong
    \pi_m\left(\imp\oC^{n-1}\right) \text{ for all } m\leq-n.
  \end{equation}
  By Remark \ref{rem-intermediate-subquotients}, we have a co/fiber sequence
  $$\oC^{n+1}[-n-1]\to \imp \oC^n/\imp \oC^{n+2} \to \oC^n[-n]$$
  whose associated long exact sequence on homotopy objects lets us identify
  $$
  \pi_m\left(\imp \oC^n/\imp \oC^{n+2}\right)\cong
  \begin{cases}
    \ker(d^n_{\oC}) \quad &\text{ for } m=-n;\\
    \coker(d^n_{\oC}) \quad &\text{ for } m=-n-1;\\
    0 \quad &\text{ else.}
  \end{cases}
  $$
  Again by Remark \ref{rem-intermediate-subquotients},
  we have a co/fiber sequence
  $$
  \imp \oC^n/\imp \oC^{n+2}\to\imp \oC^{n-1}/\imp \oC^{n+2}\to \oC^{n-1}[-n+1]
  $$
  whose associated long exact sequence starting at $-n+1$ looks as follows
  $$
  \cdots 0 \to \pi_{-n+1}\left(\imp \oC^{n-1}/\imp \oC^{n+2}\right)
    \to \oC^{n-1} \to \ker(d^n) \to
    \pi_{-n}\left(\imp \oC^{n-1}/\imp \oC^{n+2}\right) \to 0 \cdots
  $$
  letting us identify $\pi_{-n}\left(\imp \oC^{n-1}/\imp \oC^{n+2}\right)$ as
  $$\coker \left(\oC^{n-1}\xto{d^{n-1}} \ker(d^n)\right) \cong \rrH^n\oC.$$
  As $\imp \oC^{n+2}$ is $(-n-2)$-coconnective, the long exact sequence
  associated to the co/fiber sequence
  $$\imp \oC^{n+2}\to\imp\oC^{n-1}\to\imp \oC^{n-1}/\imp \oC^{n+2}$$
  shows that
  $$\pi_m\left(\imp\oC^{n+2}\right)\cong\left(\pi_m\imp\oC^{n-1}\right)
  \text{ for }m\geq-n-1$$
  and thus in particular that $\pi_{-n}\left(\imp\oC^{n-1}\right)
  \cong\rrH^n\oC$.
  Finally, (\ref{eq-pi-tothom}) implies the desired result.
\end{proof}

\begin{rem}\label{rem-tothom-eilenberg-abelian}
  In the special case of $\scrC=\Sp$, the proof of Proposition
  \ref{prop-tothom-eilenberg} shows that $\imp\oC$ for an ordinary chain
  complex of Abelian groups $\oC$ gives precisely the tower obtained through
  the ``brutal truncations''\footnote{sometimes referred to also as the
  ``stupid truncation'' in the literature} of the complex $\oC$. That is,
  $$\imp\oC^n\simeq H(\tau^{\geq n}\oC),$$
  where $\tau^{\geq n}$ here denotes the brutal truncation (1-)functor and $H$
  denotes the Eilenberg--MacLane ($\infty$-)functor $\coCh(\Ab)\to\Sp$.

  In particular, we have that in the case of spectra the
  total homology functor $\tothom$ restricted to $\coCh(\Ab)\subset\coCh(\Sp)$
  coincides with the Eilenberg--MacLane functor.
\end{rem}

We conclude the section with the following variant of Definition
\ref{defi-k-decomp}.

\begin{variant}\label{variant-k-decomp}
  Let $\scrC$ be a stable presentable \cat equipped with a right separated
  t-structure, and let $q\in\bbN$ be fixed.
  Let $\Ch_{\scrC}^{+q}(\scrC^\heartsuit)$ denote the full
  subcategory of $\coCh\scrC$ spanned by those objects $C$ such that
  $C^n\in\scrC^\heartsuit[n(q+1)]$ for all $n\in\bbZ$.
  Consider the full subcategory $\scrC_{q\text{-periodic}}
  \subset\scrC$
  spanned by the objects $X$ such that the t-structure homotopy objects
  $\pi_n X$ are isomorphic to $0$ for $n$ not a multiple
  of $q$:
  $$
  \scrC_{q\text{-periodic}}\coloneqq\{X \ | \ \pi_n X \cong 0 \text { for }
  n\not\equiv 0 \mod q\}.
  $$
  We define $\kinv_q$ to be the factorization of
  $\caush\circ\wtrunc{q\istar}$
  (where $\wtrunc{q\istar}$ denotes the sub-filtration of $\wtrunc{\istar}$
  obtained by skipping all the stages of the Whitehead tower that are
  not multiples of $q$)
  through the inclusion
  $\Ch_\scrC^{+q}(\scrC^\heartsuit)\subset\coCh(\scrC)$,
  and refer to it as the
  \emph{$q$-periodic k-decomposition} functor.
  Similarly to what happens for $\kinv$, the functor $\kinv_q$ induces
  an equivalence $\scrC_{q\text{-periodic}}\simeq\Ch_\scrC^{+q}
  (\scrC^\heartsuit)$.
\end{variant}

We learned about the following example from Achim Krause.
\begin{ex}
  As an instance of Variant \ref{variant-k-decomp}, we can consider
  the $2(p^n-1)$-periodic k-decomposition of the $n$-th Morava
  K-theory ($n\geq1$) spectrum, for some fixed prime $p$.
  As
  $$\bbF_p[2m(p^n-1)+m]\simeq\bbF_p[m(p^n-1)]$$
  the coherent cochain complex $\kinv_{2(p^n-1)}K(n)$ looks as follows
  $$\cdots\to\bbF_p[(m-1)(p^n-1)]\to\bbF_p[m(p^n-1)]\to
  \bbF_p[(m+1)(p^n-1)]\to\cdots$$
  (where $\bbF_p[m(p^n-1)]$ sits in degree $m$), and the differentials are
  given by suitable shifts of the $n$-th Milnor primitive for the mod $p$
  Steenrod algebra:
  $$\partial^m\simeq Q_n[m]\in\calA_p.$$
\end{ex}

\section{The spectral sequence associated to a coherent cochain complex}
\label{section-spseq}

In this section, we discuss how coherent cochain complexes give rise to
spectral sequences.
We have the following result.

\begin{thm}\label{thm-specseq-of-coch}
  Let $\scrC$ be a stable \cat with sequential limits and sequential colimits,
  equipped with a right complete t-structure.
  Then, every coherent cochain complex $C\in\coCh\scrC$ generates a spectral
  sequence
  $$E_1^{i,j}\cong\pi_{-j}C^i$$
  whose $E_1$ page is given by the homotopy groups of the components of
  $C$ and whose $E_1$ differentials
  $$d_1^{i,j}=\pi_{-j}\partial^i_C$$
  are obtained from the coherent differentials of $C$ by passing to
  homotopy.
  When the spectral sequence collapses at a finite stage, it converges
  strongly to the homotopy groups of the total homology of $C$
  $$E_1^{i,j}\cong\pi_{-j}C^i \Longrightarrow \pi_{-i-j}\tothom C.$$
\end{thm}

We defer the proof to later in this section.
Of course, the above theorem follows at once from Theorem
\ref{thm-general-equivalence} together with Theorem \ref{thm-beil-recoll}
and the existence of the spectral sequence
associated to a filtered object in a suitable stable \cat (whose incarnation in
the case of spectra has been known by the experts for a long time and an
account of which can be found in \cite[1.2.2]{HA}, in the generality of this
paper).
The goal of this section is to give a self-contained construction of the above
spectral sequence.

\begin{rem}\label{rem-convergence-meaning}
  The convergence statement in Theorem \ref{thm-specseq-of-coch}
  is far from being optimal; it is somewhat of a folk result that the
  corresponding spectral sequence for a filtered object is conditionally
  convergent. In forthcoming work by Hedenlund--Krause--Nikolaus, the authors
  give a new proof of this fact, using the filtered counterpart of
  the d\'{e}calage construction we illustrate in this section using the
  coherent cochain complexes perspective.
\end{rem}

We learned about the relation between the Beilinson t-structure and Deligne's
d\'{e}calage functor from Benjamin Antieau.
Most of the ideas discussed in this section are already
present in some form in \cite{An2019periodic} and \cite[5.5]{BMS2};
we believe that the language of coherent cochain complexes gives a
particularly pleasant perspective on the topic.

The key ingredient in the construction of the spectral sequence of a coherent
cochain complex is given by the following construction.

\begin{defi}\label{defi-decalage}
  Let $\scrC$ be a stable \cat with sequential limits and sequential colimits,
  equipped with a right separated t-structure, and
  let $C\in\coCh\scrC$ be a coherent cochain complex. We can apply the functor
  $\kinv$ levelwise, to obtain a bicomplex (denoted $\kinv^\text{lvl}C$)
  $$
  \begin{tikzcd}
    \cdots \ar[r] & \kinv C^0 \ar[r] \ar[d, phantom, "\rotatebox{90}{=}"]
                  & \kinv C^1 \ar[r] \ar[d, phantom, "\rotatebox{90}{=}"]
                  & \kinv C^2 \ar[r] \ar[d, phantom, "\rotatebox{90}{=}"]
                  & \cdots \\
    & \vdots \ar[d]& \vdots \ar[d] & \vdots \ar[d] & \\
    \cdots \ar[r] & \pi_0C^0[0] \ar[r] \ar[d]
                  & \pi_0C^1[0] \ar[r] \ar[d]
                  & \pi_0C^2[0] \ar[r] \ar[d] & \cdots \\
    \cdots \ar[r] & \pi_1C^0[2] \ar[r] \ar[d]
                  & \pi_1C^1[2] \ar[r] \ar[d]
                  & \pi_1C^2[2] \ar[r] \ar[d] & \cdots \\
    \cdots \ar[r] & \pi_2C^0[4] \ar[r] \ar[d]
                  & \pi_2C^1[4] \ar[r] \ar[d]
                  & \pi_2C^2[4] \ar[r] \ar[d] & \cdots \\
    & \vdots & \vdots & \vdots &
  \end{tikzcd}
  $$
  i.e.\ such that $(\kinv^\text{lvl}C)^{i,j}\simeq\pi_jC^i[2j]$,
  where the horizontal differentials are the k-invariants for the relevant
  objects, and the vertical ones are induced from the differentials of $C$.
  We define the
  \emph{d\'{e}cal\'{e}e complex}\footnote{our choice of terminology
    is justified by the construction of the spectral sequence of
    Theorem \ref{thm-specseq-of-coch} (see the proof at the end of the section):
    the functor $\dec$ provides an
    incarnation in coherent chain complexes of Deligne's
    \emph{d\'{e}cal\'{e}e} filtration (see \cite[1.3]{DelHodge2})}
  denoted $\dec C$ to be the levelwise
  total homology of the bicomplex $\kinv^\text{lvl} C$ with respect to the
  horizontal maps of the above diagram.
  That is,
  $$
  \begin{tikzcd}
    \vdots \ar[d]&[-2em]&[-2em]& \vdots \ar[d]& \vdots \ar[d] &\vdots \ar[d]&\\
    \dec C^0 \ar[d]&=&
      \tothom\bigg(\cdots \ar[r] & \pi_0C^0[0] \ar[r] \ar[d]
      & \pi_0C^1[0] \ar[r] \ar[d]
      & \pi_0C^2[0] \ar[r] \ar[d] & \cdots\bigg) \\
    \dec C^1\ar[d]&=&
      \tothom\bigg(\cdots \ar[r] & \pi_1C^0[2] \ar[r] \ar[d]
                  & \pi_1C^1[2] \ar[r] \ar[d]
                  & \pi_1C^2[2] \ar[r] \ar[d] & \cdots\bigg) \\
    \dec C^2\ar[d]&=&
      \tothom\bigg(\cdots \ar[r] & \pi_2C^0[4] \ar[r] \ar[d]
                  & \pi_2C^1[4] \ar[r] \ar[d]
                  & \pi_2C^2[4] \ar[r] \ar[d] & \cdots\bigg) \\
    \vdots&&& \vdots & \vdots & \vdots &
  \end{tikzcd}
  $$
  or, more precisely:
  $$\dec \coloneqq \coCh(\tothom)\circ\kinv^\text{lvl}.$$
  As $\dec$ gives an endofunctor for $\coCh(\scrC)$, we will denote
  by
  $$\dec^n\colon\coCh(\scrC)\to\coCh(\scrC), \text{ with } n\geq 1$$
  its iterations.
  We will also use the convention $\dec^0\coloneqq\id$.
\end{defi}

\begin{rem}
  In the situation of Definition \ref{defi-decalage}, using
  Remark \ref{rem-comparison-formulae} together with
  Remark \ref{rem-tothom-eilenberg-abelian} we see that (omitting as usual
  the Eilenberg--Maclane functor from the notations)
  \begin{displaymath}
  \begin{split}
    (\dec C)^n
    &\simeq \tothom\left(\left(\pi_n^\text{lvl}C\right)[2n]\right)\\
    &\simeq \colim\imp\left(\left(\pi_n^\text{lvl}C\right)[2n]\right)\\
    &\simeq \left(\colim\imp\left(\pi_n^\text{lvl}C\right)\right)[2n]\\
    &\simeq \left(\tothom\left(\pi_n^\text{lvl}C\right)\right)[2n]\\
    &\simeq \left(\pi_n^\text{lvl}C\right)[2n]\\
    &\simeq \left(\pi_n^B\imp C\right)[2n]\\
  \end{split}
  \end{displaymath}
  that is, $\dec C^n$ is equivalent (up to a shift) to the $n$-th Beilinson
  homotopy chain complex of the piled up filtered object of $C$.
\end{rem}

\begin{rem}\label{rem-dec-same-tothom}
  In the situation of Definition \ref{defi-decalage},
  we have that
  $\tothom C \simeq \tothom\dec C$
  for any $C\in\coCh\scrC$, and thus
  $$\tothom C\simeq \tothom\dec^n C$$
  for any $n\in\bbN$.
  To see this, let us consider the bifiltered object
  $$\imp^2\left(\kinv^\text{lvl}C\right)\coloneqq
  \cFild(\imp)\circ\imp\left(\kinv^\text{lvl}C\right)
  \simeq\imp\circ\coCh(\imp)\left(\kinv^\text{lvl}C\right)$$
  associated to the levelwise k-decomposition of $C$ (where we are using the
  fact that the diagram
  $$
  \begin{tikzcd}
    \coCh(\coCh\scrC) \ar[r, "\coCh(\imp)"]\ar[d, "\imp"]&
    \coCh(\cFild\scrC) \ar[d, "\imp"]\\
    \cFild(\coCh\scrC) \ar[r, "\cFild(\imp)"]& \cFild(\cFild\scrC)
  \end{tikzcd}
  $$
  commutes).
  It follows from Theorem \ref{thm-general-equivalence} that given any
  $F\in\cFild(\scrC)$, we have an equivalence of bifiltered objects
  $$\tau^B_{\geq \istar}F\simeq\imp\caush\tau^B_{\geq \istar}F
  \simeq\imp\left(\pi^B_{*} F[2*]\right) \quad \in
  \cFild\left(\cFild(\scrC)\right).$$
  In particular, for $F\simeq\imp C$, we have
  $$\tau^B_{\geq \istar}\imp C
  \simeq\imp\left(\pi^B_{*} \imp C[2*]\right)
  \simeq\imp\left(\imp\pi^\text{lvl}C[2*]\right)
  \simeq\imp^2\left(\kinv^\text{lvl}C\right).
  $$
  Now, we can use the above observation to deduce our initial claim:
  \begin{displaymath}
  \begin{split}
    \tothom C
    &\simeq \colim_n \imp C^n\\
    &\simeq \colim_{m,n} \left(\tau^B_{\geq m}\imp C\right)^n\\
    &\simeq \colim_{m,n} \imp^2\left(\kinv^\text{lvl}C\right)^{m,n}\\
    &\simeq \colim_m \imp\left(
      \colim_n \imp(\pi^\text{lvl}_{*} C[2*])^n
      \right)^m\\
    &\simeq \colim_m \imp\left(\dec C\right)^m\\
    &\simeq \tothom \dec C.
  \end{split}
  \end{displaymath}
\end{rem}

\begin{proof}[Proof of Theorem \ref{thm-specseq-of-coch}]
  Given any $n\in\bbZ$, we have that $\pi_n C^\bullet\in
  \coCh(\scrC^\heartsuit)$
  is an ordinary chain complex.
  Let
  $$\pi_{-j}\left(\dec^nC\right)^i\coloneqq E_{n+1}^{(n+1)i+nj,-ni+(-n+1)j}$$
  we want to prove that the above objects define the pages of a spectral
  sequence.
  We start by defining the differentials of the $E_1$ page to be
  $$d_1^{i,j}\coloneqq\pi_{-j}(\partial^i_C)$$
  Using Proposition \ref{prop-tothom-eilenberg}, we have
  \begin{displaymath}
  \begin{split}
    \pi_{-j}(\dec C)^i
    &\cong \pi_{-j}\tothom \left(\pi_i^\text{lvl} C\right) [2i]\\
    &\cong \pi_{-j-2i}\tothom \left(\pi_i^\text{lvl} C\right)\\
    &\cong \rrH^{j+2i}\left(\pi_i^\text{lvl} C\right);
    \end{split}
  \end{displaymath}
  the above isomorphism allows us to define
  $$d_2^{2i+j,-i}\coloneqq\pi_{-j}(\partial^i_{\dec C})$$
  and it's easy to check that $d_2$ has the correct bidegree.
  As
  $$
    \pi_{-j}\left(\dec^nC\right)^i
    \cong \rrH^{j+2i}\left(\pi_i^\text{lvl}\dec^{n-1}C\right)
  $$
  we can now proceed by induction to show that the $E_{n+1}$ page is given by the
  cohomology of the $E_{n}$ page: using that
  \begin{displaymath}
  \begin{split}
    \pi_i(\dec^{n-1}C)^{j+2i}
    &\coloneqq E_n^{n(j+2i)+(n-1)(-i), (-n+1)(j+2i)+(-n+2)(-i)}\\
    &\cong E_n^{(n+1)i+nj,-ni+(-n+1)j}
  \end{split}
  \end{displaymath}
  we see that
  \begin{displaymath}
  \begin{split}
    E&\vphantom{x}_{n+1}^{(n+1)i+nj , -ni+(-n+1)j}
    \coloneqq \pi_{-j}\left(\dec^nC\right)^i \\
    &\cong \rrH^{j+2i} \left(\pi_i^\text{lvl}\dec^{n-1}C\right) \\
    &\cong \frac{\ker\left(\pi_i(\dec^{n-1}C)^{j+2i}\to
    \pi_i(\dec^{n-1}C)^{j+2i+1}\right)}
    {\im\left(\pi_i(\dec^{n-1}C)^{j+2i-1}\to\pi_i(\dec^{n-1}C)^{j+2i}\right)} \\
    &\cong \frac{\ker\left(E_n^{(n+1)i+nj,-ni+(-n+1)j}\to
    E_n^{(n+1)i+n(j+1),-ni+(-n+1)(j+1)}\right)}
    {\im\left(E_n^{(n+1)i+n(j-1),-ni+(-n+1)(j-1)}\to
    E_n^{(n+1)i+nj,-ni+(-n+1)j}\right)} \\
    &\cong \rrH^{(n+1)i+nj,-ni+(-n+1)j}E_n
  \end{split}
  \end{displaymath}
  concluding the existence proof.

  The convergence statement follows at once from Remark \ref{rem-dec-same-tothom}.
\end{proof}

\section{Examples}
\label{section-examples}
In what follows, we present some results which appeared recently in the
literature using the language developed in the rest of the paper.
The aim of this section is to give a new perspective on known results, hence
we make no claims of originality regarding the results presented; still, we
believe that looking at these results through the lenses of
coherent cochain complexes buys some insight about these results and adds to
the clarity of their statement. In fact, some of the examples discussed here
motivated the investigation of the formalism of coherent chain complexes in
the first place.

\begin{ex} [\cite{An2019periodic}]
  Let $k$ be a commutative ring, and let $\rr{Sch}^{\natural}_k$ denote
  the category of qcqs $k$-schemes whose cotangent complex $\mathrm{L}_{X/k}$
  has Tor-amplitude concentrated in $[0,1]$. Then
  \begin{enumerate}
    \item there exists a functor $\cHCm(-/k)\colon
      \rr{Sch}^{\natural, \mathrm{op}}_k\to\coCh(\rrD(k))$
      sending each scheme $X$ to a coherent cochain complex $\cHCm(X/k)$
      such that
      \begin{enumerate}
        \item the total homology of $\cHCm(X/k)$ is given by negative cyclic
          homology:
          $$\tothom(\cHCm(X/k)) \simeq \rr{HC}^{-}(X/k)$$
        \item the components of $\cHCm(X)$ are given up to a shift by
          truncations of the Hodge-completed derived de Rham complex:
          $$\cHCm(X/k)^n \simeq \widehat{\rrL\Omega}^{\geq n}_{X/k}[3n]$$
        \item the coherent cochain complex $\cHCm(X/k)$ generates a conditionally
          convergent spectral sequence
          $$E_1^{i,j}\cong H^{j+3i}
          \left(\widehat{\rrL\Omega}^{\geq i}_{X/k}\right)\Longrightarrow
          \pi_{-i-j} \rr{HC}^{-}(X/k)$$
      \end{enumerate}
    \item there exists a functor $\cHP(-/k)
      \colon\mathrm{Sch}^{\natural, \mathrm{op}}_k\to \coCh(\rrD(k))$
      sending each scheme $X$ to a coherent cochain complex $\cHP(X/k)$
      such that
      \begin{enumerate}
        \item the total homology of $\cHP(X/k)$ is given by periodic homology:
          $$\tothom(\cHP(X/k)) \simeq \rr{HP}(X/k)$$
        \item the components of $\cHP(X/k)$ are given by shifts of
          the Hodge-completed derived de Rham complex:
          $$\cHP(X/k)^n \simeq \widehat{\rrL\Omega}_{X/k}[3n]$$
        \item the coherent cochain complex $\cHP(X/k)$ generates a conditionally
          convergent spectral sequence
          $$E_1^{i,j}\cong H^{j+3i}_{\rr{dR}} (X/k)
          \Longrightarrow \pi_{-i-j} \rr{HP}(X/k).$$
      \end{enumerate}
  \end{enumerate}
\end{ex}

\begin{rem}\label{rem-derham}
  Let $k$ be an animated ring\footnote{following \cite{cesnavicius2019purity},
  an element of the \cat obtained by Dwyer--Kan localization of the 1-category
  of simplicial commutative rings at the subcategory of weak equivalences};
  then, the hull-shelled complex of the Hodge-completed derived de Rham complex
  $\widehat{\rrL\Omega}_{X/k}$ with the Hodge filtration
  $$\bbL\Omega_{X/k}\coloneqq
  \caush\left(\widehat{\rrL\Omega}^{\geq \istar}_{X/k}\right)$$
  is such that its components are successive exterior powers of the cotangent
  complex:
  $$\left(\bbL\Omega_{X/k}\right)^n \simeq \Lambda^n\rrL_{R/k};$$
  thus, we can represent $\bbL\Omega_{X/k}$ as
  $$\cdots 0\to R\to \rrL_{R/k}\to\Lambda^2\rrL_{R/k}\to
  \Lambda^3\rrL_{R/k}\to\cdots$$
  showing how it can be seen as a homotopy coherent
  generalization of the usual algebraic de Rham complex.
\end{rem}

\begin{ex} [\cite{raksit}]
  With notation as in \ref{rem-derham}, $\bbL\Omega_{X/k}$
  can be given the structure of a derived commutative algebra in coherent
  cochain complexes (i.e.\ a module over a suitably defined
  $\infty$-operad $\rr{LSym}$; see \cite[\S 4]{raksit} for details);
  the resulting object is initial in the full subcategory of
  $\rr{LMod}_{\rr{LSym}}\coCh(\Mod_k)$ consisting of objects $C$ that are
  concentrated in non-negative degrees and are equipped with a map of $k$-algebras
  $R\to C^0$.
  Moreover, the first differential of the complex $\partial^0\colon
  R\to\rrL_{R/k}$ is given by the universal derivation.
\end{ex}

\begin{ex} [\cite{BMS2}, \cite{AnNik20}]
  Let $k$ be a perfect field of characteristic $p$, and let
  $\rr{Alg}^{\text{ind-sm}}_k$ denote
  the category of ind-smooth $k$-algebras. Then, there exists a functor
  $\cTP\colon\rr{Alg}^{\text{ind-sm}}_k\to\coCh(\Sp)$ sending a $k$-algebra
  $R$ to a coherent cochain complex $\cTP(R)$ such that:
  \begin{enumerate}
    \item the total homology of $\cTP(R)$ is given by topological periodic
      homology: $$\tothom(\cTP(R)) \simeq \rr{TP}(R).$$
    \item the components of $\cTP(R)$ are given by the de Rham--Witt complex
      $$\cTP(R)^n \simeq \rrW\Omega^\bullet_{R}[3n];$$ in particular,
      when $R$ is smooth, the components are given by crystalline
      cohomology
      $$\cTP(R)^n \simeq \rrR\Gamma_{\text{crys}}(R/\rrW(k))[3n];$$
    \item the coherent cochain complex $\cTP(R)$ generates a conditionally
      (strongly, if $\Spec R$ has finite dimension over $\Spec k$)
      convergent spectral sequence
      $$E_1^{i,j}\cong H_{\rrW\Omega}^{j+3i}(R)
      \Longrightarrow \pi_{-i-j} \rr{TP}(R);$$
      in particular, for $R$ smooth, we get
      $$E_1^{i,j}\cong H_{\text{crys}}^{j+3i}(R/W(k))
      \Longrightarrow \pi_{-i-j} \rr{TP}(R).$$
  \end{enumerate}
\end{ex}

\appendix
\section{$\infty$-categorical and stable nerve-realization paradigm}
\label{appendix-first}
The content of this appendix is well-known to experts, but to the best of
our knowledge there is no systematic presentation of these results in the
existing literature. Thus we thought it appropriate to include them here for
the convenience of the reader.
We begin by showing a ``density formula'' for $\infty$-categorical presheaves,
then proceed to describe the nerve-realization paradigm in the unstable and
in the stable case. For the unstable case, we closely follow the
presentation of the analogue results in ordinary category theory given in
\cite[Chapter 3]{coend}, adapting the few things that need to be adapted in
order to translate the arguments to the $\infty$-categorical setting.
For the stable case, we leverage on results of \cite{bgt} and
\cite{NikStableOperads} to give a description of stable nerves.

\subsection{A ``density formula'' for presheaves}

One of the features making the \cat of presheaves of a small \cat $\scrC$ so
important, is the fact that the Yoneda embedding is in some sense the
``free cocompletion functor'' (see \cite[5.1.5.6]{HTT} for a precise statement
and a proof). In particular, every presheaf is the colimit of representable
ones (see \cite[5.1.5.8]{HTT}). The proof contained in \cite{HTT} of this last
fact is somewhat indirect, as it does not explicitly give a way to
construct for any presheaf $P$, a diagram having $P$ as a colimit. As
later we will need such an explicit diagram, we will now show that, similarly
to what happens in the ordinary case, any presheaf is the colimit of the
inclusion of a suitable comma category in the category of presheaves.

The following is just a special case of the most general definition of a
comma \cat it is possible to give, but it is general enough for our purposes.
\begin{defi}\label{defi-comma-cat}
  Let $F\colon\scrC\to\scrD$ be a functor between \cats, and $d\in\scrD$
  an object. We will denote by $(F \downarrow d)$ the \emph{comma category}
  obtained as the following pullback
  \begin{equation}\label{diag-defi-comma}
  \begin{tikzcd}
    (F \downarrow d) \ar[r] \ar [d, "\pjct"'] & \scrD_{/d} \ar[d] \\
    \scrC \ar[r, "F"] & \scrD.
  \end{tikzcd}
  \end{equation}
\end{defi}

\begin{lemma}\label{lemma-comma-class}
  The map $(F\downarrow d)\to\scrC$ is a right fibration
  classifying $\Map_{\scrD}\left(F(-),d\right)$.
\end{lemma}
\begin{proof}
  As the projection $\scrD_{/d}\to\scrD$ classifies $\Yo_{d}$,
  \cite[3.2.1.4]{HTT} implies that the pullback
  (\ref{diag-defi-comma}) classifies the composite $\Yo_d\circ F$.
\end{proof}

The above proposition, together with the Yoneda lemma, specialize to
presheaves as follows.

\begin{rem}\label{rem-class-P}
  By Lemma \ref{lemma-comma-class}, the right fibration
  $(\Yo\downarrow P)\to\scrC$ classifies the functor
  $\Map_{\Pre{\scrC}}(\Yo_{-},P)$,
  but, by \cite[5.5.2.1]{HTT}, this is precisely $P$, hence
  $(\Yo\downarrow P)\to\scrC$ is a right fibration classifying $P$.
\end{rem}

\begin{prop}\label{prop-dens}
  Let $\scrC$ be a small \cat, and $P$ any presheaf on $\scrC$.
  Then
  $$P \simeq \colim\left( (\Yo\downarrow P)\xto{\pjct}\scrC\xto{\Yo}
  \Pre{\scrC} \right).$$
\end{prop}
\begin{proof}
  We will prove that $P$ and $X$ corepresent the same functor in $\Pre\scrC$.
  Given any presheaf $Y$, have that
  \begin{equation}\label{eq-dens-first}
    \Map_{\Pre\scrC}
      \left(\colim_{\alpha\in(\Yo\downarrow P)}\Yo\circ\pi(\alpha),Y\right)
    \simeq\lim_{\alpha\in(\Yo\downarrow P)}
      \Map_{\Pre\scrC}(\Yo_{\pi(\alpha)},Y)
  \end{equation}
  Let us now notice that, by Lemma \ref{lemma-comma-class}, the functor
  \begin{displaymath}
    \Map_{\Pre\scrC}(\Yo\circ \pi(-),Y)
    \colon(\Yo\downarrow P)\to\spaces
  \end{displaymath}
  is classified by the right fibration
  $$(\Yo\circ\pi \downarrow Y) \to (\Yo \downarrow P)$$
  hence, by \cite[3.3.3.4]{HTT}, (\ref{eq-dens-first})
  is equivalent to the space of sections
  \begin{equation}\label{yo-P-sections}
  \Map_{/(\Yo\downarrow P)}\left((\Yo\downarrow P),
    (\Yo\circ\pi \downarrow Y)\right).
  \end{equation}
  Let us note that, as $(\Yo\circ\pi \downarrow Y)$ fits in the following
  pasting of pullback squares
  $$
  \begin{tikzcd}
    (\Yo\circ\pi \downarrow Y) \ar[r] \ar[d]
      & (\Yo\downarrow Y) \ar[r] \ar[d]
      & \Pre\scrC_{/ Y} \ar[d] \\
    (\Yo\downarrow P) \ar[r] & \scrC \ar[r] & \Pre\scrC
  \end{tikzcd}
  $$
  we have that (\ref{yo-P-sections}) is equivalent to
  $$
  \Map_{/ \scrC}\left((\Yo\downarrow P),
    (\Yo\downarrow Y)\right),
  $$
  and, by Remark \ref{rem-class-P}, and by virtue of the
  straightening/unstraightening equivalence, we have that
  \begin{displaymath}
  \begin{split}
    \Map_{/ \scrC}\left((\Yo\downarrow P),(\Yo\downarrow Y)\right)
      &\simeq \Map_{/ \scrC}\left(\Un(P),\Un(Y)\right)\\
      &\simeq \Map_{\Pre\scrC}\left(\St\left(\Un(P)\right),Y\right)\\
      &\simeq \Map_{\Pre\scrC}\left(P,Y\right),
  \end{split}
  \end{displaymath}
  completing the proof.
\end{proof}

\subsection{The unstable nerve-realization paradigm}

With the density formula at our disposal, we can now easily adapt the
results on nerves and realizations from ordinary category theory to the
coherent setting. In order to do so, it will be useful to express Kan
extensions using co/ends, hence we will start by recollecting some definitions
and proving a few useful lemmata.

\begin{defi}\cite[5.2.1.1]{HA},\cite[2.2]{GHN17}
  Let $\scrC$ be an \cat, and let $\eps\colon\simplex\to\simplex$ be the functor
  given by $[n]\mapsto[n]\star[n]\op$.
  The \emph{twisted arrow category} $\tw\scrC$ is the simplicial set given
  by $\eps^*\scrC$.
\end{defi}

\begin{rem}
  In particular, we have
  $$
  (\tw\scrC)_n\cong\Hom(\Delta^n\star(\Delta^n)\op,\scrC)
  $$
  and two canonical projections $\tw\scrC\to\scrC$ and $\tw\scrC\to\scrC\op$
  induced by the natural transformations $\Delta^\bullet\to\Delta^\bullet\star
  (\Delta^\bullet)\op$ and $(\Delta^\bullet)\op\to\Delta^\bullet\star
  (\Delta^\bullet)\op$. By \cite[5.2.1.3]{HA}, the induced map
  $\tw\scrC\to\scrC\times\scrC\op$ is a right fibration, and by
  \cite[5.2.1.11]{HA} it classifies the bifunctor $\Map_\scrC\colon
  \scrC\op\times\scrC\to\spaces$. Equivalently, the opposite map
  $(\tw\scrC)\op\to\scrC\op\times\scrC$ is the left fibration classifying
  the bifunctor $\Map_\scrC$.
\end{rem}

\begin{defi}\cite[2.5]{GHN17}
  If $F\colon\scrC\op\times\scrC\to\scrD$ is a functor of \cats, the
  \emph{end} (resp. \emph{coend}) of $F$ is defined to be the limit
  (resp. colimit) of the composite functor
  $$\tw\scrC\to\scrC\op\times\scrC\to\scrD$$
  and is denoted by
  $$\int_{c\in\scrC}F \quad \left(\text{ resp. }\quad \int^{c\in\scrC}F\right).$$
\end{defi}

\begin{lemma}\label{lemma-tw-projection-is-final}
  Both projection maps $\tw\scrC\to\scrC\op$ and $\tw\scrC\to\scrC$ are
  final and initial functors.
  \footnote{Sometimes in the literature the terms ``cofinal functor''
  and ``final functor'' are used in place of what we call respectively ``final
  functor'' and ``initial functor''.}
\end{lemma}
\begin{proof}
  We will prove that the projection $\tw\scrC\to\scrC\op$ is both final and
  initial, the other case being entirely analogous.
  Let us start by proving finality.
  By \cite[4.1.3.1]{HTT}, it is enough to check the fibers of
  $\tw\scrC\to\scrC\op$ are weakly contractible (notice that, for any $C\in
  \scrC$, as $\scrC_{C/}$ has an initial object, it is weakly contractible,
  and thus $\tw\scrC\times_{\scrC\op}\scrC_{C/}$ has the same homotopy type
  of $\tw\scrC\times_{\scrC\op} \{C\}$).
  Given any $C\in\scrC$, let us denote the fiber $\tw\scrC\times_{\scrC\op}\{C\}$
  by $(\tw\scrC)_C$.
  We claim that each $(\tw\scrC)_C$ is equivalent to $\scrC_{C/}$ as an \cat,
  and thus weakly contractible.

  By pasting of pullbacks, we know that $(\tw\scrC)_C$ fits into a diagram
  $$
  \begin{tikzcd}
    (\tw\scrC)_C \ar[r]\ar[d]& \tw\scrC \ar[d]\\
    \Delta^0\times\scrC \ar[r,"C\times\id"]\ar[d]& \scrC\op\times\scrC \ar[d]\\
    \Delta^0 \ar[r,"C"]& \scrC\op
  \end{tikzcd}
  $$
  and, again by pasting of pullbacks, we know we have the following diagram
  $$
  \begin{tikzcd}
    \Map_\scrC(C,D)\ar[r]\ar[d]&(\tw\scrC)_C \ar[r]\ar[d]& \tw\scrC \ar[d]\\
    \Delta^0\ar[r,"D"]&\Delta^0\times\scrC \ar[r,"C\times\id"]&
      \scrC\op\times\scrC
  \end{tikzcd}
  $$
  where both squares are Cartesian. As
  $\tw\scrC\to\scrC\op\times\scrC$ is a right fibration, hence
  $(\tw\scrC)_C\to\scrC$ is a right fibration as well. But, as we identified
  the fibers over any $D\in\scrC$ as $\Map(C,D)$, this has to be the
  classifying fibration for $\Map(C,-)$, and thus $(\tw\scrC)_C$ has to be
  equivalent to $\scrC_{C/}$ as desired.
  Recall that, by definition, a functor $F\colon\scrC\to\scrD$ is initial if
  and only if its opposite functor $F\op\colon\scrC\op\to\scrD\op$ is final.
  The initiality statement follows from an analogous argument
  applied to the opposite projection $(\tw\scrC)\op\to\scrC$, using
  that the map $(\tw\scrC)\op\to\scrC\op\times\scrC$ is a left fibration.
\end{proof}

The above lemma has as an immediate consequence that a co/end with a dummy
variable is just a co/limit.

\begin{cor}\label{cor-dummy-coend}
  Let $F\colon\scrC\op\times\scrC\to\scrD$ be a functor that factors
  through the projection $\scrC\op\times\scrC\to\scrC$ (resp.
  $\scrC\op\times\scrC\to\scrC\op$), and let $\widetilde F\colon
  \scrC\to\scrD$ (resp. $\widetilde F\colon\scrC\op\to\scrD$) denote such
  factorization. Then
  $$
  \quad
  \int^\scrC F \simeq \colim_\scrC \widetilde F \text{ and }
  \int_\scrC F \simeq \lim_\scrC \widetilde F
  $$
  $$
  \left(
  \text{ resp. }
  \int^\scrC F \simeq \colim_{\scrC\op} \widetilde F \text{ and }
  \int_\scrC F \simeq \lim_{\scrC\op} \widetilde F
  \right).
  $$
\end{cor}

\begin{lemma}\label{lemma-tensor-functor}
  Let $F\colon\scrC\to\scrD$ be a functor between \cats, with $\scrD$
  cocomplete. Then
  $$F\simeq\int^{c\in\scrC}\Map_{\scrC}(c,-)\tensor F(c)$$
  where $S\tensor X\coloneqq\colim_{S} X$ denotes the canonical tensoring in
  spaces for \cats.
\end{lemma}
\begin{proof}
  The following computation shows that the two objects corepresent the same
  functor
  \begin{displaymath}
  \begin{split}
    \Map_{\Fun (\scrC,\scrD)}
      \Bigg(\int^{c\in\scrC}&\Map_{\scrC}(c,-)\otimes F(c),G\Bigg)\simeq \\
    &\simeq \int_{c^\prime\in\scrC}\Map_{\scrD}
      \left(\int^{c\in\scrC}\Map_{\scrC}(c,c^\prime)\tensor F(c),
      G(c^\prime)\right)\\
    &\simeq \int_{c^\prime\in\scrC}\int_{c\in\scrC}\Map_{\scrD}
      \Big(\Map_{\scrC}(c,c^\prime)\tensor F(c),G(c^\prime)\Big)\\
    &\simeq \int_{c^\prime\in\scrC}\int_{c\in\scrC}
      \lim_{\Map_{\scrC}(c,c^\prime)}\Map_{\scrD}
      \left(F(c),G(c^\prime)\right)\\
    &\simeq \int_{c^\prime\in\scrC}\int_{c\in\scrC}
      \Map_{\spaces}\Big(\Map_{\scrC}(c,c^\prime),\Map_{\scrD}
      \left(F(c),G(c^\prime)\right)\Big)\\
    &\simeq \int_{c^\prime\in\scrC}
      \Map_{\Pre\scrC}\Big(\Yo_{c^{\prime}},\Map_{\scrD}
      \big(F(-),G(c^\prime)\big)\Big)\\
    &\simeq \int_{c^\prime\in\scrC}\Map_{\scrD}
      \big(F(c^\prime),G(c^\prime)\big)\\
    &\simeq \vphantom{\int_{c^\prime\in\scrC}}\Map_{\Fun(\scrC,\scrD)}
      \left(F,G\right).
  \end{split}
  \end{displaymath}
\end{proof}

We now turn to the discussion of Kan extensions in terms of co/ends.

\begin{lemma}\label{lemma-coend-formula}
  Let $\scrC$, $\scrD$ and $\scrE$ be \cats, with $\scrE$ cocomplete, and
  let $G\colon\scrC\to\scrD$ and $F\colon\scrC\to\scrE$ be functors. Then
  $$\Lan_{G}F(d)\simeq\int^{c\in\scrC}\Map_{\scrD}(Gc,d)\tensor F(c).$$
\end{lemma}
\begin{proof}
  Unraveling the definitions (see \cite[4.3.3.2]{HTT} and \cite[4.3.2.2]{HTT}),
  we know that
  \begin{equation}\label{eq-ptwise-formula}
    \Lan_G F (d) \simeq \colim \left((G\downarrow d)\xto{\pjct}\scrC
    \xto{F}\scrE\right).
  \end{equation}
  Now, the result holds from the following computation
  \begin{displaymath}
  \begin{split}
    \int^{c\in\scrC}\Map_{\scrD}(Gc,d)\tensor F(c)
      &\stackrel{\vphantom{(}^{(\ref{prop-dens})}}{\simeq}
      \int^{c\in\scrC}\left(\colim_{\alpha\in(G\downarrow d)}
      \Yo\circ\pjct(\alpha)\right)(c)\tensor F(c)\\
    &\stackrel{\vphantom{(}^{\hphantom{(\ref{prop-dens})}}}{\simeq}
      \colim_{\alpha\in(G\downarrow d)}\int^{c\in\scrC}
      \left(\Yo\circ\pjct(\alpha)\right)(c)\tensor F(c)\\
    &\stackrel{\vphantom{(}^{\hphantom{(\ref{prop-dens})}}}{\simeq}
      \colim_{\alpha\in(G\downarrow d)}\int^{c\in\scrC}
      \Map_{\scrC}\big(c,\pjct(\alpha)\big)\tensor F(c)\\
    &\stackrel{\vphantom{(}^{(\ref{lemma-tensor-functor})}}{\simeq}
      \vphantom{\int^{c\in\scrC}}
      \colim_{\alpha\in(G\downarrow d)} \ F\circ\pjct(\alpha)
      \stackrel{\vphantom{(}^{(\ref{eq-ptwise-formula})}}{\simeq}
      \vphantom{\int^{c\in\scrC}}
      \Lan_G F (d).
  \end{split}
  \end{displaymath}
\end{proof}

\begin{rem}\label{rem-end-ran-formula}
  A suitable dualization of the above arguments shows that
  $$
  \Ran_G F(d) \simeq \int_{c\in\scrC} [\Map_{\scrD}(d,Gc),F(c)]
  $$
  where $[S,X]\coloneqq\lim_{S} X$ denotes the canonical cotensoring in spaces
  for \cats.
\end{rem}

\begin{defi}\label{defi-nerve-real}
  Let $F\colon\scrC\to\scrD$ be a functor from a small \cat $\scrC$ to a
  cocomplete \cat $\scrD$. The left Kan extension of $F$ along the Yoneda
  embedding $\Yo\colon\scrC\to\Pre{\scrC}$ is called the
  \emph{$F$-realization}, denoted
  $$\realiz{F}\colon\Pre\scrC\to\scrD.$$
\end{defi}

\begin{prop}\label{prop-nerve-real}
  Let $F\colon\scrC\to\scrD$ be a functor from a small \cat $\scrC$ to a
  cocomplete \cat $\scrD$. Then, the $F$-realization admits a right adjoint,
  denoted $\nerve_{F}\colon\scrD\to\Pre\scrC$.
  Moreover, the value of $\nerve_{F}$ on an object
  $d\in\scrD$ is given by the presheaf $\Yo_d \circ F$, or, more informally
  \begin{displaymath}
    \nerve_{F}(d)\colon c\mapsto\Map_{\scrD}(F(c),d).
  \end{displaymath}
\end{prop}
\begin{proof}
  By \cite[5.1.5.5, 5.1.5.6]{HTT}, precomposition with $\Yo$ and left Kan
  extension along it give mutually inverse functors
  $$\Fun^{L}(\Pre{\scrC},\scrD)\stackrel{\sim}{\rightleftarrows}
    \Fun(\scrC,\scrD)$$
  hence $\nerve_{F}$ has a right adjoint.
  The objectwise description follows formally from the following computation
  \begin{displaymath}
  \begin{split}
    \Map_{\scrD}\left(\lvert P \rvert_{F},d\right)
      &\stackrel{\vphantom{(}^{(\ref{lemma-coend-formula})}}{\simeq}
        \Map_{\scrD}\left(\int^{c\in\scrC}\Map_{\Pre{\scrC}}(\Yo_c,P)
        \tensor F(c),d\right)\\
    &\simeq \int_{c\in\scrC}\Map_{\scrD}\Big(P(c) \tensor F(c),d\Big) \\
    &\simeq \int_{c\in\scrC}\Map_{\spaces}\Big(P(c),
      \Map_{\scrD}\big(F(c),d\big)\Big) \\
    &\simeq \vphantom{\int_{c\in\scrC}} \Map_{\Pre{\scrC}}\Big(P,
      \Map_{\scrD}\big(F(-),d\big)\Big)
  \end{split}
  \end{displaymath}
  from which we see that $\Yo_d \circ F$ is right adjoint to $F$-realization.
\end{proof}

\begin{defi}
  In the situation of Proposition \ref{prop-nerve-real}, we refer to the right
  adjoint
  $$\nerve_{F}\colon\scrD\to\Pre\scrC$$
  as the \emph{$F$-nerve}.
\end{defi}

\subsection{The stable nerve-realization paradigm}
\label{subsec-appendix-stable-nerve}
When, in the situation of Definition \ref{defi-nerve-real} the \cat $\scrD$
is stable and presentable, we can take the nerve-realization paradigm one step
further.

\begin{rem}\label{rem-stab-nerve-real}
  Let $F\colon\scrC\to\scrD$ be a functor from a small \cat to a stable
  presentable \cat $\scrD$. By \cite[1.4.4.5]{HA}, we have an equivalence
  $$\Fun^L(\Pst\scrC,\scrD)\simeq\Fun^L(\Pre\scrC,\scrD)$$
  given by precomposition with $\susinftyp\colon\Pre\scrC\to\St(\Pre\scrC)$,
  and, equivalently, an equivalence
  $$\Fun^R(\scrD,\Pst\scrC)\simeq\Fun^R(\scrD,\Pre\scrC)$$
  given by postcomposition with $\ominfty$.
  In particular, we get an extension of the nerve-realization adjunction
  $$
  \begin{tikzcd}[column sep=huge]
    \Pre\scrC \ar[r, shift left=1.1ex, "\lvert - \rvert_{F}"] &
    \scrD \ar[l, shift left=1.1ex,"\nerve_{F}"]
      \ar[l,phantom,"\text{\rotatebox{-90}{$\dashv$}}"]
  \end{tikzcd}
  $$
  to an adjunction
  \begin{equation}\label{eq-stab-nerve-real}
  \begin{tikzcd}[column sep=huge]
    \Pst\scrC \ar[r, shift left=1.1ex, "\lvert - \rvert_{F}^{\mathrm{st}}"] &
    \scrD. \ar[l, shift left=1.1ex,"\nerve_{F}^{\mathrm{st}}"]
      \ar[l,phantom,"\text{\rotatebox{-90}{$\dashv$}}"]
  \end{tikzcd}
  \end{equation}
  Notice that, again by \cite[1.4.4.5]{HA}, we have that
  $$\lvert - \rvert_{F}
  \simeq \lvert \susinftyp(-) \rvert_{F}^{\mathrm{st}}
  \quad \text{and} \quad
  \nerve_{F}\simeq \ominfty \left(\nerve_{F}^{\mathrm{st}}(-)\right).$$
\end{rem}

\begin{defi}\label{defi-stab-nerve-real}
  We refer to the adjunction (\ref{eq-stab-nerve-real}) as the
  \emph{stable nerve-realization paradigm}.
\end{defi}

Putting together what we got this far, we get the following equivalence.

\begin{lemma}\label{st-nerve-real}
  Let $\scrC$ be a small \cat and $\scrD$ be a stable, presentable \cat. Then,
  the stable nerve-realization paradigm gives an equivalence of \cats
  \begin{displaymath}
  \begin{split}
    \Fun(\scrC,\scrD)&\to\Fun^{L}(\Pst\scrC,\scrD) \\
    F&\mapsto \lvert - \rvert_{F}^{\mathrm{st}}
  \end{split}
  \end{displaymath}
\end{lemma}
\begin{proof}
  As $\scrC$ is small, and $\scrD$ is cocomplete, by \cite[5.1.5.7]{HTT}, we
  have
  $$\Fun(\scrC,\scrD)\simeq\Fun^{L}(\Pre\scrC,\scrD)$$
  and the equivalence is given by $F\mapsto\lvert - \rvert_{F}$.
  Moreover, as both $\Fun(\scrC\op,\spaces)$ and $\scrD$ are presentable and
  $\scrD$ is stable, by \cite[1.4.4.5]{HA}, we have a further equivalence
  $$\Fun^L(\Pst\scrC,\scrD)\simeq\Fun^L(\Pre\scrC,\scrD)$$
  as reviewed in Remark \ref{rem-stab-nerve-real}, completing the proof.
\end{proof}

It is also possible to give an explicit description of stable nerves, similar
to the one we gave in the unstable case, using mapping spectra instead of
mapping spaces. To this end, let us first recall the stable Yoneda embedding.

\begin{defi}[](See \cite[Definiton 2.15]{bgt}\footnote{notice that we use
  the opposite capitalization convention of \cite{bgt} to distinguish between
  mapping spaces and the mapping spectra.})
  \label{defi-st-yo}
  Let $\scrD$ be a stable \cat. Recall that, by \cite[1.4.2.23]{HA},
  postcomposition with $\ominfty$ induces an equivalence
  $$\Fun^{\rm{lex}}(\scrD,\Pst\scrD)\simeq\Fun^{\rm{lex}}(\scrD,\Pre\scrD).$$
  We call the \emph{stable Yoneda embedding} the map corresponding to
  $\Yo$ under the above equivalence and denote it as $\stYo$.
  We denote the adjoint functor $\scrD\op\times\scrD\to\Sp$ as
  $$\map_{\scrD}(=,-)$$
  and call it the \emph{mapping spectrum functor}.
\end{defi}

\begin{rem}\label{rem-map-colim-commute}
  In particular, the functor $\map_{\scrD}(=,-)$ is such that
  $$\ominfty\map_{\scrD}(=,-)\simeq\Map_{\scrD}(=,-)$$
  and is the unique such functor sending finite colimits in the first
  variable to limits and finite limits in the second variable to limits.
  As shown in \cite[Remark 6.2]{NikStableOperads}, the functor actually sends
  \emph{all} (small) colimits in the first variable to limits, and \emph{all}
  (small) limits in the second variable to limits.
\end{rem}

\begin{rem}\label{rem-styo-is-fff}
  The functor $\stYo$ is extensively studied in \cite{NikStableOperads}. In
  particular, in \emph{op. cit.} Section 6 it is proven that it is fully
  faithful.
\end{rem}

We have the following variant of Yoneda's Lemma for spectral presheaves
of general \cats.

\begin{prop}\label{prop-st-yo-kinda}
  Let $\scrC$ be a small \cat; then
  $$
  \map_{\Pst \scrC}(\susinftyp \Yo , - ) \simeq \id_{\Pst\scrC}.
  $$
\end{prop}
\begin{proof}
  By \cite[1.4.2.23]{HA}, we know that
  \begin{equation}\label{eq-prop-st-yo-kinda}
    \Fun^{\rm{lex}}(\Pst\scrC,\Pst\scrC)\simeq
    \Fun^{\rm{lex}}(\Pst\scrC,\Pre\scrC)
  \end{equation}
  and that $\id_{\Pst\scrC}$ on the left hand side, corresponds to $\ominfty$
  on the right hand side.
  Now, by the usual $\infty$-categorical Yoneda lemma (see \cite[5.5.2.1]{HTT})
  we have
  $$\ominfty\simeq\Map_{\Pre\scrC}(\Yo,\ominfty)\simeq
    \Map_{\Pst\scrC}(\susinftyp \Yo, -).$$
  But, by Remark \ref{rem-map-colim-commute}, we know that the latter
  corresponds to
  $$\map_{\Pst\scrC}(\susinftyp \Yo , -)$$
  under the equivalence (\ref{eq-prop-st-yo-kinda}).
\end{proof}

We conclude the section with an explicit description of stable nerves.

\begin{prop}\label{prop-mapping-nerve}
  Let $F\colon\scrC\to\scrD$ be a functor from a small \cat $\scrC$ to a
  stable presentable \cat $\scrD$. Then, the value of the stable $F$-nerve
  on an object $d\in\scrD$ is given by the stable presheaf
  $\map_{\scrD}(F(-),d)$.
\end{prop}
\begin{proof}
  The functor $\nerve_F^{\rm{st}}$ is adjoint to a functor
  $\chi\colon\scrC\op\times\scrD\to\Sp$,
  which in turn is adjoint to a functor $\xi\colon\scrC\op\to\Pst{\scrD\op}$.
  By Remark \ref{rem-stab-nerve-real}, together with Proposition
  \ref{prop-nerve-real}, we have that
  $$\ominfty\xi(c)\simeq\Map_{\scrD}(Fc,-)\colon\scrD\to\spaces$$
  and that each $\xi(c)$ is left exact (as $\nerve_F^{\rm{st}}$ is so),
  thus, by Remark \ref{rem-map-colim-commute}, $\xi(c)\simeq
  \map_{\scrD}(Fc,-)$, from
  which it follows that $\chi\simeq\map_{\scrD}(F(=),-)$, and thus that
  $$\nerve_F^{\rm{st}}(d)\simeq\map_{\scrD}(F(-),d)$$
  as requested.
\end{proof}

\subsection{The case of pointed \cats}

We recollect some variations of the results of \S
\ref{subsec-appendix-stable-nerve} that apply to pointed \cats.
We will not give full proofs for them, as the arguments are entirely analogue
to the ones used in the stable case.

We will denote by
$\Prept\scrC\coloneqq\Fun(\scrC\op,\spaces_*)\simeq \Pre\scrC\tensor\spaces_*$
the \cat of presheaves with values in pointed spaces.

\begin{defi}\label{defi-pt-yo}
  Let $\scrC$ be a pointed \cat. By \cite[4.8.2.12]{HA}, postcomposition
  with the functor forgetting the basepoint induces an equivalence
  $$\Fun^0{}(\scrC,\Prept\scrC)\simeq
    \Fun^{\prime}(\scrC,\Pre\scrC)$$
  where $\Fun^{0}$ denotes the \cat of pointed functors,
  and $\Fun^{\prime}$ denotes the \cat of functors carrying the zero objects
  of $\scrC$ to terminal objects of $\Pre\scrC$.
  We call the \emph{pointed Yoneda embedding} the map corresponding to
  $\Yo$ under the above equivalence, and denote it as $\ptYo$.
  We denote the adjoint functor $\scrC\op\times\scrC\to\spaces_*$ as
  $$\Map{}_{\scrC,*}(=,-)$$
  (or simply as $\Map{}_*(=,-)$, when there is no risk of confusion)
  and call it the \emph{pointed mapping space functor}.
\end{defi}

\begin{defi}
  Let $\scrC$ be a pointed essentially small \cat. Then $\Prefinpt \scrC$
  denotes the smallest full subcategory of
  $\Prept\scrC$ containing the essential image
  of $\ptYo\colon\scrC\to\Prept\scrC$ that is closed under finite colimits.
\end{defi}

\begin{rem}\label{rem-pt-nerve-real}
  Let $F\colon\scrC\to\scrD$ be a functor from a small \cat to a pointed
  presentable \cat $\scrD$.
  Then, we have an extension of the nerve-realization adjunction
  \begin{equation}\label{eq-pt-nerve-real}
  \begin{tikzcd}[column sep=huge]
    \Prept\scrC \ar[r, shift left=1.1ex, "\lvert - \rvert_{F}^{\mathrm{pt}}"] &
    \scrD \ar[l, shift left=1.1ex,"\nerve_{F}^{\mathrm{pt}}"]
      \ar[l,phantom,"\text{\rotatebox{-90}{$\dashv$}}"]
  \end{tikzcd}
  \end{equation}
  such that
  $$\lvert - \rvert_{F}
  \simeq \lvert (-)_+ \rvert_{F}^{\mathrm{pt}}
  \quad \text{and} \quad
  \nerve_{F}\simeq U \left(\nerve_{F}^{\mathrm{pt}}(-)\right)$$
  where $U$ denotes the forgetful functor $\spaces_*\to\spaces$ (usually omitted
  from notations), and $(-)_+$ denotes its left adjoint.
\end{rem}

\begin{defi}\label{defi-pt-nerve-real}
  We refer to the adjunction (\ref{eq-pt-nerve-real}) as the
  \emph{pointed nerve-realization paradigm}.
\end{defi}

\begin{prop}\label{prop-pointed-rex-univ}
  Let $\scrC$ be an essentially small \cat. For any pointed and finitely
  cocomplete \cat $\scrD$, composition with
  $\ptYo\colon\scrC\to\Prefinpt\scrC$
  induces an equivalence
  $$\Fun^{\text{Rex}}(\Prefinpt\scrC,\scrD)\to\Fun(\scrC,\scrD).$$
  If $\scrD$ is also presentable, we have
  $$\Fun^{\text{L}}(\Prept\scrC,\scrD)\simeq
  \Fun^{\text{Rex}}(\Prefinpt\scrC,\scrD)\simeq\Fun(\scrC,\scrD)$$
  and the right-to-left composite is given by
  $$F\mapsto \lvert - \rvert_{F}^{\mathrm{pt}}.$$
\end{prop}

\begin{prop}\label{prop-pt-mapping-nerve}
  Let $F\colon\scrC\to\scrD$ be a functor from a small \cat $\scrC$ to a
  pointed presentable \cat $\scrD$. Then, the value of the pointed $F$-nerve
  on an object $d\in\scrD$ is given by the pointed presheaf
  $\ptMap(F(-),d)$.
\end{prop}

\section{An $\bbE_1$ presentation of exterior algebras (by Achim Krause)}
\label{appendix-achim}

In this appendix, we produce an explicit $\bE_1$ presentation of exterior algebras over $\Z$. The main result is the following:

\begin{thm}
\label{thm:e1presentation-ungraded}
Let $\Lambda(e_k)$ denote the $\bE_1$-algebra in $\cD(\Z)$ represented by the dga
which is exterior on a generator in degree $k$, with zero differential.
Then $\Lambda(e_k)$ admits a description as a colimit 
\[
\colim_{\Z_{\geq 1}} \Lambda^{(n)} \simeq \Lambda(e_k)
\]
where $\Lambda^{(1)}$ is a free $\bE_1$-algebra on a generator of degree $n$, and each of the maps $\Lambda^{(n-1)}\to \Lambda^{(n)}$ is an $\bE_1$-pushout
\[
\begin{tikzcd}
\Free_{\bE_1}(r_n)\rar\dar & \Lambda^{(n-1)}\dar\\
\Z\rar & \Lambda^{(n)}
\end{tikzcd}
\]
Here $r_n$ is an element of degree $nk+n-2$.
\end{thm}

The significance of Theorem \ref{thm:e1presentation-ungraded} is that it allows us to describe an obstruction theory for $\bE_1$-maps out of exterior algebras.
Indeed, to find a map $\Lambda(e_k) \to R$, we need to provide a cycle $e$ of degree $k$ in $R$,
a nullhomotopy of $e\cdot e$, a nullhomotopy of a certain element of degree $3k+1$
(depending on the chosen nullhomotopy of $e\cdot e$), and so on. Having provided the first $i$ pieces of this data, i.e. a map $\Lambda^{(i)}\to R$,
the next obstruction, i.e. the image of $r_{i+1}$, is a version of an $i+1$-fold Massey power of $e$.

Theorem \ref{thm:e1presentation-ungraded} will follow from the following bigraded version, where the degree of the generator is abstracted away into a second, formal grading.
We work in $\gr \cD(\Z) = \Fun(\Z^{\delta}, \cD(\Z))$ (where $\Z^{\delta}$ denotes the ``discrete'' category $\Z$, i.e. without nontrivial morphisms). An object is thus simply a list of objects of $\cD(\Z)$.
The monoidal structure is Day convolution. Homology groups now have two degrees: If our object is given by $(X_n)_{n\in \Z}$, we write $H_{k,n}(X) = H_k(X_n)$, and similarly we can shift in both those degrees,
writing $\Sigma^{k,n}$. Note that $\Sigma^{1,0}$ is ordinary suspension.
Whenever confusion is imminent, we will refer to the second grading (i.e. the $n$ in $(k,n)$ above) as ``formal'' degree.
Similarly, we will call an object of $\cD(\Z)$ ``formally $n$-connective'' if it vanishes in formal degrees $<n$, regardless of ``topological'' connectivity of the individual components $X_n$.
Note that the notion of formal connectivity is somewhat simpler than usual topological connectivity: The $n$-connective objects are closed under any limits and colimits, since those are formed pointwise.
(In particular, the notion of formal connectivity is unrelated to t-structures.)

\begin{thm}
\label{thm:e1presentation}
Let $\Lambda(e)$ denote the $\bE_1$-algebra in $\gr \cD(\Z)$ represented by the graded exterior algebra on one generator $e$ in bidegree $(0,1)$. Then $\Lambda(e)$ admits a description as colimit
\[
\colim_{\Z_{\geq 1}} \Lambda^{(n)} \simeq \Lambda(e)
\]
where $\Lambda^{(1)}$ is a free $\bE_1$ algebra on a generator of bidegree $(0,1)$, and each of the maps $\Lambda^{(n-1)}\to \Lambda^{(n)}$ is an $\bE_1$-pushout
\[
\begin{tikzcd}
\Free_{\bE_1}(r_{n})\dar\rar &\Lambda^{(n-1)}\dar\\
\Z \rar & \Lambda^{(n)}.
\end{tikzcd}
\]
Here $r_{n}$ is an element of bidegree $(n-2, n)$.
\end{thm}
The exterior algebra discussed in the theorem has underlying object of the form $(X_n)_{n\in \Z}$ with $X_0$ and $X_1$ both given by the complexes $\Z[0]$.
Another way to express this is to say that there is a functor
\[
\gr \Ab \to \gr \cD(\Z),
\]
simply given by applying the ``discrete complex'' functor pointwise, and it is lax symmetric monoidal.
The $\bE_1$ algebra $\Lambda(e)$ is precisely the image of the ordinary graded exterior algebra on a degree $1$ generator in $\gr\Ab$.

Theorem \ref{thm:e1presentation} should be regarded as a universal version of \ref{thm:e1presentation-ungraded}.
Indeed, the following lemma shows how \ref{thm:e1presentation-ungraded} is a special cases of Theorem \ref{thm:e1presentation}:

\begin{lemma}
Let $\cC$ be a closed symmetric-monoidal $\infty$-category and colimit-preserving symmetric-monoidal functor $\cD(\Z)\to \cC$, and let $L\in \cC$ be an invertible object.
Then there is a colimit-preserving monoidal functor
\[
\gr\cD(\Z)\to \cC
\]
taking the graded object $X$ to $\bigoplus_{\Z} L^{\otimes n} X_n$.
\end{lemma}
\begin{proof}
An invertible object determines a unique monoidal functor $\Z\to \cC$, since $\Z$ is the free $\bE_1$ monoid on a single invertible object.
Together with the symmetric-monoidal functor $\cD(\Z)\to \cC$, we obtain a monoidal functor $\Z\times \cD(\Z)\to \cC$, as the composite
\[
F: \Z\times \cD(\Z)\to \cC\times \cC\xrightarrow{\otimes} \cC.
\]
We also have a symmetric-monoidal functor $\Z\times \cD(\Z)\to \Fun(\Z,\cD(\Z))$, which simply takes $(n,C)$ to the graded object which is $C$ in degree $n$.
The left Kan extension of $F$ along this embedding is the desired functor, as one directly computes from the pointwise formula for Kan extensions.
\end{proof}

For example, we have a monoidal functor from $\gr\cD(\Z)$ to $\cD(\Z)$ sending $\Z$ in bidegree $(0,1)$ to $\Z$ in arbitrary degree $n$.
This sends the exterior algebra $\Lambda(z)$ to the exterior algebra in $\cD(\Z)$ on a degree $n$ generator, and our $\bE_1$ presentation to an $\bE_1$-presentation in $\cD(\Z)$.

The proof of Theorem \ref{thm:e1presentation} is based on bar-cobar duality.
Specifically, we say that an algebra $A$ in $\gr\cD(\Z)$ is \emph{formally connected} if the unit map $\Z\to \gr\cD(\Z)$ induces an equivalence of the ``formal degree'' $0$ part,
and the algebra $A$ vanishes in negative formal degrees.
Analogously, we call a coalgebra $C$ formally connected if the counit map $C\to \Z$ induces an equivalence in nonpositive formal degrees.
More generally, we call an algebra $A$ formally $n$-connective if the cofiber $\widetilde{A}$ of $\Z\to A$ is formally $n$-connective as an object of $\gr\cD(\Z)$, and analogously for coalgebras.
With this convention, an algebra or coalgebra is formally connected if it is formally $1$-connective.
We write $\Alg^{cn}(\gr\cD(\Z))$ for the full subcategory of $\Alg(\gr\cD(\Z))$ on all formally connected algebras, and analogously for coalgebras.

\begin{prop}
\label{prop:barcobar}
There is an equivalence
\[
\Alg^{cn}(\gr\cD(\Z))\to \CoAlg^{cn}(\gr\cD(\Z))
\]
given by the bar-cobar adjunction.
\end{prop}

This is an instance of more general results on Koszul duality, and can for example be extracted out of \cite[Proposition 4.1.2]{francis-gaitsgory},
where the requisite pronilpotency condition is guaranteed by our graded connected setup.
Since this in particular involves an identification of the usual 2-sided bar construction with the operadic one for $\bE_1$,
we also give a slightly more elementary argument,
relying on connectivity behaviour of the bar and cobar constructions which we will require anyways.
The overall structure of this argument is similar to the structure of the argument in Section 4 of \emph{loc. cit.}

\begin{lemma}
\label{lem:connectivitypowers}
Let $X\to Y$ be a map of formally $1$-connective objects of $\gr\cD(\Z)$ whose fiber is formally $(k+1)$-connective. Then the fiber of
\[
X^{\otimes n}\to Y^{\otimes n}
\]
is formally $(n+k)$-connective
\end{lemma}
\begin{proof}
The map $X^{i}\otimes Y^{n-i}\to X^{i-1} \otimes Y^{n-i+1}$ is obtained from tensoring the map $X\to Y$ with the formally $(n-1)$-connective object $X^{i-1}\otimes Y^{n-i}$.
Thus its fiber is formally $(n+k)$-connective, i.e. it induces isomorphisms in formal degrees $\leq (n+k)$. Since we can write the map $X^{\otimes n}\to Y^{\otimes n}$ as a composite of $n$ such maps, the claim follows.
\end{proof}

For a coalgebra $C$, $\Cobar(C)$ is the limit of a cosimplicial diagram with $C^{\otimes n}$ in level $n$. The $\Tot^n$-tower for that diagram takes the form
\[
\Cobar(C) \simeq \lim (\ldots \to \Tot^n(\Cobar^{\bullet}(C)) \to \Tot^{n-1}(\Cobar^{\bullet}(C)) \to\ldots)
\]
with fiber of $\Tot^n\to \Tot^{n-1}$ given by $\Omega^n \widetilde{C}^{\otimes n}$. Here $\widetilde{C}$ is the fiber of the counit map $C\to \Z$. If $C$ is connected, note that $\Omega^n \widetilde{C}^{\otimes n}$ is formally $n$-connective (since $\Omega$ does not change formal degree!), so it follows that $\Cobar(C)\to \Tot^n(\Cobar^{\bullet}(C))$ is an equivalence in formal degrees $\leq n$.

The connectivity of the $\Tot^n$-tower has the following two immediate consequences:

\begin{lemma}
\label{lem:cobarsifted}
$\Cobar: \CoAlg^{cn}(\gr\cD(\Z))\to \gr\cD(\Z)$ commutes with sifted colimits.
\end{lemma}
\begin{proof}
$\Tot^n(\Cobar^\bullet)$ commutes with sifted colimits, since it is a finite limit of terms of the form $C^{\otimes k}$. Since the $\Tot^n$-tower stabilizes degreewise, the result follows.
\end{proof}

\begin{lemma}
\label{lem:conservative}
For a map of connected coalgebras $C\to D$, the formal connectivity of $C\to D$ (i.e. the lowest degree in which $C\to D$ is not an equivalence) agrees with the formal connectivity of $\Cobar(C)\to \Cobar(D)$. Dually the formal connectivity of a map of connected algebras agrees with the formal connectivity after $\BBar$. In particular both $\Cobar$ and $\BBar$ are conservative.
\end{lemma}
\begin{proof}
Assume $C\to D$ is (at least) $k$-connective, with $k\geq 1$. Then so is $\widetilde{C}\to\widetilde{D}$, and by Lemma \ref{lem:connectivitypowers}, the map $\widetilde{C}^{\otimes n}\to \widetilde{D}^{\otimes n}$ is $n+k-1$-connective. 
By looking at the total homotopy fiber of the square
\[
\begin{tikzcd}
\Tot^n(\Cobar^{\bullet}(C)) \rar\dar &  \Tot^n(\Cobar^{\bullet}(D))\dar\\
\Tot^{n-1}(\Cobar^{\bullet}(C)) \rar & \Tot^{n-1}(\Cobar^{\bullet}(D)),
\end{tikzcd}
\]
we now see that the fiber of $\Tot^n(\Cobar^{\bullet}(C)) \to \Tot^n(\Cobar^{\bullet}(D))$ agrees with the fiber on $\Tot^1$ in degree $k$. This fiber is simply $\Omega \fib(\widetilde{C}\to \widetilde{D})$. In the limit, we see that 
\[
\Omega \fib(\widetilde{C}\to\widetilde{D})_k \simeq \fib(\Cobar(C)\to \Cobar(D))_k.
\]
This shows inductively that $\Cobar(C)\to \Cobar(D)$ has exactly the same formal connectivity as $C\to D$. The argument for $\BBar$ proceeds analogously, using the skeletal filtration.
\end{proof}

\begin{proof}[Proof of Proposition \ref{prop:barcobar}]
Since formally connected algebras are canonically augmented (the map $A\to \Z$ is determined as the inverse equivalence to the unit in degree $0$, and as $0$ in all other degrees, both of which are contractible choices), and formally connected coalgebras are canonically coaugmented, we can equivalently state that Bar and Cobar constitute inverse equivalences between
\[
\Alg(\gr\cD(\Z)^{cn}_{\Z/\kern-0.2em/\Z})\to \CoAlg(\gr\cD(\Z)^{cn}_{\Z/\kern-0.2em/\Z}),
\]
where the subscript denotes the ``double slice'' category of pointed augmented objects, and the formal connectivity assumption is now stated on the pointing or augmentation morphism. Recall that formally connected objects are actually closed under limits and colimits.

In this setting, \cite[5.2.2.19]{HA} guarantees that we at least have an adjunction, with the left adjoint taking an algebra to a coalgebra with underlying object given by the bar construction, and right adjoint taking a coalgebra to an algebra having underlying object given by the cobar construction.

To check that this adjunction is an equivalence, it suffices to check that the unit $A\to \Cobar(\BBar(A))$ is an equivalence, and that $\Cobar$ is conservative. The latter follows from Lemma \ref{lem:conservative}. For the former, since we can write every algebra $A$ as sifted colimit of free ones, and Lemma \ref{lem:cobarsifted} tells us that $\Cobar$ commutes with sifted colimits (while $\BBar$ is even a left adjoint), it suffices to check that
\[
A\to \Cobar(\BBar(A))
\]
is an equivalence for $A$ free.
If $A$ is free on a single element of bidegree $(k,n)$, then $\BBar(A)$ has homology concentrated in degree $(k+1,n)$ (and $(0,0)$). For degree readons, it is thus a square-zero coalgebra. Since both $\Free$ and $\BBar$ commute with colimits (the latter because it is left adjoint), it follows that $\BBar(\Free(X)) = \Z\oplus \Sigma^{1,0} X$ with square-zero coalgebra structure for \emph{any} $X\in \gr_+\cD(\Z)$. Finally, for such a square-zero coalgebra, we have that $\Cobar(\Z\oplus \Sigma^{1,0} X) = \bigoplus X^{\otimes n} \simeq \Free(X)$.
\end{proof}

In the proof, we have also shown
\begin{lemma}
\label{lem:squarezerofree}
$\BBar$ and $\Cobar$ take the free algebra $\Free_{\bE_1}(\Sigma^{0,n} X)$ to the square-zero coalgebra $\Z\oplus \Sigma^{1,n} X$ and vice-versa.\qed
\end{lemma}

Under this bar-cobar equivalence, the statement of Theorem \ref{thm:e1presentation} amounts to giving a similar decomposition of $\BBar(\Lambda(e))$. Morally, since $\BBar$ takes a free $\bE_1$-algebra to a square zero coalgebra, and pushouts in coalgebras can be computed underlying, where $\BBar$ turns an $\bE_1$-presentation of an algebra into a sort of square-zero cell structure on the resulting coalgebra. In particular, just the underlying homology of $\BBar(A)$ should already tell us how many cells we need.

\begin{lemma}
\label{lem:crucialconnectivity}
Let $C\to D$ be a map of connected coalgebras which is an equivalence in formal degrees $\leq n$. Assume
\[
\begin{tikzcd}
F \rar\dar & C\dar\\
\Z \rar & D
\end{tikzcd}
\]
is a pullback diagram of coalgebras. Then $F$ is a formally $(n+1)$-connective coalgebra, and in formal degree $n+1$ the diagram is a pullback (and pushout) square in $\cD(\Z)$.
\end{lemma}
\begin{proof}
By an induction on $n$, we may automatically assume the first part of the claim, namely that $F$ is zero in formal degrees $\leq n$.
From bar-cobar duality, we see that the diagram
\begin{equation}
\label{diag:cobarpullback}
\begin{tikzcd}
\Cobar(F) \dar\rar & \Cobar(C)\dar\\
\Z \rar & \Cobar(D)
\end{tikzcd}
\end{equation}
is a pullback diagram of algebras, in particular an underlying pullback diagram.

Since the fiber of the map $\widetilde{C}\to \widetilde{D}$ is formally $(n+1)$-connective by assumption, the fiber of the map $\widetilde{C}^{\otimes k}\to \widetilde{D}^{\otimes k}$ is formally $n+k$-connective by Lemma \ref{lem:connectivitypowers}. As $\widetilde{F}^{\otimes k}$ is even formally $(n+1)k= nk+k$-connective, it follows that the total homotopy fiber of the diagram
\[
\begin{tikzcd}
\widetilde{F}^{\otimes k} \rar\dar & \widetilde{C}^{\otimes k}\dar\\
0 \rar & \widetilde{D}^{\otimes k}
\end{tikzcd}
\]
is formally $n+k$-connective. It follows that the total homotopy fiber of the diagram \eqref{diag:cobarpullback} agrees in formal degree $(n+1)$ with the total homotopy fiber of the corresponding diagram for $\Tot^1(\Cobar(-))$. Since the total fiber of the corresponding diagram for $\Tot^0$ is simply $0$, we get that the total fiber of \eqref{diag:cobarpullback} agrees in formal degree $(n+1)$ with $\Sigma^{0,-1}$ of the total fiber of
\[
\begin{tikzcd}
\widetilde{F} \rar\dar & \widetilde{C}\dar\\
0 \rar & \widetilde{D}.
\end{tikzcd}
\]
But since \eqref{diag:cobarpullback} is a pullback diagram in $\gr\cD(\Z)$, this proves the claim.
\end{proof}

\begin{defi}
By a \emph{generalized cell structure} on a connected coalgebra $C$ we mean a sequence of $C^{(n)}\to C$, equivalences in formal degrees $\leq n$, with $C^{(0)}=\Z$, and compatible pushout diagrams
\[
\begin{tikzcd}
\Z\oplus \Sigma^{-1,n} X_{n} \rar\dar & C^{(n-1)}\dar\\
\Z \rar & C^{(n)}.
\end{tikzcd}
\]
We will call $X_n$ the \emph{complex of $n$-cells}.
\end{defi}

Analogously, we define on the algebra side:
\begin{defi}
By a \emph{generalized $\bE_1$-cell structure} on a connected algebra $A$ we mean a sequence of $A^{(n)}\to A$, equivalences in formal degrees $\leq n$, with $A^{(0)}=\Z$, and compatible pushout diagrams
\[
\begin{tikzcd}
\Free_{\bE_1}(\Sigma^{-1,n}X_{n}) \rar\dar & A^{(n-1)}\dar\\
\Z \rar & A^{(n)}.
\end{tikzcd}
\]
We will call $X_n$ the \emph{complex of $n$-cells}.
\end{defi}

\begin{thm}
\label{thm:existence}
Every connected coalgebra $C$ admits a generalized cell structure whose complex of $n$-cells is given by $C_n$. Every connected algebra $A$ admits a generalized $\bE_1$-cell structure whose complex of $n$-cells is given by $\Sigma^{-1} \BBar(A)_n$.
\end{thm}
\begin{proof}
We first discuss the coalgebra case. We let $C^{(n)}$ be the formal truncation which is $0$ in formal degrees $>n$. Just like t-structure truncation, the corresponding functor $\tau^f_{\leq n}: \gr_{\geq 0} \cD(\Z)\to \gr_{\geq 0}\cD(\Z)$ arises from an adjunction between $\gr_{\geq 0} \cD(\Z)$ and $\gr_{[0,n]} \cD(\Z)$, where the functors are the obvious ones (forgetting degrees $>n$, and extending by $0$). Here we write $\gr_{[0,n]}(\cD(\Z))$ for objects concentrated in formal degrees between and including $0$ and $n$. Unlike the case of t-structures, these functors are adjoint in both possible ways! Since the functor $\gr_{\geq 0} \cD(\Z)\to \gr_{[0,n]}\cD(\Z)$ is strict monoidal, we obtain a lax and an oplax structure on the composite $\tau^f_{\leq n}$. This gives us a canonical coalgebra structure on $C^{(n)}$. Informally, it is given by the composite
\[
C^{(n)} \to C \to C\otimes C \to C^{(n)}\otimes C^{(n)}.
\]

If we define the coalgebra $F$ by the pullback diagram
\[
\begin{tikzcd}
F\rar\dar & C^{(n-1)}\dar\\
\Z \rar & C,
\end{tikzcd}
\]
Lemma \ref{lem:crucialconnectivity} implies that $F$ is concentrated in formal degrees $\geq n$, and that in degree $n$ the diagram is a pullback in $\cD(\Z)$, given by
\[
\begin{tikzcd}
F_n\rar\dar & 0\dar\\
0\rar & C_n.
\end{tikzcd}
\]
In particular, $F_n = \Sigma^{-1} C_n$.
If we compose with the canonical map of coalgebras $\tau^f_{\leq n} F \to F$ (again using the discussion of adjoints and truncation at the beginning of this proof), we thus obtain a diagram of coalgebras
\[
\begin{tikzcd}
\tau_{\geq n}^f F\rar\dar & C^{(n-1)}\dar\\
\Z\rar & C,
\end{tikzcd}
\]
which is a pullback (and pushout) of objects of $\cD(\Z)$ in degree $n$. Since pushouts of coalgebras are formed underlying, the pushout vanishes in degrees $>n$, and maps equivalently to $C$ in degrees $\leq n$, i.e. agrees with $C^{(n)}$. For degree reasons, $\tau_{\leq n}^f F$ is $\Z \oplus \Sigma^{-1,n} C_n$, and so the claim about coalgebras follows.

For the algebra claim, we simply apply the coalgebra statement to $\BBar(A)$ and apply Lemma \ref{lem:squarezerofree} to the resulting cell structure.
\end{proof}

\begin{lemma}
\label{lem:barexterior}
$\BBar(\Lambda(e))$ can be described by the graded dg coalgebra $\Z\{ x_1, x_2,\ldots\}$, with zero differential, $x_i$ in bidegree $(i,i)$ and comultiplication given by
\[
\Delta(x_n) = \sum_{i+j=n} x_i\otimes x_j.
\]
\end{lemma}
\begin{proof}
The bigraded homology of $\BBar(\Lambda(e))$ is given by the bigraded Tor groups $\Tor^{\Lambda(e)}(\Z,\Z)$. The standard free resolution over an exterior algebra allows us to compute these to be given by $\Z$ in degree $(n,n)$ for each $n$. The coalgebra structure can also be computed either directly from the resolution, or by dualizing and thinking about the algebra structure given by composition on $\Ext^{\Lambda(e)}(\Z,\Z)$ (which is polynomial). Since the dual algebra has polynomial homology, it is equivalent to the algebra given by the dga $\Z[x]$ with zero differential, from which we get our description by dualizing back.
\end{proof}

\begin{proof}[Proof of Theorem \ref{thm:e1presentation}]
We simply apply Theorem \ref{thm:existence} to $\Lambda(e)$, using Lemma \ref{lem:barexterior} to see that the complex of $n$-cells in the resulting $\bE_1$-cell structure is $\Z[n-1]$.
\end{proof}

In the resulting cell structure on the coalgebra side, the $n$-skeleton $C^{(n)}$ is explicitly represented by the sub-coalgebra of $\Z\{x_1,\ldots,x_n\}\subseteq \Z\{x_1,\ldots\}$.

We can also study the $\bE_1$-algebras $\Lambda^{(n)}\simeq \Cobar(C^{(n)})$ arising as skeleta on the algebra side more closely:
\begin{lemma}
\label{lem:pnhomology}
For $n=1$, the homology $H_*(\Lambda^{(1)})$ is of the form $\Z[e]$, with $e$ in bidegree $(0,1)$. For $n>1$, the homology $H_*(\Lambda^{(n)})$ is of the form $\Lambda(e)\otimes \Z[r_{n+1}]$, with $r_{n+1}$ of bidegree $(n-1,n+1)$.
\end{lemma}
\begin{proof}
Since $\Lambda^{(n)} = \Cobar(C^{(n)})$, we can compute its homology groups as derived cotensor product $\Cotor^{C^{(n)}}_*(\Z,\Z)$, or equivalently (since the cotensor product with $\Z$ and $\Hom_{C^{(n)}}(\Z,-)$ agree, they both simply take a comodule to its primitives), $\Ext_*^{C^{(n)}}(\Z,\Z)$. Since everything is finite type and thus dualizable, this agrees with $\Ext(\Z,\Z)$ over the graded algebra dual to the coalgebra $C^{(n)}$. Since $C^{(n)}$ is represented by the graded coalgebra $\Z\{ x_1, x_2,\ldots, x_n\}$, its dual agrees with the truncated polynomial algebra $\Z[x]/x^{n+1}$, with $x$ now in bidegree $(-1,-1)$. This truncated polynomial algebra admits a $2$-periodic minimal resolution, and from the Yoneda description of the product structure on $\Ext_*(\Z,\Z)$ we recover the claimed multiplicative structure.
\end{proof}

Since $\Lambda^{(n)}$ arises from $\Lambda^{(n-1)}$ by a pushout
\[
\begin{tikzcd}
\Free_{\bE_1}(\Z[n,n-2])\dar\rar & \Lambda^{(n-1)}\dar\\
\Z \rar & \Lambda^{(n)},
\end{tikzcd}
\]
and the degree $n$-part of $\Lambda^{(n)}$ vanishes by Lemma \ref{lem:pnhomology}, the upper horizontal attaching map needs to be an equivalence in formal degree $n$. Thus, it is given by a choice of generator for the polynomial part of $H_*(\Lambda^{(n-1)})$ as in \ref{lem:pnhomology}.

\begin{rem}
\label{rem:explicitdga}
From the explicit description of $\BBar(\Lambda(e))$ in Lemma \ref{lem:barexterior} one can actually explicitly describe the skeleta of the $\bE_1$-cell structure on $\Lambda(e)$: The coalgebra skeleta of $\BBar(\Lambda(e))$ are represented by the sub-coalgebras $\Z\{x_1, x_2,\ldots, x_n\}$, i.e. the skeleta of the $\bE_1$-cell structure of $\Lambda(e)$ are obtained by applying the cobar construction to these coalgebras.
Given an explicit dg coalgebra which is flat over $\Z$, the cobar construction can be described explicitly as reduced total complex of the cosimplicial complex
\[
\begin{tikzcd}
\Z \ar[r,shift left=0.5]\ar[r,shift right] & C \ar[r,shift left]\ar[r]\ar[r,shift right] & C\otimes C \ar[r,shift left=1.5]\ar[r,shift left=0.5]\ar[r,shift right=0.5]\ar[r,shift right=1.5]&\ldots
\end{tikzcd}
\]
One can see explicitly describe this reduced total complex (as a dga) as the free graded associative algebra on $\widetilde{C}[-1]$, with differential given by the sum of the differential on $\widetilde{C}[-1]$ and a component coming from the (reduced) comultiplication $\widetilde{C}\to \widetilde{C}\otimes \widetilde{C}$, with signs coming from the shift.
This way, one can describe $\Lambda^{(n)}$ explicitly as
\[
\Lambda^{(n)}\simeq \Z\langle e_1,\ldots,e_n \rangle,
\]
with differential $\partial e_k = \sum_{i+j=k} (-1)^{i-1} e_i e_j$.  Here $\Z\langle-\rangle$ denotes a free associative algebra, and $e_k$ is in bidegree $(n-1,n)$.

In this description, one can take the generator $r_n$ in \ref{lem:pnhomology}, and thus also the attaching map for $\Lambda^{(n-1)}\to \Lambda^{(n)}$, to be
\[
r_n = \sum_{i+j=n} (-1)^{i-1} e_i e_j.
\]
For example, we know abstractly from \ref{lem:pnhomology} that $r_n$ can be taken as any generator of the kernel of $H_*(\Lambda^{(n-1)})\to H_*(\Lambda^{(n)})$ in formal degree $n$. But the explicit dgas
\[
\Z\langle e_1,\ldots,e_{n-1} \rangle \to \Z\langle e_1,\ldots,e_n \rangle
\]
differ in formal degree $n$ only by the single generator $e_n$, so the kernel on homology is generated by $\partial e_n$, which is the claimed expression.
\end{rem}

Using the explicit description from Remark \ref{rem:explicitdga}, we can also fix the sign in the choice of generator in Lemma \ref{lem:pnhomology}, and equivalently in the choice of attaching maps in the presentation of $\Lambda(e)$.

The description of $\Lambda(e)$ as sequence of cell-attachments along these elements $r_n$ gives us an obstruction theory for $\bE_1$-maps out of $\Lambda(e)$. Indeed, an $\bE_1$-map $\Lambda^{(1)} \to R$ corresponds precisely to a choice of element of $H_{(0,1)}(R)$. Having fixed an $\bE_1$-map $\Lambda^{(n-1)}\to R$, we obtain a well-defined element of $H_{(n-2,n)}(R)$ as the image of $r_n$, and we get a further extension to $\Lambda^{(n)}$ if and only of this element vanishes (with choices of extensions being in bijection to ``nullhomotopies'' of a representing cycle, i.e. a torsor over $H_{(n-1,n)}(R)$).

\begin{defi}
Given a map $\Lambda^{(n-1)}\to R$ sending $e\in H_{0,1}(\Lambda^{(n-1)})$ to $a\in R$, we call the image of $r_n$ in $H_{(n-2,n)}$ an \emph{$n$-fold Massey power} of $a$, and write $\langle a\rangle^n$.
\end{defi}

Looking at the explicit dga description in Remark \ref{rem:explicitdga}, we see that $\langle a\rangle^n$ indeed is a representative of the $n$-fold Massey product $\langle a, \ldots, a\rangle$. However, it is somewhat more tightly defined: An arbitrary Massey product depends of choices of nullhomotopies for every adjacent product, nullhomotopies for the resulting $3$-fold Massey products, and so on. In the Massey power above, we choose the same homotopy ($e_2$) for all the adjacent products $a\cdot a$, a single homotopy for all the resulting $3$-fold Massey products, and so on.

The obstruction theory discussed above can thus be summarized as follows: Having fixed nullhomotopies for all Massey powers $\langle a\rangle^i$ for $2\leq i<n$, we get a well-defined $n$-fold Massey power $\langle a\rangle^n$. A map $\Lambda(e)\to R$ taking $e$ to an element $a$ in bidegree $(0,1)$ corresponds precisely to a coherent choice of nullhomotopies for all Massey powers of $a$.

\bibliography{standard}

\providecommand{\bysame}{\leavevmode\hbox to3em{\hrulefill}\thinspace}
\providecommand{\MR}{\relax\ifhmode\unskip\space\fi MR }
\providecommand{\MRhref}[2]{%
  \href{http://www.ams.org/mathscinet-getitem?mr=#1}{#2}
}
\providecommand{\href}[2]{#2}
\begin{thebibliography}{GGN15}

\bibitem[AN20]{AnNik20}
B.~Antieau and T.~Nikolaus, \emph{Cartier modules and cyclotomic spectra},
  Journal of the American Mathematical Society (2020).

\bibitem[Ant19]{An2019periodic}
B.~Antieau, \emph{Periodic cyclic homology and derived de {R}ham cohomology},
  Annals of K-Theory \textbf{4} (2019), no.~3, 505--519.

\bibitem[BBD83]{BBD83}
A.~A. Beilinson, J.~Bernstein, and P.~Deligne, \emph{Faisceaux pervers},
  Ast{\'e}risque \textbf{100} (1983).

\bibitem[Bei87]{Bei87}
A.~A. Beilinson, \emph{On the derived category of perverse sheaves}, K-Theory,
  arithmetic and geometry (Moscow 1984), Lecture notes in Math. \textbf{1289}
  (1987), 27--41.

\bibitem[BG16]{BGRecoll}
C.~Barwick and S.~Glasman, \emph{A note on stable recollements},
  arXiv:1607.02064 (2016).

\bibitem[BGT13]{bgt}
A.~J. Blumberg, D.~Gepner, and G.~Tabuada, \emph{A universal characterization
  of higher algebraic {K}-theory}, Geometry \& Topology \textbf{17} (2013),
  no.~2, 733--838.

\bibitem[BMS19]{BMS2}
B.~Bhatt, M.~Morrow, and P.~Scholze, \emph{Topological {H}ochschild homology
  and integral $p$-adic {H}odge theory}, Publications math{\'e}matiques de
  l'IH{\'E}S \textbf{129} (2019), no.~1, 199--310.

\bibitem[CS19]{cesnavicius2019purity}
K.~Cesnavicius and P.~Scholze, \emph{Purity for flat cohomology}, arXiv
  preprint arXiv:1912.10932 (2019).

\bibitem[Del71]{DelHodge2}
P.~Deligne, \emph{Th{\'e}orie de {H}odge, {II}}, Publications Math{\'e}matiques
  de l'Institut des Hautes {\'E}tudes Scientifiques \textbf{40} (1971), no.~1,
  5--57.

\bibitem[DJW19]{DJW}
T.~Dyckerhoff, G.~Jasso, and T.~Walde, \emph{Simplicial structures in higher
  {Auslander}–{Reiten} theory}, Advances in Mathematics \textbf{355} (2019).

\bibitem[FG12]{francis-gaitsgory}
J.~Francis and D.~Gaitsgory, \emph{Chiral {Koszul} duality}, Selecta
  Mathematica \textbf{18} (2012), no.~1, 27--87.

\bibitem[FL15]{FLRecoll}
D.~Fiorenza and F.~Loregian, \emph{Recollements in stable
  {$\infty$}-categories}, arXiv:1507.03913 (2015).

\bibitem[GGN15]{ggn15}
D.~Gepner, M.~Groth, and T.~Nikolaus, \emph{Universality of multiplicative
  infinite loop space machines}, Algebraic \& Geometric Topology \textbf{15}
  (2015), 3107--3153.

\bibitem[GHN17]{GHN17}
D.~Gepner, R.~Haugseng, and T.~Nikolaus, \emph{Lax colimits and free fibrations
  in {$\infty$}-categories}, Documenta Mathematica (2017), no.~22, 1225--1266.

\bibitem[Gla13]{GlasmanDay}
S.~Glasman, \emph{Day convolution for {$\infty$}-categories}, arXiv:1308.4940
  (2013).

\bibitem[Gla16]{glasman2016spectrum}
\bysame, \emph{A spectrum-level {H}odge filtration on topological {H}ochschild
  homology}, Selecta Mathematica \textbf{22} (2016), no.~3, 1583--1612.

\bibitem[GP18]{Gwilliam-Pavlov}
O.~Gwilliam and D.~Pavlov, \emph{Enhancing the filtered derived category},
  Journal of Pure and Applied Algebra \textbf{222} (2018), no.~11, 3621--3674.

\bibitem[Hed21]{hedenlundPhD}
A.~P. Hedenlund, \emph{Multiplicative {Tate} spectral sequences}, Ph.D. thesis,
  University of Oslo, 2021.

\bibitem[Joy08a]{JoyalNotes}
A.~Joyal, \emph{Notes on quasi-categories}, preprint (2008).

\bibitem[Joy08b]{JoyThQcat}
\bysame, \emph{The theory of quasi-categories and its applications}, Notes from
  a course at CRM, Barcelona (2008).

\bibitem[Lor21]{coend}
F.~Loregian, \emph{Coend calculus}, London Mathematical Society Lecture Note
  Series, vol. 468, Cambridge University Press, 2021.

\bibitem[Lur09]{HTT}
J.~Lurie, \emph{Higher topos theory}, vol. Annals of Mathematics Studies,
  Princeton University Press, Princeton, NJ, 2009.

\bibitem[Lur15]{lurie2015rotation}
\bysame, \emph{Rotation invariance in algebraic {K}-theory}, preprint (2015).

\bibitem[Lur17]{HA}
\bysame, \emph{Higher algebra},
  \url{https://www.math.ias.edu/~lurie/papers/HA.pdf} (2017).

\bibitem[Lur19]{kerodon}
\bysame, \emph{\textit{Kerodon}}, \url{https://https://kerodon.net}, 2019.

\bibitem[Nik16]{NikStableOperads}
T.~Nikolaus, \emph{Stable $\infty$-operads and the multiplicative {Y}oneda
  lemma}, arXiv:1608.02901 (2016).

\bibitem[Rak20]{raksit}
A.~Raksit, \emph{Hochschild homology and the derived de {Rham} complex
  revisited}, arXiv:2007.02576 (2020).

\bibitem[RV16]{RV2}
E.~Riehl and D.~Verity, \emph{Homotopy coherent adjunctions and the formal
  theory of monads}, Advances in Mathematics \textbf{286} (2016), 802--888.

\bibitem[Sag08]{SagaveToda}
S.~Sagave, \emph{Universal {Toda} brackets of ring spectra}, Transactions of
  the American Mathematical Society \textbf{360} (2008), no.~5, 2767--2808.

\bibitem[Shi07]{shipleyHZ}
B.~Shipley, \emph{{HZ}-algebra spectra are differential graded algebras},
  American journal of mathematics \textbf{129} (2007), no.~2, 351--379.

\bibitem[Wal19]{waldeDK}
T.~Walde, \emph{Homotopy coherent theorems of {D}old-{K}an type},
  arXiv:1912.06368 (2019).

\end{thebibliography}
\bibliographystyle{amsalpha}
\end{document}